\documentclass[aop]{imsart}
\RequirePackage[T1]{fontenc}
\RequirePackage[utf8]{inputenc}
\RequirePackage{amsthm,amsmath,amsfonts,amssymb}
\RequirePackage[numbers]{natbib}
\RequirePackage{newtxtext}
\RequirePackage[varvw]{newtxmath}
\RequirePackage{mathtools}
\RequirePackage{float}
\RequirePackage{graphicx}
\RequirePackage{enumitem}
\RequirePackage{tikz-cd}
\RequirePackage[colorlinks,citecolor=blue,urlcolor=blue]{hyperref}

\startlocaldefs

\theoremstyle{plain}
\newtheorem{theorem}{Theorem}
\newtheorem{corollary}[theorem]{Corollary}
\newtheorem{lemma}[theorem]{Lemma}
\newtheorem{proposition}[theorem]{Proposition}

\theoremstyle{remark}
\newtheorem{definition}[theorem]{Definition}
\newtheorem{assumption}[theorem]{Assumption}

\newtheorem{example}[theorem]{Example}
\newtheorem{remark}[theorem]{Remark}

\numberwithin{equation}{section}
\numberwithin{theorem}{section}

\newcommand{\R}{\mathbb{R}}
\newcommand{\Rd}{\mathbb{R}^{d}}
\newcommand{\Cbq}{C_{\textnormal{b},\textnormal{q}}(\mathbb{R}^{d})}
\newcommand{\Cq}{C_{\textnormal{q}}(\mathbb{R}^{d})}
\newcommand{\Cqaff}{C_{\textnormal{q}}^{\textnormal{aff}}(\mathbb{R}^{d})}

\newcommand{\conv}{\operatorname{conv}}
\newcommand{\supp}{\operatorname{supp}}

\newcommand{\aff}{\textnormal{aff}}
\newcommand{\W}{\mathcal{W}}

\newcommand{\E}{\mathbb{E}}
\newcommand{\MT}{\mathsf{MT}}
\newcommand{\Cpl}{\mathsf{Cpl}}

\newcommand{\D}{\mathcal{D}}

\newcommand{\bary}{\operatorname{bary}}
\newcommand{\lc}{\preceq_{\textnormal{c}}}
\newcommand{\MCov}{\textnormal{MCov}}

\newcommand{\pisbm}{\pi^{\textnormal{SBM}}}
\newcommand{\pidmt}{\pi^{\textnormal{DMT}}}

\newcommand{\Law}{\operatorname{Law}}

\newcommand{\dom}{\operatorname{dom}}
\newcommand{\PP}{\mathcal{P}}
\newcommand{\psiopt}{\psi_{\textnormal{opt}}}

\newcommand{\psixsbm}{\psi_{\textnormal{SBM}}^{x}}
\newcommand{\psilim}{\psi_{\textnormal{lim}}}
\newcommand{\tpsilim}{\tilde{\psi}_{\textnormal{lim}}}

\newcommand{\Clim}{C_{\textnormal{lim}}}

\newcommand{\Ilim}{I_{\textnormal{lim}}}

\newcommand{\Nlim}{\mathscr{N}_{\textnormal{lim}}}

\newcommand{\vxsbm}{v_{\textnormal{SBM}}^{x}}
\newcommand{\OB}{Ob{\l}{\'o}j}

\endlocaldefs

\begin{document}

\begin{frontmatter}

\title{Existence of Bass martingales and the martingale Benamou--Brenier problem in \texorpdfstring{$\mathbb{R}^{d}$}{Rd}}

\runtitle{Bass martingales and martingale Benamou--Brenier}

\begin{aug}

\author[A]{\fnms{Julio}~\snm{Backhoff-Veraguas}\ead[label=e1]{julio.backhoff@univie.ac.at}},
\author[A]{\fnms{Mathias}~\snm{Beiglb\"ock}\ead[label=e2]{mathias.beiglboeck@univie.ac.at}},
\author[A]{\fnms{Walter}~\snm{Schachermayer}\ead[label=e3]{walter.schachermayer@univie.ac.at}},
\and
\author[A]{\fnms{Bertram}~\snm{Tschiderer}\ead[label=e4]{bertram.tschiderer@univie.ac.at}}

\address[A]{Faculty of Mathematics, University of Vienna\printead[presep={,\ }]{e1,e2,e3,e4}}

\end{aug}

\begin{abstract}
In classical optimal transport, the contributions of Benamou--Brenier and McCann regarding the time-dependent version of the problem are cornerstones of the field and form the basis for a variety of applications in other mathematical areas.

In this article, we characterize solutions to the martingale Benamou--Brenier problem as \emph{Bass martingales}, i.e.\ transformations of Brownian motion through the gradient of a convex function. Our result is based on a new (static) Brenier-type theorem for a particular weak martingale optimal transport problem. As in the classical case, the structure of the primal optimizer is derived from its dual counterpart, whose derivation forms the technical core of this article. A key challenge is that dual attainment is a subtle issue in martingale optimal transport, where dual optimizers may fail to exist, even in highly regular settings.
\end{abstract}


\end{frontmatter}

\section{Introduction}

Optimal transport as a field in mathematics goes back to Monge \cite{Mo81} and Kantorovich \cite{Ka42}, who established its modern formulation. The seminal results of Benamou, Brenier, and McCann \cite{Br87, Br91, BeBr99, Mc94} form the basis of the modern theory, with striking applications in a variety of different areas; see the monographs \cite{Vi03, Vi09, AmGi13, Sa15}. 

We are interested in transport problems where the transport plan satisfies an additional martingale constraint. This additional requirement arises naturally in finance, but is of independent mathematical interest. For example, there are notable consequences for the study of martingale inequalities (e.g.\ \cite{BoNu13,HeObSpTo12,ObSp14}) and the Skorokhod embedding problem (e.g.\ \cite{BeCoHu14, KaTaTo15, BeNuSt19}). Early articles on this topic of \textit{martingale optimal transport} include \cite{HoNe12, BeHePe12, TaTo13, GaHeTo13, DoSo12, CaLaMa14}. 

In view of the central role taken by the results of Benamou--Brenier \cite{BeBr99} and McCann \cite{Mc01} on optimal transport for squared (Euclidean) distance, the related continuous-time transport problem and McCann's displacement interpolation, it is paramount to search for similar concepts in the martingale context. This is the main motivation of the present article. Before describing our results, we briefly recap the classical role models. 

\subsection{Benamou--Brenier transport problem and McCann interpolation in probabilistic terms} 

Given $\mu, \nu$ in the space $\PP_{2}(\Rd)$ of $d$-dimensional distributions with finite second moment, Brenier's theorem asserts that if $\mu$ is absolutely continuous, the following are equivalent for a coupling $\pi \in \Cpl(\mu,\nu)$, i.e., a probability on $\Rd \times \Rd$ with marginals $\mu$, $\nu$:
\begin{enumerate}[label=(\arabic*)] 
\item $\pi$ minimizes the transport costs with respect to the squared distance between $\mu$ and $\nu$. 
\item $\pi$ is concentrated on the graph of a function $f \colon \Rd \rightarrow \Rd$ which is \textit{monotone} in the sense that $f = \nabla v$ for a convex function $v \colon \Rd \rightarrow (-\infty,+\infty]$.  
\end{enumerate} 
Brenier's theorem gives a structural description of the optimal transport plan for the most widely used cost function. Another essential consequence of the result is that it provides a particularly natural way to move probabilities: it implies that $\mu$ can be transported to $\nu$ via the gradient of a convex function. In fact, in many applications it is the mere existence of this monotone transport map that is required, irrespective of its optimality properties. 

Ideally, a martingale counterpart of Brenier's theorem should mimic these aspects. That is, starting from a natural optimization problem, our goal is to define a natural martingale that connects the probabilities $\mu$ and $\nu$. 

While \cite{BeJu16, HeTo13, GhKiLi19} have proposed martingale versions of Brenier's monotone transport map based on static transport problems, our starting point is a continuous formulation in the spirit of the Benamou--Brenier continuous-time transport problem given in \cite{BaBeHuKa20, HuTr17}. 

We first recapitulate the continuous-time formulation of Brenier's theorem in probabilistic language. For $\mu, \nu \in \PP_{2}(\Rd)$, consider
\begin{equation*} \label{MBBB} \tag{BB}
T_{2}(\mu, \nu) \coloneqq 
\inf_{\substack{X_{0} \sim \mu, \, X_{1} \sim \nu, \\ X_{t} = X_{0} + \int_{0}^{t} b_{s} \, ds}}
\mathbb{E}\Big[\int_{0}^{1} \vert b_{t} \vert^{2} \, dt\Big].
\end{equation*}
Under the above assumptions, \eqref{MBBB} has a unique optimizer (in law), and the following are equivalent for a process $X = (X_{t})_{0 \leqslant t \leqslant 1}$ with $X_{0} \sim \mu$, $X_{1} \sim \nu$: 
\begin{enumerate}[label=(\arabic*)] 
\item $X$ solves \eqref{MBBB}.
\item $X_{1} = f(X_{0})$, where $f$ is the gradient of a convex function $v \colon \Rd \rightarrow (-\infty,+\infty]$ and all particles move with constant speed, i.e.\ $X_{t}= X_{0} + b t$, $t \in [0,1]$, for the random variable $b = X_{1} -X_{0}$. 
\end{enumerate}
In this case, setting $\mu_t \coloneqq \Law(X_t)$, $t\in [0,1]$, defines McCann's displacement interpolation between $\mu$ and $\nu$. This interpolation is time-consistent in the sense that interpolation between $\mu_s$ and $\mu_t$, for $0\leqslant s< t\leqslant 1$, recovers the family $(\mu_r)_{s \leqslant r \leqslant t}$ up to the obvious affine-time transformation.

\subsection{Martingale optimization problem}

Assume that $\mu, \nu \in \PP_{2}(\Rd)$ are in convex order, in signs $\mu \lc \nu$, that is, $\int \phi \, d\mu \leqslant \int \phi \, d\nu$ for all convex functions $\phi \colon \Rd \rightarrow \R$ with linear growth. We consider the optimization problem 
\begin{equation*} \label{MBMBB} \tag{MBB}
MT(\mu, \nu) \coloneqq 
\inf_{\substack{M_{0} \sim \mu, \, M_{1} \sim \nu, \\ M_{t} = M_{0} + \int_{0}^{t}  \sigma_{s} \, dB_{s}}} 
\mathbb{E}\Big[\int_{0}^{1} \vert \sigma_{t} - I_{d} \vert^{2}_{\textnormal{HS}} \, dt\Big],
\end{equation*} 
where $\vert \cdot \vert_{\textnormal{HS}}$ denotes the Hilbert--Schmidt norm. Here, the infimum is taken over all filtered probability spaces $(\Omega,\mathcal{F},\mathbb{P})$, with $\sigma$ an $\mathbb{R}^{d \times d}$-valued $\mathcal{F}$-progressive process and $B$ a $d$-dimensional $\mathcal{F}$-Brownian motion such that $M$ is a martingale. The stochastic integral $\int_{0}^{t}  \sigma_{s} \, dB_{s}$ is in the It\^{o} sense.

The main result of \cite{BaBeHuKa20} is that \eqref{MBMBB} admits a unique optimizer $M^{\ast}$. Moreover, $M^{\ast}$ is a continuous strong Markov martingale. While \eqref{MBBB} implies that particles move along straight lines, the functional in \eqref{MBMBB} stipulates that $M^{\ast}$ maximizes the correlation with Brownian motion subject to the given marginal constraints. It is also shown in \cite{BaBeHuKa20} that $M^{\ast}$ is the process whose evolution follows the movement of a Brownian particle as closely as possible with respect to an \textit{adapted Wasserstein distance} (see, e.g., \cite{BaBaBeEd19a, Fo22a}) subject to the marginal conditions $M^{\ast}_{0} \sim \mu$ and $M^{\ast}_{1} \sim \nu$. As in the classical case, setting $\mu_t \coloneqq (\Law (M^{\ast}_t))_{0 \leqslant t \leqslant 1}$ defines a time-consistent interpolation between $\mu$ and $\nu$.
These properties motivate to interpret \eqref{MBMBB} as a Martingale Benamou--Brenier problem and to call the martingale $M^{\ast}$ \textit{stretched Brownian motion} between $\mu$ and $\nu$ as in \cite{BaBeHuKa20}.

\subsection{The Bass martingale}

Our main result relates the optimality property in the definition of stretched Brownian motion to a structural description. As motivation we recall a classical construction of Bass \cite{Ba83} which provides a particularly appealing martingale $M = (M_{t})_{0 \leqslant t \leqslant 1}$ that terminates in a fixed measure $\nu$ on the real line and starts at the barycenter of $\nu$. This amounts to a solution of the Skorokhod embedding problem (modulo time change). Let $(B_{t})_{0 \leqslant t \leqslant 1}$ be Brownian motion (started in $B_{0}=0$), let $\gamma \coloneqq \Law(B_{1})$ denote the standard Gaussian and define $f \colon \R \rightarrow \R$ as the $\gamma$-a.e.\ unique increasing mapping that pushes $\gamma$ to $\nu$. Bass then defines the martingale 
\begin{equation} \label{Brownian} 
M_{t} \coloneqq \E[f(B_{1}) \, \vert \, \sigma(B_{s} \colon s \leqslant t)]
= \E[f(B_{1}) \, \vert \, B_{t}], \qquad 0 \leqslant t \leqslant 1,
\end{equation}
so that $M_{0}$ starts at the barycenter of $\nu$ and $\Law(M_{1}) = \nu$. 

As we are interested in martingales in multiple dimensions with possibly non-degenerate starting law $\mu$, we consider the construction of Bass in the following generality: 

\begin{definition} \label{def:BassMarti_intro} Let $B = (B_{t})_{0 \leqslant t \leqslant 1}$ be  Brownian motion on $\Rd$, where $B_{0} \sim \alpha \in \PP(\Rd)$, and let $v \colon \Rd \rightarrow \R$ be convex such that $\nabla v(B_{1})$ is square-integrable. Then we call 
\[
M_{t} \coloneqq 
\E[\nabla v(B_{1}) \, \vert \, \sigma(B_{s} \colon s \leqslant t)]
= \E[\nabla v(B_{1}) \, \vert \, B_{t}], \qquad 0 \leqslant t \leqslant 1
\]
a \textit{Bass martingale} from $\mu \coloneqq \Law(M_{0})$ to $\nu \coloneqq \Law(M_{1})$.
\end{definition}

The question arises under which conditions on $\mu, \nu \in \PP_{2}(\Rd)$ there is a Bass martingale $M$ from $\mu$ to $\nu$, i.e.\ satisfying $M_{0} \sim \mu$ and $M_{1} \sim \nu$, for $\mu$ and $\nu$ prescribed in advance?

\subsection{Structure of stretched Brownian motion}

We need a connectivity assumption on the marginals, known under the name of \textit{irreducibility} in the martingale transport literature.

\begin{definition} \label{defi:irreducible_intro} 
For $\mu, \nu \in \PP(\R^d)$ we say that the pair $(\mu,\nu)$ is \textit{irreducible} if for all measurable sets $A, B \subseteq \Rd$ with $\mu(A), \nu(B)>0$ there is a martingale $X= (X_{t})_{0 \leqslant t \leqslant 1}$ with $X_{0} \sim \mu$, $X_{1} \sim \nu$ such that $\mathbb{P}(X_{0}\in A, X_{1}\in B) >0$. 
\end{definition} 

With this definition in hand, we can announce our first main result.

\begin{theorem} \label{MainTheorem} Let $\mu \lc \nu$ be probabilities on $\Rd$ with finite second moments and suppose that $(\mu,\nu)$ is irreducible. Then the following are equivalent for a martingale $M = (M_{t})_{0 \leqslant t \leqslant 1}$ with $M_{0} \sim \mu$ and $M_{1} \sim \nu$:
\begin{enumerate}[label=(\arabic*)] 
\item \label{MainTheorem_1} $M$ is stretched Brownian motion, i.e.\ the optimizer of \eqref{MBMBB}.
\item \label{MainTheorem_2} $M$ is a Bass martingale.  
\end{enumerate}
\end{theorem}

We briefly provide some context for the irreducibility assumption in Theorem \ref{MainTheorem}. In classical optimal transport, the product coupling of the marginals guarantees that mass can be transported from an arbitrary starting position to an arbitrary target position. However, this is not necessarily true in the case of martingale transport, where it can happen that $\Rd$ decomposes into disjoint (convex) regions that do not communicate with each other. The irreducibility assumption excludes this, as it guarantees that for any sets $A$ and $B$ that are charged by $\mu$ and $\nu$, respectively, there exists a martingale $X$ connecting them (in the sense that $\mathbb{P}(X_{0} \in A, X_{1} \in B)>0$). For $d>1$, the appearance of more than one irreducible component leads to intricate phenomena,  analyzed in the remarkable contributions \cite{GhKiLi19, DeTo17, ObSi17}. This problem is revisited in the follow-up article \cite{ScTs24} in terms of Bass martingales. For equivalent characterizations of the irreducibility assumption we refer to Theorem \ref{theo_coic} in Appendix \ref{app_sec_irr}. In particular, it is equivalent to consider continuous- or discrete-time martingales in Definition \ref{defi:irreducible_intro}. 

We emphasize that irreducibility is not only a sufficient assumption for Theorem \ref{MainTheorem}, but it is in fact necessary. Indeed, the Bass martingale connects any two sets which are charged by $\mu$ and $\nu$, see Remark \ref{rem_irr_nec} in Appendix \ref{app_sec_irr}. In particular, if there is a Bass martingale from $\mu$ to $\nu$, then $(\mu, \nu)$ is irreducible.

\smallskip  

An important consequence of Theorem \ref{MainTheorem} is that for any irreducible pair $\mu \lc \nu$ there exists a unique Bass martingale 
\begin{equation} \label{BassOnceMore}
M_{t} = \E[\nabla v(B_{1}) \, \vert \, \sigma(B_{s} \colon s \leqslant t)]
= \E[\nabla v(B_{1}) \, \vert \, B_{t}], \qquad 0 \leqslant t \leqslant 1,
\end{equation}
with $M_{0} \sim \mu$, $M_{1} \sim \nu$ and it is worthwhile to comment on the properties of $M$. We write $\gamma^{t}$ for the $d$-dimensional centered Gaussian distribution with covariance matrix $tI_{d}$ and $v_{t} \coloneqq v \ast \gamma^{1-t} \colon \Rd \rightarrow \R$ for the convolution of the function $v$ and the measure $\gamma^{1-t}$. In these terms, \eqref{BassOnceMore} amounts to
\begin{equation} \label{BassOnceMore2}
M_{t} = \nabla v_{t} (B_{t}), \qquad 0 \leqslant t \leqslant 1,
\end{equation}
which emphasizes that the Bass martingale is obtained as a monotone transformation of Brownian motion at each time point. Finally, as a Brownian martingale, $M$ is a diffusion which connects $\mu$ and $\nu$. Indeed, applying It\^{o}'s formula to \eqref{BassOnceMore2} we obtain 
\[
dM_{t} = \nabla^{2} v_{t} \circ \nabla v_{t}^{\ast}(M_{t})\, dB_{t}, \qquad 0 \leqslant t \leqslant 1,
\]
where $v_{t}^{\ast}$ denotes the convex conjugate of $v_{t}$.

\smallskip 

McCann \cite{Mc95} managed to extend the validity of Brenier's theorem: using Aleksandrov's lemma he showed that the Brenier map is well defined without any moment assumptions. We leave the question of whether a similar extension is possible in the case of Bass martingales for future research. 

\smallskip

Finally, we note that the implication ``\ref{MainTheorem_2} $\Rightarrow$ \ref{MainTheorem_1}''  is comparably easy and follows from \cite[Theorem 1.10]{BaBeHuKa20}. We will establish the reverse implication that every optimizer of \eqref{MBMBB} is a Bass martingale, based on a duality result, which we describe in the next section.

\subsection{Weak transport formulation and dual viewpoint} \label{wtfadv}
Our arguments rely on a novel dual viewpoint on stretched Brownian motion, which is in turn based on a reformulation of problem \eqref{MBMBB} as a weak transport problem. 

In classical transport, minimization of the squared distance is equivalent to maximization of correlation and the latter formulation yields a dual problem which is simpler to interpret. A similar fact holds true in the present martingale setting: problem \eqref{MBMBB} is equivalent to maximizing the covariance with Brownian motion
\begin{equation} \label{MBMBB2}
P(\mu, \nu) \coloneqq 
\sup_{\substack{M_0 \sim \mu, \, M_1 \sim \nu, \\ M_t = M_0 + \int_0^t \sigma_s \, dB_s}}
\mathbb{E}\Big[\int_0^1 \textnormal{tr} (\sigma_t) \, dt\Big],
\end{equation} 
in the sense that both problems have the same optimizer and the values are related via $MT(\mu,\nu) = d + \int \vert y \vert^2\, d\nu(y)  - \int \vert x \vert^2 \, d\mu(x) -2 P(\mu,\nu)$. 

We go on to reformulate \eqref{MBMBB2} as a weak (martingale)  transport problem in the sense of \cite{GoRoSaTe14}. Indeed 
\begin{equation} \label{eq_primal} 
P(\mu,\nu) 
= \sup_{\pi \in \MT(\mu,\nu)}  \int \MCov(\pi_{x},\gamma) \, \mu(dx),
\end{equation}
where the \textit{maximal covariance} ($\MCov$) between  $p_{1},p_{2} \in \PP_{2}(\Rd)$ is defined as
\[
\MCov(p_{1},p_{2}) \coloneqq \sup_{q \in \Cpl(p_{1},p_{2})} \int \langle x_{1},x_{2} \rangle \, q(dx_{1},dx_{2}).
\]
The optimization problem \eqref{eq_primal} is a \textit{weak} martingale optimal transport problem  since the function $\pi \mapsto \int \MCov(\pi_{x},\gamma) \, \mu(dx)$ is non-linear, as opposed to classical optimal transport, where linear problems of the form $\pi \mapsto \int c(x,y) \, \pi(dx,dy)$ are studied.

The equality \eqref{eq_primal} was established in \cite[Theorem 2.2]{BaBeHuKa20}. Furthermore, there exists a unique optimizer $\pisbm \in \MT(\mu,\nu)$ of \eqref{eq_primal} and if $(M_{t})_{0 \leqslant t \leqslant 1}$ is the stretched Brownian motion from $\mu$ to $\nu$, then the law of $(M_{0},M_{1})$ equals $\pisbm$. Indeed,  the discrete-time formulation \eqref{eq_primal} will play a central role in this paper. 
In particular, it will be the basis for the following duality result, which implies the difficult implication ``\ref{MainTheorem_1} $\Rightarrow$ \ref{MainTheorem_2}'' in Theorem \ref{MainTheorem}.

\begin{theorem} \label{theorem_new_duality}
Assume that $\mu, \nu \in \PP_{2}(\Rd)$ are in convex order. The value $P(\mu,\nu)$ of the continuous-time optimization problem \eqref{MBMBB2} is equal to 
\begin{equation} \label{WeakDual}
D(\mu,\nu) \coloneqq \inf_{\substack{\psi \in L^{1}(\nu), \\ \textnormal{$\psi$ convex}}} \Big( \int \psi \, d\nu - \int (\psi^{\ast} \ast \gamma)^{\ast} \, d \mu \Big).
\end{equation}
The infimum is attained by a lower semicontinuous convex function $\psiopt \colon \Rd \rightarrow (-\infty,+\infty]$ satisfying $\mu(\operatorname{ri}(\dom\psiopt)) = 1$ if and only if $(\mu,\nu)$ is irreducible. In this case the (unique) optimizer to \eqref{MBMBB} is given by the Bass martingale
\[
M_{t} \coloneqq \E[\nabla v(B_{1}) \, \vert \, \sigma(B_{s} \colon s \leqslant t)]
= \E[\nabla v(B_{1}) \, \vert \, B_{t}], \qquad 0 \leqslant t \leqslant 1,
\]
where $v = \psiopt^{\ast}$ and $B_{0} \sim \nabla (\psiopt^{\ast} \ast \gamma)^{\ast}(\mu)$.
\end{theorem} 

Here, $\operatorname{ri}(\dom\psiopt)$ denotes the relative interior of the domain of $\psiopt$, i.e.\ the set on which $\psiopt$ is finite. The symbol $\ast$ used as a superscript denotes the convex conjugate of a function, otherwise it is the standard convolution operator.

\smallskip

It turns out that the optimizer $\psiopt$ is not necessarily $\nu$-integrable. Therefore, in order for the difference of the integrals in \eqref{WeakDual} to be well defined for $\psi = \psiopt$, attainment of $D(\mu,\nu)$ has to be understood in a ``relaxed'' sense frequently encountered in martingale transport problems; see \cite{BeJu16,BeNuSt19,BeNuTo16} and Propositions \ref{prop_new_duality_first_part}, \ref{prop_new_duality_rel} below. 

\smallskip

Note that, since the primal optimizer $\pisbm$ equals the law of $(M_0,M_1)$, Theorem \ref{theorem_new_duality} describes the optimizer of the static weak martingale transport problem \eqref{eq_primal} in terms of the gradient of a convex function, which originates from its dual problem \eqref{WeakDual}. More precisely, we have
\[
\pisbm = \Law\big((\nabla v \ast \gamma)(B_{0}),\nabla v(B_{1})\big).
\]
This is analogous to the classical Brenier theorem and hence one might view Theorem \ref{theorem_new_duality} as a Brenier-type theorem for weak martingale transport. 

We emphasize that while Brenier's theorem is a direct consequence of attainment for the dual transport problem (which is now well understood), the situation is more delicate in case of Theorem \ref{theorem_new_duality}. The main reason is that dual attainment for martingale transport can fail even in very regular settings, e.g.\ for Lipschitz costs and compactly supported  measures on the real line, see \cite[Section 4.3]{BeHePe12}. Positive results are only available for $d=1$ and under strong assumptions \cite{BeJu16, BeLiOb19, BeNuTo16}. In addition, the duality theory for weak optimal transport is known to be significantly more complicated than its classical counterpart. Specifically first results for dual attainment have appeared only recently \cite{BePaRiSc25} and are not applicable to weak martingale transport problems. 
Accordingly, the technical core of our work is to establish dual attainment in the framework of Theorem \ref{theorem_new_duality}.

\subsection{Bass martingales for specific marginals}

Theorem \ref{MainTheorem} is only an existence result, which does not provide an explicit Bass martingale connecting the given marginals. This is parallel to the situation in optimal transport, where the question to numerically approximate optimal transport plans has generated enormous interest. Already in dimension one, the problem to explicitly determine Bass martingales is non-trivial (as well as highly important for financial applications). Here \cite{CoHe21} gives an efficient iteration algorithm. See \cite{AcMaPa23} for a rigorous proof of the algorithm's linear convergence and \cite{JoLoOb23} for a multi-dimensional variant of the algorithm. 

\smallskip

On the other hand (as in the case of Brenier's theorem), one can easily generate examples of Bass martingales by specifying a convex function $v$ and the starting distribution of $B_0$. We give here one such example that also serves to illustrate some of the intricacies related to irreducible decompositions. Fix parameters $\rho, k>0$ and consider the radially symmetric function $v(x,y)= h(\vert(x,y)\vert)$, where $h$ is continuous and piecewise affine with slope $1/2$ on $[0,k]$, slope $8/5$ on $[k, \infty)$, and $h(0)=0$. We consider the Bass martingale $M_t=\E[\nabla v(B_1) \, \vert \, \mathcal F_t]$, where $B_0$ is uniformly distributed on the centered circle with radius $\rho$. Then $\nabla v(x)=\frac{x}{2\vert x \vert}$, $\vert x \vert\in (0,k)$ and $\nabla v(x)=\frac{8x}{5\vert x\vert}$, $\vert x\vert\in (k,\infty)$, from which it follows easily that $M$ terminates on either the circle with radius $1/2$ or the circle with radius $8/5$ and that $M$ starts uniformly distributed on some centered circle. Specifying $\rho = 3$, $k= 3.17$, we find (numerically) that $M_0$ is uniformly distributed on the unit circle, while $M_1$ charges the circles with radius $1/2$ and $8/5$ equally, see Figures \ref{figure1} and \ref{figure2}.

\begin{figure}[htbp]
\centering
\begin{minipage}{0.49\textwidth}
\includegraphics[width=\textwidth]{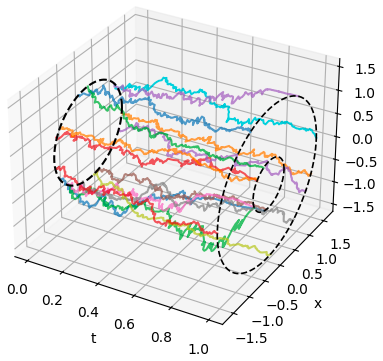} 
\caption{Sample paths of $M$, different starting values.}
\label{figure1}
\end{minipage}
\hfill
\begin{minipage}{0.49\textwidth}
\includegraphics[width=\textwidth]{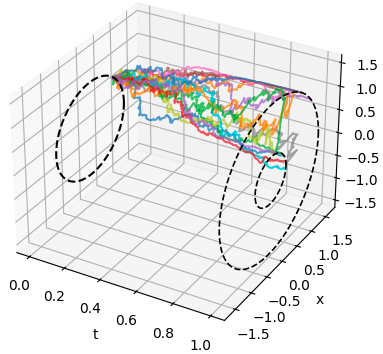} 
\caption{Sample paths of $M$, identical starting value.}
\label{figure2}
\end{minipage}
\end{figure}

Setting $\mu= \Law M_0$, $\nu=\Law M_1$, we have, of course, $\mathbb P((M_0, M_1)\in A\times B) >0 $ whenever $\mu(A), \nu(B)>0$.

\smallskip

Notably this is not the case for every martingale $M$ starting in $\mu$ and terminating in $\nu$. Specifically we may consider martingales that move only along the (violet)  segments indicated in Figure \ref{figure3}, we refer to \cite[Example 6.2]{ScTs24} for a detailed justification of this claim.

\begin{figure}[htbp]
\centering
\includegraphics[width=0.49\textwidth]{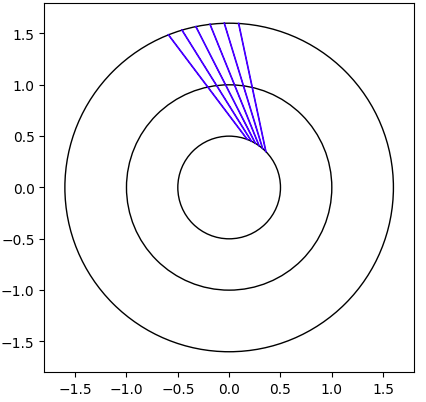} 
\caption{There is a martingale which starts uniformly distributed on the middle circle ($r=1$), terminates with equal probability on the inner circle ($r=1/2$) or outer circle ($r=8/5$), and moves only on the violet line segments. The angle between the segments and the middle circle is chosen so that the lengths to the inner and outer circle are equal.}
\label{figure3}
\end{figure}

This example illustrates that even if $(\mu, \nu)$ is irreducible, there can be martingales which are entirely confined to smaller subpartitions. This phenomenon cannot appear in the trivial case $d=1$, see \cite[Appendix A]{BeJu16}, a fact that is heavily exploited in dual attainment results for martingale transport in $d=1$. On the other hand, we believe that the existence of martingales  confined to smaller subpartitions present a significant obstacle to general dual attainment results for martingale transport (and weak martingale transport) for $d\geqslant 2$.

\subsection{Literature} \label{sec_literature}

The first article that provides structural results for the martingale transport problem in general dimensions is the work  \cite{GhKiLi19} of Ghoussoub--Kim--Lim. They obtain descriptions for the minimizers and maximizers for the cost function $c(x,y)=\vert x-y \vert$, when marginals are supported on $\R^2$, as well as for marginals on higher-dimensional state spaces that are in subharmonic order. Given a specific martingale $X$, Ghoussoub--Kim--Lim also define a finest paving of the source space into cells that are invariant under the martingale $X$. 

The contributions of De March--Touzi \cite{DeTo17} and \OB--Siorpaes \cite{ObSi17} put the theme of irreducible decompositions center stage. In contrast to the work of Ghoussoub--Kim--Lim, their interest lies in pavings that are invariant under \emph{all} martingales which start in $\mu$ and terminate in $\nu$. Specifically it is shown in \cite{DeTo17} that there exists a unique finest paving with this property. This De March--Touzi paving is revisited in the follow-up paper \cite{ScTs24}, where it is characterized in terms of Bass martingales. While the Ghoussoub--Kim--Lim paving and the De March--Touzi paving agree for $d=1$, they can be different for $d\geqslant 2$, as discussed in the previous section. The interested reader is referred to Ciosmak's works \cite{Ci23a,Ci23b}, where a more general notion of irreducibility (going beyond martingale transports) is studied.

Huesmann--Trevisan \cite{HuTr17} investigate Benamou--Brenier-type formulations for the martingale transport problem on $\R^d$. In particular, they provide equivalent PDE-formulations and establish existence and duality results. In the context of market impact in finance Loeper \cite{Lo18} arrives independently to problem \eqref{MBMBB}.

For further contributions to the martingale transport in continuous time, we mention \cite{DoSo12, BeHeTo15, CoObTo19, GhKiPa19,GuLoWa19, GhKiLi20, ChKiPrSo21, GuLo21} among many others. 

Finally, \cite{BaBeHuKa20} is a predecessor of the present article in which it is established that the martingale transport problem \eqref{MBMBB} admits a unique solution, solves further related optimization problems and has properties of time consistency in the spirit of classical optimal transport for the squared distance cost function. Furthermore, the counterpart of our main result is established in dimension one (without reference to duality). As discussed in length above, the main difficulty in higher dimensions stems from the subtleties surrounding the concept of irreducible component.




\subsection{Structure of the paper} 

In Section \ref{sec_def_a_not} we introduce some definitions and frequently used notation. Duality results for stretched Brownian motion and the important role of convexity are discussed in Section \ref{sec_dual_sbm}. The proof of the first part of Theorem \ref{theorem_new_duality}, namely that there is no duality gap between the primal problem \eqref{MBMBB2} and the dual problem \eqref{WeakDual}, is given in Section \ref{sec_pr_int_b_i}. In Section \ref{sec_prep} we outline and prepare the proof of the second part of Theorem \ref{theorem_new_duality}, which gives a necessary and sufficient condition for dual attainment in terms of the irreducibility assumption. To this end, we analyze the connection between the existence of dual optimizers and Bass martingales, which is the content of Section \ref{sec_dual_bass}. In Section \ref{sec_irr_bass} we show that the irreducibility assumption implies the existence of a dual optimizer. After these preparations we are in a position to prove Theorem \ref{MainTheorem} in Subsection \ref{subsec_pr_int_a} and complete the proof of Theorem \ref{theorem_new_duality} in Subsection \ref{subsec_pr_int_b_ii}.

\smallskip

In Appendix \ref{app_proof_no_duality_gap} we prove Theorem \ref{theorem_no_duality_gap}, which shows that there is no duality gap between the auxiliary optimization problems \eqref{eq_primal} and  \eqref{eq_dual}. The rather technical proofs of Lemmas \ref{lem_phi_psi_gen}, \ref{lem_co_fin}, \ref{lem:dim_dom_psi}, \ref{lem_aux_grad_convolution} and \ref{lem_int_dom_range} are collected in Appendix \ref{app_tech_lemm}. In Appendix \ref{app_prop_gen_case} we provide the proof of Proposition \ref{prop_1b}, a result which will be of crucial importance in the follow-up paper \cite{ScTs24} on the non-irreducible case, but which also seems of independent interest. Finally, in Appendix \ref{app_sec_irr} we give equivalent characterizations of irreducibility, as introduced in Definition \ref{defi:irreducible_intro}.

\tableofcontents

\newpage

In response to the thoughtful suggestions of a diligent referee, we provide in Figure \ref{graphic} an overview of the logical structure of our paper. This is intended to help readers navigate the dependencies between the results of this work, as well as the related literature, more effectively.

\begin{figure}[htbp]
\centering
\includegraphics[width=\textwidth]{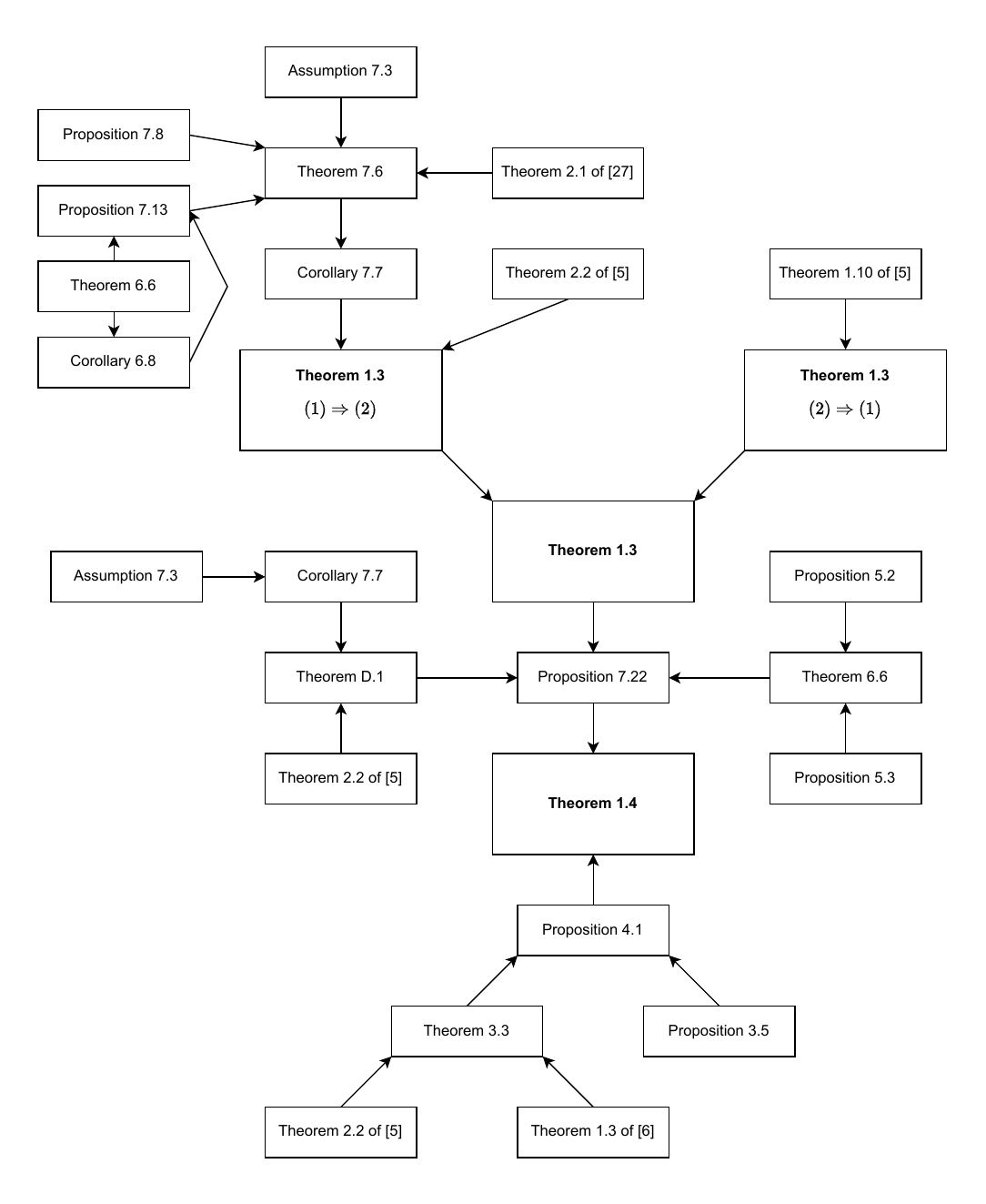}
\caption{Dependencies between results.}
\label{graphic}
\end{figure}

\newpage

\section{Definitions and notation} \label{sec_def_a_not}

\begin{itemize}
\item We write $\PP(\Rd)$ for the probability measures on $\Rd$, $\PP_{p}(\Rd)$ for the subset of probability measures satisfying $\int \vert x \vert^{p} \, d\mu < +\infty$, and $\PP_{p}^{x}(\Rd)$ for the elements of $\PP_{p}(\Rd)$ with barycenter $x \in \Rd$, for $p \in [1,+\infty)$.
\item For $\mu,\nu \in \PP(\Rd)$, we denote by $\Cpl(\mu,\nu)$ the set of all couplings $\pi \in \PP(\Rd \times \Rd)$ between $\mu$ and $\nu$, i.e., probability measures $\pi$ on $\Rd \times \Rd$ with first marginal $\mu$ and second marginal $\nu$.
\item We say that $\mu \in \PP_{1}(\Rd)$ is dominated by $\nu \in \PP_{1}(\Rd)$ in convex order and write $\mu \lc \nu$, if for all convex functions $\phi \colon \Rd \rightarrow \R$ with linear growth we have $\int \phi \, d\mu \leqslant \int \phi \, d\nu$. 
\item For $\mu,\nu \in \PP_{1}(\Rd)$ with $\mu \lc \nu$ we define the collection of martingale transports $\MT(\mu,\nu)$ as those couplings $\pi \in \Cpl(\mu,\nu)$ satisfying $\bary(\pi_{x}) \coloneqq \int y \, \pi_{x}(dy) = x$, for $\mu$-a.e.\ $x \in \Rd$. Here, the family of probability measures $\{\pi_{x}\}_{x\in\Rd} \subseteq \PP(\Rd)$ is obtained by disintegrating the coupling $\pi$ with respect to its first marginal $\mu$, i.e., $\pi(dx,dy) = \pi_{x}(dy) \, \mu(dx)$.
\item The $d$-dimensional Gaussian distribution with barycenter $x \in \Rd$ and covariance matrix $tI_{d}$ is denoted by $\gamma_{x}^{t}$ and we set $\gamma^{t} \coloneqq \gamma_{0}^{t}$, $\gamma_{x} \coloneqq \gamma_{x}^{1}$ as well as $\gamma \coloneqq \gamma_{0}$. 
\item We denote by $\Cq$ the set of continuous functions $\psi \colon \Rd \rightarrow \R$ with quadratic growth, meaning that there are constants $a,k,\ell \in \R$ with
\[
\ell + \tfrac{\vert \, \cdot \,  \vert^{2}}{2} \leqslant \psi(\, \cdot \,) \leqslant a + k \vert \cdot  \vert^{2}.
\]
We also introduce the set
\begin{equation} \label{def_cqaff}
\Cqaff \coloneqq \big\{ \psi(\, \cdot \,) + \aff(\, \cdot \,) \colon \ \psi \in \Cq, \, \aff \colon \Rd \rightarrow \R \textnormal{ affine} \big\}.
\end{equation}
\item The pushforward of a measure $\mu$ under a function $f$ is denoted by $f(\mu)$.
\item For two measures $\varrho$ and $\rho$ we write $\varrho \ast \rho$ for their convolution. If $f$ is a function, the convolution of $f$ and $\rho$ is defined as
\[
(f \ast \rho)(x) \coloneqq \int f(x-y) \, \rho(dy).
\]
In particular, $(f \ast \gamma)(x) = \int f(x+y) \, \gamma(dy)$.
\item For a function $f \colon \Rd \rightarrow (-\infty,+\infty]$, its convex conjugate is given by 
\[
f^{\ast}(y) \coloneqq \sup_{x \in \Rd} \big(\langle x,y \rangle - f(x)\big)
\]
and we write
\[
\dom f \coloneqq \{ x \in \Rd \colon f(x) < + \infty \}
\]
for the domain of $f$. If $f$ is convex and $\dom f \neq \varnothing$, we say that $f$ is a proper convex function. If $f$ is additionally lower semicontinuous, then $f = f^{\ast\ast}$ by the Fenchel--Moreau theorem, a fact that we will use repeatedly in this work.
\item The operators $\operatorname{int}$ and $\operatorname{ri}$ denote the interior and relative interior of a set, respectively. We write $\overline{A}$ for the closure of a set $A \subseteq \Rd$. 
\item The operator $\conv$ applied to a function or a set denotes the convex hull. The convex hull $\conv\psi$ of a function $\psi$ is the greatest convex function smaller or equal to $\psi$.
\item The support and the closed convex hull of the support of a measure $\rho$ are denoted by $\supp(\rho)$ and $\widehat{\supp}(\rho)$, respectively.
\item The symbol $\partial$, applied to a convex function, denotes its subdifferential or --- by abuse of notation --- also a subgradient.
\end{itemize}

\section{Duality for stretched Brownian motion} \label{sec_dual_sbm}

We fix $\mu, \nu \in \PP_{2}(\Rd)$ with $\mu \lc \nu$.
Recall from \eqref{eq_primal} that our main focus will lie on the (primal) \textit{weak martingale transport problem} 
\begin{align*}
P(\mu,\nu) 
\coloneqq  \sup_{\pi \in \MT(\mu,\nu)}  \int \MCov(\pi_{x},\gamma) \, \mu(dx), 
\end{align*}
where the maximal covariance ($\MCov$) between $p_{1},p_{2} \in \PP_{2}(\Rd)$ is given by
\[
\MCov(p_{1},p_{2}) \coloneqq \sup_{q \in \Cpl(p_{1},p_{2})} \int \langle x_{1},x_{2} \rangle \, q(dx_{1},dx_{2}).
\]

\smallskip

Throughout this paper we shall make the following assumption.

\begin{assumption} \label{ass_ful_dim}
The support of $\nu$ affinely spans $\mathbb{R}^{d}$.
\end{assumption}

This assumption will be convenient, mainly for notational reasons. But it does not restrict the generality of the results of this paper. Indeed, let $A$ denote the affine space spanned by $\supp(\nu)$. For $\pi(dx,dy) = \pi_{x}(dy) \, \mu(dx) \in \MT(\mu,\nu)$ we have, for $\mu$-a.e.\ $x \in \Rd$, that $\supp(\pi_{x}) \subseteq \supp(\nu) \subseteq A$. Choose an arbitrary point $x_{0} \in A$ and denote by $\gamma_{x_{0}}^{A}$ the standard Gaussian distribution on the affine space $A$ centered at $x_{0}$. It is easy to see that there is a constant $c = c(x_{0},A)$ such that, for $\pi \in \MT(\mu,\nu)$, we have
\[
\int \MCov(\pi_{x},\gamma) \, \mu(dx) 
= \int \MCov(\pi_{x},\gamma_{x_{0}}^{A}) \, \mu(dx) + c.
\]
In other words, replacing $\gamma$ by $\gamma_{x_{0}}^{A}$ in \eqref{eq_primal} only changes the primal problem in a trivial way by adding a constant. Denoting by $m$ the affine dimension of $A$, we can choose an isometry from $A$ to $\mathbb{R}^{m}$ which maps $x_{0}$ to $0$ and which maps $\mu, \nu$ to measures $\overline{\mu}, \overline{\nu}$ on $\mathbb{R}^{m}$. In this way we have transformed the primal problem \eqref{eq_primal} for general $\mu, \nu \in \PP_{2}(\Rd)$ with $\mu \lc \nu$ to the problem \eqref{eq_primal} for $\overline{\mu}, \overline{\nu} \in \PP_{2}(\mathbb{R}^{m})$ with $\overline{\mu} \lc \overline{\nu}$ such that $\supp(\overline{\nu})$ affinely spans $\mathbb{R}^{m}$. In conclusion, we shall assume without loss of generality Assumption \ref{ass_ful_dim} throughout this paper.

\smallskip

We define the dual problem 
\begin{equation} \label{eq_dual} 
\tilde{D}(\mu,\nu) \coloneqq 
\inf_{\psi \in \Cq} \Big( \int \psi \, d\nu - \int \varphi^{\psi} \, d \mu \Big),
\end{equation}
where the function $\Rd \ni x \mapsto \varphi^{\psi}(x)$ is given by
\begin{equation} \label{eq_phi_psi_x}
\varphi^{\psi}(x) \coloneqq \inf_{p \in \PP_{2}^{x}(\Rd)} \Big( \int \psi \, dp - \MCov(p,\gamma)\Big).
\end{equation}

\begin{lemma} \label{prop_vppx} Let $\psi \colon \Rd \rightarrow (-\infty,+\infty]$ be a proper, lower semicontinuous convex function. The function $\varphi^{\psi}$ is convex on $\dom\psi$ and $\varphi^{\psi} \leqslant \psi$.
\begin{proof} The convexity of $x \mapsto \varphi^{\psi}(x)$ is repeated and proved in Lemma \ref{lem_co_fin} below. By taking $p = \delta_{x}$ in \eqref{eq_phi_psi_x}, we immediately see that $\varphi^{\psi}(x) \leqslant \psi(x)$. 
\end{proof}
\end{lemma}

\begin{theorem} \label{theorem_no_duality_gap} Let $\mu, \nu \in \PP_{2}(\Rd)$ with $\mu \lc \nu$. There is no duality gap between the primal problem \eqref{eq_primal} and the dual problem \eqref{eq_dual}, i.e., $P(\mu,\nu) = \tilde{D}(\mu,\nu)$. Moreover, the primal problem is uniquely attained and has a finite value, i.e., there exists a unique $\pisbm \in \MT(\mu,\nu)$ such that
\begin{equation} \label{eq_the_no_duality_gap}
P(\mu,\nu) = \int \MCov(\pisbm_{x},\gamma) \, \mu(dx) < +\infty.
\end{equation}
\end{theorem}

Unique attainment of the primal problem is a consequence of the results in \cite{BaBeHuKa20}, while duality can be derived from the general duality results for weak optimal transport provided in \cite{BaBePa18}, see Appendix \ref{app_proof_no_duality_gap} for details. To motivate \eqref{eq_dual} as a plausible dual formulation of \eqref{eq_primal} we provide some heuristics: 

Indeed, considering 
\[
\chi(\pi) \coloneqq 
\inf_{\psi \in \Cq}  \Big( \int \psi \, d \nu - 
\int \psi \, d\pi_{x} \, d\mu(x)  \Big)
= 
\begin{cases} 
0, & \mbox{if }  \int \pi_{x} \, d\mu(x) = \nu, \\
-\infty, & \mbox{else},
\end{cases} 
\] 
we \textit{formally} obtain the desired duality relation by interchanging $\inf$ and $\sup$: 
\begin{align*}
P(\mu,\nu) 
&=
\sup_{\substack{\pi(dx,dy) = \pi_{x}(dy) \, \mu(dx),  \\ \pi_{x} \in \PP_{2}^{x}(\Rd)}} \, 
\Big(\int \MCov(\pi_{x},\gamma) \, d \mu(x) + \chi(\pi)\Big) \\
&= \inf_{\psi \in \Cq} \bigg( \int \sup_{\pi_{x} \in \PP_{2}^{x}(\Rd)} \Big( \MCov(\pi_{x},\gamma) + \int \psi \, d(\nu -\pi_{x}) \Big) \, d \mu(x) \bigg) = \tilde{D}(\mu,\nu). 
\end{align*}

\medskip

In the remainder of this section we make the important observation that in the dual problem \eqref{eq_dual} it suffices to optimize over the class of functions $\psi \in \Cq$ which are \textit{convex}. We refer to \cite[Theorem 1.1]{GoJu18} for an analogous result. Another similar restriction to convex functions in the setting of weak transport costs has been shown in \cite[Theorem 2.11, (3)]{GoRoSaTe14}.

We then show in Lemma \ref{lem_ext_def_dual} below that it is also equivalent to optimize over all convex functions $\psi \colon \Rd \rightarrow (-\infty,+\infty]$ which are only $\mu$-a.s.\ finite, but not necessarily of quadratic growth.

\smallskip

Given some $\psi \in \Cq$, it will be convenient to have an explicit representation for the convex hull of $\psi$, as in \eqref{lem_conv_eq_1} below. This identity is usually stated in the more specific form 
\[
(\conv\psi)(y) = \inf \bigg\{ \sum_{i=1}^{d+1} \lambda_{i} \psi(y_{i}) \colon \sum_{i=1}^{d+1} \lambda_{i} y_{i} = y \bigg\},
\]
where the infimum is taken over all expressions of $y$ as a convex combination of $d+1$ points, see \cite[Corollary 17.1.5]{Ro70}. 

\begin{lemma} \label{lem_conv} Let $\psi \in \Cq$. Then the convex hull $\conv\psi$ satisfies
\begin{equation} \label{lem_conv_eq_1}
(\conv\psi)(y) = \inf_{p \in \PP_{2}^{y}(\Rd)} \int \psi \, dp, \qquad y \in \Rd
\end{equation}
and again $\conv\psi \in \Cq$.
\end{lemma}

Recalling the dual problem \eqref{eq_dual}, we define the dual function 
\begin{equation} \label{eq_def_dual_func}
\D(\psi) \coloneqq \int \psi \, d\nu - \int \varphi^{\psi} \, d \mu,
\end{equation}
for $\psi \in \Cq$. Now we can prove our crucial observation, that it suffices to optimize the dual function over the class of functions $\psi \in \Cq$ which are \textit{convex}. 

\begin{proposition} \label{prop_crucial} Let $\mu, \nu \in \PP_{2}(\Rd)$ with $\mu \lc \nu$. Then $\D(\conv \psi) \leqslant \D(\psi)$ for all $\psi \in \Cq$ and consequently
\begin{equation} \label{eq_prop_crucial_m}
\tilde{D}(\mu,\nu) = \inf_{\psi \in \Cq} \D(\psi) = \inf_{\substack{\psi \in \Cq, \\ \textnormal{$\psi$ convex}}}\D(\psi).
\end{equation}
\begin{proof} Let $\varepsilon > 0$, $\psi \in \Cq$ and $\{p_{x}\}_{x \in \Rd} \subseteq \PP_{2}(\Rd)$ be a measurable collection of probability measures with $\bary(p_{x}) = x$. To show the claim, it is sufficient to construct a measurable family $\{\bar{p}_{x}\}_{x \in \Rd} \subseteq \PP_{2}(\Rd)$ with $\bary(\bar{p}_{x}) = x$ such that
\begin{equation} \label{prop_crucial_01}
\MCov(p_{x},\gamma) + \int \conv \psi \, d(\nu-p_{x}) 
\leqslant \MCov(\bar{p}_{x},\gamma) + \int \psi \, d(\nu-\bar{p}_{x}) + \varepsilon.
\end{equation}
Let us construct appropriate probability measures $\{\bar{p}_{x}\}_{x \in \Rd}$. By Lemma \ref{lem_conv} and a measurable selection argument we can choose a measurable collection of probability measures $\{\tilde{p}_{y}\}_{y \in \Rd} \subseteq \PP_{2}(\Rd)$ with $\bary(\tilde{p}_{y}) = y$ such that 
\begin{equation} \label{prop_crucial_02}
\int \psi \, d\tilde{p}_{y} \leqslant  (\conv\psi)(y) + \varepsilon.
\end{equation}
Then we define $\bar{p}_{x}(dz) \coloneqq \int_{y} \tilde{p}_{y}(dz) \, p_{x}(dy)$, so that $\bary(\bar{p}_{x}) = x$. Integrating \eqref{prop_crucial_02} with respect to $p_{x}(dy)$ yields
\begin{equation} \label{prop_crucial_03}
\int \psi \, d\bar{p}_{x}  
\leqslant  \int \conv\psi \, dp_{x} + \varepsilon.
\end{equation}
Since $\psi, \conv \psi \in \Cq$ and $p_{x} \in \PP_{2}^{x}(\Rd)$ we conclude
\[
\ell + \tfrac{1}{2} \int \vert y  \vert^{2} \, \bar{p}_{x}(dy)
\leqslant a + k \int \vert y  \vert^{2} \, p_{x}(dy) +  \varepsilon < + \infty,
\]
so that $\bar{p}_{x} \in \PP_{2}^{x}(\Rd)$.

In order to show the inequality \eqref{prop_crucial_01}, we first observe that $p_{x} \lc \bar{p}_{x}$ by Jensen's inequality. Together with the Kantorovich duality (see, e.g.\ \cite[Theorem 5.10]{Vi09}), we conclude that\footnote{In fact, $p_{x} \lc \bar{p}_{x}$ is equivalent to the inequality $\MCov(p_{x},q) \leqslant \MCov(\bar{p}_{x},q)$ being valid for all probability measures $q \in \PP_{2}(\Rd)$, see.\ e.g.\ \cite[Theorem 1]{AcPa22} or \cite[Corollary 1.2]{WiZh22}.}
\begin{equation} \label{prop_crucial_04}
\begin{aligned}
\MCov(p_{x},\gamma) 
&= \inf_{\substack{f \colon \Rd \rightarrow \R \\ \textnormal{convex}}} 
\Big( \int f \, d p_{x} + \int f^{\ast} \, d\gamma \Big) \\
&\leqslant \inf_{\substack{f \colon \Rd \rightarrow \R \\ \textnormal{convex}}} 
\Big( \int f \, d \bar{p}_{x} + \int f^{\ast} \, d\gamma \Big)
= \MCov(\bar{p}_{x},\gamma).
\end{aligned}
\end{equation}
On the other hand, from $\conv \psi \leqslant \psi$ and \eqref{prop_crucial_03} we have the inequality 
\begin{equation} \label{prop_crucial_05}
\int \conv \psi \, d(\nu-p_{x})  
\leqslant \int \psi \, d(\nu-\bar{p}_{x}) + \varepsilon.
\end{equation}
Finally, summing \eqref{prop_crucial_04} and \eqref{prop_crucial_05}, we obtain the inequality \eqref{prop_crucial_01}.
\end{proof}
\end{proposition}

\begin{remark} \label{rem:smooth} On the right-hand side of \eqref{eq_prop_crucial_m} we can further require $\psi$ to be smooth. Indeed, if $\psi \in \Cq$ is convex, then $\psi^{\varepsilon} \coloneqq \psi \ast \gamma^{\varepsilon}$ is a smooth convex function in $\Cq$ and 
\[
\liminf_{\varepsilon \rightarrow 0}
\Big( \int \psi^{\varepsilon} \, d\nu - \int \varphi^{\psi^{\varepsilon}} \, d\mu \Big)
\leqslant \int \psi \, d\nu - \int \varphi^{\psi} \, d\mu, 
\]
as follows by dominated convergence and the inequality $\psi^{\varepsilon} \geqslant \psi$.
\end{remark}

For $\psi \in \Cq$, we recall the definition \eqref{eq_def_dual_func} of the dual function $\mathcal{D}(\, \cdot \,)$, which we rewrite as 
\begin{equation} \label{eq_rep_dual_func}
\D(\psi) = \int  \Big( \int \psi(y) \, \pisbm_{x}(dy)  - \varphi^{\psi}(x) \Big) \, \mu(dx),
\end{equation}
where $\pisbm(dx,dy) = \pisbm_{x}(dy) \, \mu(dx)$ is the unique optimizer of the primal problem \eqref{eq_primal}. The ``relaxed'' representation \eqref{eq_rep_dual_func} of \eqref{eq_def_dual_func} allows us to extend the definition of the dual function $\mathcal{D}(\, \cdot \,)$ to convex functions $\psi \colon \Rd \rightarrow (-\infty,+\infty]$ which are not confined to be in $\Cq$, but which are only required to be $\mu$-a.s.\ finite, i.e., satisfy $\mu(\dom \psi) = 1$. This is summarized in the following lemma.

\begin{lemma} \label{lem_ext_def_dual} Let $\mu, \nu \in \PP_{2}(\Rd)$ with $\mu \lc \nu$ and let $\pisbm \in \MT(\mu,\nu)$ be the unique optimizer of the primal problem \eqref{eq_primal}. Let $\psi \colon \Rd \rightarrow (-\infty,+\infty]$ be a convex function which is $\mu$-a.s.\ finite. Formula \eqref{eq_rep_dual_func} then defines $\D(\psi) \in [0,+\infty]$ and, recalling \eqref{eq_dual}, we have the inequality
\begin{equation} \label{lem_ext_dual_func_01}
\D(\psi) \geqslant \tilde{D}(\mu,\nu).
\end{equation}
In particular, recalling \eqref{eq_prop_crucial_m}, we have
\begin{equation} \label{lem_ext_dual_func_02}
\tilde{D}(\mu,\nu) = \inf_{\substack{\psi \in \Cq, \\ \textnormal{$\psi$ convex}}}\D(\psi) = \inf_{\substack{\mu(\dom \psi) = 1, \\ \textnormal{$\psi$ convex}}}\D(\psi).
\end{equation}
\begin{proof} Let $\psi \colon \Rd \rightarrow (-\infty,+\infty]$ be a convex function which is $\mu$-a.s.\ finite. First, note that by Jensen's inequality we have
\[
\int \psi(y) \, \pisbm_{x}(dy) \geqslant \psi\Big(\int y \, \pisbm_{x}(dy)\Big) = \psi(x).
\]
Together with the inequality $\varphi^{\psi}(x) \leqslant \psi(x)$ (recall Lemma \ref{prop_vppx}), we conclude that
\begin{equation} \label{eq_ext_dual_func_03}
\int \psi(y) \, \pisbm_{x}(dy)  - \varphi^{\psi}(x) \geqslant 0,
\end{equation}
for $\mu$-a.e.\ $x \in \Rd$, so that $\D(\psi) \in [0,+\infty]$. 
In order to prove the inequality \eqref{lem_ext_dual_func_01}, we take $p=\pisbm_x$ in \eqref{eq_phi_psi_x}. This gives  
\[
\varphi^{\psi}(x) \leqslant  \int \psi \, d\pisbm_{x} - \MCov(\pisbm_{x},\gamma)
\]
and hence
\[
\D(\psi) \geqslant \int  \MCov(\pisbm_{x},\gamma)  \, \mu(dx)  = P(\mu,\nu)= \tilde{D}(\mu,\nu),
\]
where the last equality follows from  Theorem \ref{theorem_no_duality_gap}. Thus \eqref{lem_ext_dual_func_01} is satisfied.

Finally, from \eqref{lem_ext_dual_func_01} and recalling \eqref{eq_prop_crucial_m}, we conclude
\[
\tilde{D}(\mu,\nu) 
\leqslant \inf_{\substack{\mu(\dom \psi) = 1, \\ \textnormal{$\psi$ convex}} }\D(\psi) 
\leqslant \inf_{\substack{\psi \in \Cq, \\ \textnormal{$\psi$ convex}} }\D(\psi) 
= \tilde{D}(\mu,\nu),
\]
which shows \eqref{lem_ext_dual_func_02}. 
\end{proof}
\end{lemma}

\section{Proof of the first part of Theorem \texorpdfstring{\ref{theorem_new_duality}}{1.4}} \label{sec_pr_int_b_i}

In Subsection \ref{wtfadv} we already noted (by citing Theorem 2.2 of \cite{BaBeHuKa20}) that the value of the continuous-time optimization problem \eqref{MBMBB2} is equal to the value of the discrete-time formulation \eqref{eq_primal}. We also know from Theorem \ref{theorem_no_duality_gap} that there is no duality gap between the primal problem \eqref{eq_primal} and the dual problem \eqref{eq_dual}, i.e., $P(\mu,\nu) = \tilde{D}(\mu,\nu)$. Therefore we can formulate the first part of Theorem \ref{theorem_new_duality} equivalently as follows.

\begin{proposition} \label{prop_new_duality_first_part}
Let $\mu, \nu \in \PP_{2}(\Rd)$ with $\mu \lc \nu$. Then $\tilde{D}(\mu,\nu)$ is equal to
\begin{equation} \label{WeakDual_first_part}
D(\mu,\nu) = \inf_{\substack{\psi \in L^{1}(\nu), \\ \textnormal{$\psi$ convex}}} \Big( \int \psi \, d\nu - \int (\psi^{\ast} \ast \gamma)^{\ast} \, d \mu \Big).
\end{equation}
\end{proposition} 

In Subsection \ref{subsec_pr_int_b_ii} we will prove the second part of Theorem \ref{theorem_new_duality} and in particular discuss the existence of dual optimizers of \eqref{WeakDual_first_part}. Since in general we cannot expect a dual optimizer to be integrable with respect to the probability measure $\nu$, we need the following ``relaxed'' formulation of Proposition \ref{prop_new_duality_first_part}.

\begin{proposition} \label{prop_new_duality_rel} Let $\mu, \nu \in \PP_{2}(\Rd)$ with $\mu \lc \nu$. Then $\tilde{D}(\mu,\nu)$ is equal to
\begin{equation} \label{theorem_new_duality_sec_eq_s}
D_{\textnormal{rel}}(\mu,\nu) \coloneqq \inf_{\substack{\mu(\dom \psi) = 1, \\ \textnormal{$\psi$ convex}}} \mathcal{E}(\psi),
\end{equation}
where
\begin{equation} \label{theorem_new_duality_sec_eq_s_rel_for}
\mathcal{E}(\psi) \coloneqq \int \Big( \int \psi(y) \, \pisbm_{x}(dy)  - (\psi^{\ast} \ast \gamma)^{\ast}(x)\Big) \, \mu(dx),
\end{equation}
with $\pisbm$ the unique optimizer of the primal problem \eqref{eq_primal}.
\end{proposition}

The main idea behind the proofs of Propositions \ref{prop_new_duality_first_part}, \ref{prop_new_duality_rel} (which we will present at the end of this section) is to apply Proposition \ref{prop_crucial} and then to show that $\varphi^{\psi} = (\psi^{\ast} \ast \gamma)^{\ast}$, for every convex function $\psi \in \Cq$. This motivates our next goal, namely to solve the minimization problem \eqref{eq_phi_psi_x}, which we rewrite as a maximization problem
\begin{equation} \label{eq_phi_psi_x_max}
-\varphi^{\psi}(x) = \sup_{p \in \PP_{2}^{x}(\Rd)} \Big( \MCov(p,\gamma) - \int \psi \, dp\Big).
\end{equation}
As a preliminary step, we consider the simpler problem
\begin{equation} \label{eq_phi_psi}
\varrho^{\psi} \coloneqq \sup_{p \in \PP_{2}(\Rd)} \Big( \MCov(p,\gamma) - \int \psi \, dp \Big),
\end{equation}
where we do not prescribe the barycenter $x$ of $p \in \PP_{2}(\Rd)$. In Lemma \ref{lem_phi_psi_gen} below we will show for an arbitrary proper convex function $\psi \colon \Rd \rightarrow (-\infty,+\infty]$, that the value $\varrho^{\psi}$ equals $\int \psi^{\ast} \, d\gamma$. In fact, this information is essentially enough to prove Propositions \ref{prop_new_duality_first_part}, \ref{prop_new_duality_rel}.

\smallskip

Solving the maximization problem \eqref{eq_phi_psi} leads to an interesting connection with Brenier maps in Lemma \ref{lem_phi_psi} below. By \textit{Brenier's theorem} (see, e.g., \cite[Theorem 2.12]{Vi03}), the optimal transport for quadratic cost between $\gamma$ and $p \in \PP_2(\Rd)$ is induced by the $\gamma$-a.e.\ defined gradient $\nabla v$ of some convex function $v \colon \Rd \rightarrow \R$ via $(\nabla v)(\gamma) = p$. 

\begin{lemma}[``reverse Brenier''] \label{lem_phi_psi} Let $v \colon \Rd \rightarrow \R$ be a finite-valued convex function and $\psi \coloneqq v^{\ast}$ its convex conjugate. Assume that the probability measure $\hat{p} \coloneqq (\nabla v)(\gamma)$ has finite second moment. Then $\hat{p}$ is the unique maximizer of the optimization problem \eqref{eq_phi_psi} and
\[
\varrho^{\psi} = \int v \, d\gamma = \int \psi^{\ast} \, d \gamma < + \infty.
\]
\begin{proof} We denote by $T_{\gamma}^{p}$ the Brenier map from $\gamma$ to $p \in \PP_{2}(\Rd)$ and note that $T_{\gamma}^{\hat{p}} = \nabla v$. Since the convex function $v$ is finite-valued, its gradient exists $\gamma$-a.e.\ and we have that
\[
v(z) = 
\sup_{y \in \Rd} \big(\langle y,z \rangle - v^{\ast}(y)\big) 
= \langle \nabla v(z),z \rangle - v^{\ast}( \nabla v(z)),
\]
for $\gamma$-a.e.\ $z \in \Rd$. Using these observations, for $p \in \PP_{2}(\Rd)$ we get
\begin{align*}
\MCov(p,\gamma) - \int \psi \, dp 
&= \int \Big( \big\langle T_{\gamma}^{p}(z), z\big\rangle 
-  v^{\ast}\big(T_{\gamma}^{p}(z)\big) \Big)  \, \gamma(dz) \\
&\leqslant \int \sup_{y \in \Rd} \big( \langle y, z\rangle 
-  v^{\ast}(y) \big)  \, \gamma(dz) 
= \int v(z)  \, \gamma(dz) \\
&= \int \Big( \big\langle \nabla v(z),z \big\rangle - v^{\ast}\big( \nabla v(z)\big) \Big)  \, \gamma(dz) \\
&= \MCov(\hat{p},\gamma) - \int \psi \, d\hat{p},
\end{align*}
with equality if and only if $T_{\gamma}^{p}(z) = T_{\gamma}^{\hat{p}}(z)$, for $\gamma$-a.e.\ $z \in \Rd$. This in turn is the case if and only if $p = \hat{p}$. Finally, from the convexity of $v$ and the Cauchy--Schwarz inequality we obtain
\begin{equation} \label{lem_phi_psi_ccse}
\int \vert v \vert \, d \gamma \leqslant \vert v(0) \vert + \sqrt{\int \vert \nabla v \vert^{2} \, d \gamma} \cdot \sqrt{d} < + \infty,
\end{equation}
which proves that $\varrho^{\psi} < +\infty$. 
\end{proof}
\end{lemma}

Note that in Lemma \ref{lem_phi_psi} we started with a \textit{finite-valued} convex function $v \colon \Rd \rightarrow \R$ and then \textit{defined} $\psi \coloneqq v^{\ast}$. If we additionally \textit{know} that the probability measure $\hat{p} = (\nabla v)(\gamma)$ has finite second moment, then $\hat{p}$ is the unique maximizer of the optimization problem \eqref{eq_phi_psi}. These are rather strong assumptions. However, if we are given just a proper convex function $\psi \colon \R^{d} \rightarrow (-\infty,+\infty]$, we can still compute the value of the supremum in \eqref{eq_phi_psi}, without explicitly constructing a maximizer of this optimization problem.

\begin{lemma} \label{lem_phi_psi_gen} Let $\psi \colon \Rd \rightarrow (-\infty,+\infty]$ be a proper convex function. Then
\begin{equation} \label{lem_phi_psi_gen_04}
\varrho^{\psi} = \int \psi^{\ast}  \, d\gamma.
\end{equation}
\end{lemma}

We postpone the proof of Lemma \ref{lem_phi_psi_gen} to Appendix \ref{app_tech_lemm}. As already announced, with the help of \eqref{lem_phi_psi_gen_04}, we are now able to prove Propositions \ref{prop_new_duality_first_part}, \ref{prop_new_duality_rel}. We first show a simpler variant, where we optimize over the class of convex functions $\psi$ in $\Cq$ (i.e., which have quadratic growth), as in \eqref{prop_new_duality_first_eq_s} of Proposition \ref{prop_new_duality_first} below. Proposition \ref{prop_new_duality_rel}, where we optimize over \textit{all} $\mu$-a.s.\ finite-valued convex functions $\psi \colon \Rd \rightarrow (-\infty,+\infty]$ as in \eqref{theorem_new_duality_sec_eq_s}, is then a straightforward consequence. Proposition \ref{prop_new_duality_first_part}, where we optimize over all $\nu$-integrable convex functions $\psi$ as in \eqref{WeakDual_first_part}, follows from a ``sandwich argument''.

\begin{proposition} \label{prop_new_duality_first} Let $\mu, \nu \in \PP_{2}(\Rd)$ with $\mu \lc \nu$. Then $\tilde{D}(\mu,\nu)$ is equal to
\begin{equation} \label{prop_new_duality_first_eq_s} 
D_{\textnormal{q}}(\mu,\nu) \coloneqq \inf_{\substack{\psi \in \Cq, \\ \textnormal{$\psi$ convex}} } \Big( \int \psi \, d\nu - \int (\psi^{\ast} \ast \gamma)^{\ast} \, d \mu \Big).
\end{equation}
\begin{proof} By Proposition \ref{prop_crucial} we have
\[
\tilde{D}(\mu,\nu) = \inf_{\substack{\psi \in \Cq, \\ \textnormal{$\psi$ convex}}}  \Big( \int \psi \, d\nu - \int \varphi^{\psi} \, d \mu \Big).
\]
Hence it remains to show that $\varphi^{\psi} = (\psi^{\ast} \ast \gamma)^{\ast}$, for every convex function $\psi \in \Cq$. In order to do this, we will first prove that $(\varphi^{\psi})^{\ast} = \psi^{\ast} \ast \gamma$. By definition of the convex conjugate and \eqref{eq_phi_psi_x_max}, for $\zeta \in \Rd$, we have
\begin{align}
(\varphi^{\psi})^{\ast}(\zeta) 
&= \sup_{x \in \Rd} \big(\langle x,\zeta \rangle - \varphi^{\psi}(x)\big) \label{prop_new_duality_first_a} \\
&= \sup_{x \in \Rd} \sup_{p \in \PP_{2}^{x}(\Rd)} \Big( \MCov(p,\gamma)-\int \psi_{\zeta} \, dp\Big) \label{prop_new_duality_first_b} \\
&= \sup_{p \in \PP_{2}(\Rd)} \Big(\MCov(p,\gamma)- \int \psi_{\zeta} \, dp\Big)
= \varrho^{\psi_{\zeta}}, \label{prop_new_duality_first_c}
\end{align}
where the function $\psi_{\zeta}$ is defined by $\psi_{\zeta}(y) \coloneqq \psi(y) - \langle \zeta,y \rangle$, for $y \in \Rd$. Now applying Lemma \ref{lem_phi_psi_gen} to the proper convex function $\psi_{\zeta}$ yields
\begin{equation} \label{prop_new_duality_first_d}
(\varphi^{\psi})^{\ast}(\zeta)  
= \varrho^{\psi_{\zeta}} 
= \int \psi_{\zeta}^{\ast} \, d\gamma 
= \int \psi^{\ast}(\zeta+z)  \, d\gamma(z) 
= (\psi^{\ast} \ast \gamma)(\zeta).
\end{equation}
To complete the proof, we must justify that $(\varphi^{\psi})^{\ast\ast} = \varphi^{\psi}$, i.e.\ that the Fenchel--Moreau theorem is applicable. To see this, we first note that the function $x \mapsto \varphi^{\psi}(x)$ is convex (recall Lemma \ref{prop_vppx}) and we have the upper bound $\varphi^{\psi} \leqslant \psi < + \infty$. Furthermore, for any $p \in \PP_{2}^{x}(\Rd)$ we have the inequalities
\[
\MCov(p,\gamma) \leqslant \tfrac{1}{2} \int \vert y \vert^{2} \, dp(y) + \tfrac{d}{2}
\]
as well as
\[
\int \psi \, dp \geqslant \ell + \tfrac{1}{2} \int \vert y \vert^{2} \, dp(y),
\]
the latter following from the fact that $\psi \in \Cq$. As a consequence, we get the lower bound $\varphi^{\psi} \geqslant \ell - \frac{d}{2} > - \infty$. Altogether, $\varphi^{\psi}$ is a convex function which is finite everywhere on $\R^{d}$, thus it is continuous and we indeed have $(\varphi^{\psi})^{\ast\ast} = \varphi^{\psi}$.
\end{proof}
\end{proposition}

\begin{proof}[Proof of Proposition \ref{prop_new_duality_rel}] We first show that the function $\mathcal{E}(\, \cdot \,)$ in \eqref{theorem_new_duality_sec_eq_s_rel_for} is well defined for every convex function $\psi \colon \Rd \rightarrow (-\infty,+\infty]$, which is $\mu$-a.s.\ finite. To this end, we will prove the inequality $(\psi^{\ast} \ast \gamma)^{\ast} \leqslant \psi$. By Jensen's inequality, this implies that the integrand in \eqref{theorem_new_duality_sec_eq_s_rel_for} is $\mu$-a.s.\ non-negative, and hence $\mathcal{E}(\psi)$ is well defined and $[0,+\infty]$-valued. Recalling the equation \eqref{prop_new_duality_first_d} above, we have $\psi^{\ast} \ast \gamma = (\varphi^{\psi})^{\ast}$. Taking the convex conjugate and using that $\varphi^{\psi} \leqslant \psi$, we obtain $(\psi^{\ast} \ast \gamma)^{\ast} = (\varphi^{\psi})^{\ast\ast} \leqslant \varphi^{\psi} \leqslant \psi$, as required.

\smallskip

Now let us turn to the proof of $\tilde{D}(\mu,\nu) = D_{\textnormal{rel}}(\mu,\nu)$. Recalling \eqref{theorem_new_duality_sec_eq_s}, \eqref {prop_new_duality_first_eq_s}, and using Proposition \ref{prop_new_duality_first}, we have $D_{\textnormal{rel}}(\mu,\nu) \leqslant D_{\textnormal{q}}(\mu,\nu) = \tilde{D}(\mu,\nu)$, so that we need to show the inequality $\tilde{D}(\mu,\nu) \leqslant D_{\textnormal{rel}}(\mu,\nu)$. Let $\pi \in \MT(\mu,\nu)$ and $\psi \colon \Rd \rightarrow (-\infty,+\infty]$ be convex with $\mu(\dom \psi) = 1$. Since $\tilde{D}(\mu,\nu) = P(\mu,\nu)$ (recall Theorem \ref{theorem_no_duality_gap} and \eqref{eq_the_no_duality_gap}), we have to verify the inequality
\[
\int \MCov(\pisbm_{x},\gamma) \, d\mu(x)
\leqslant  \int \Big( \int \psi(y) \, \pisbm_{x}(dy)  - (\psi^{\ast} \ast \gamma)^{\ast}(x)\Big) \, \mu(dx).
\]
It is sufficient to prove, for $\mu$-a.e.\ $x \in \Rd$, that
\begin{equation} \label{cor_new_duality_sec_01} 
\MCov(\pisbm_{x},\gamma) 
\leqslant  \int \psi \, d\pisbm_{x} - (\psi^{\ast} \ast \gamma)^{\ast}(x).
\end{equation}
We express the convex conjugate on the right-hand side of \eqref{cor_new_duality_sec_01} as
\[
-(\psi^{\ast} \ast \gamma)^{\ast}(x) 
= \inf_{\zeta \in \Rd} \big( (\psi^{\ast} \ast \gamma)(\zeta) - \langle \zeta,x \rangle \big) 
= \inf_{\zeta \in \Rd} \Big( \int \psi_{\zeta}^{\ast} \, d\gamma - \langle \zeta,x \rangle \Big),
\]
where the function $\psi_{\zeta}$ is defined by $\psi_{\zeta}(y) \coloneqq \psi(y) - \langle \zeta,y \rangle$, for $y \in \Rd$. Substituting back into \eqref{cor_new_duality_sec_01} yields
\begin{equation} \label{cor_new_duality_sec_02} 
\MCov(\pisbm_{x},\gamma) \leqslant \inf_{\zeta \in \Rd} \Big(\int \psi_{\zeta} \, d \pisbm_{x} + \int \psi_{\zeta}^{\ast} \, d\gamma\Big).
\end{equation}
Now observe that by the Fenchel--Young inequality we have
\[
\MCov(\pisbm_{x},\gamma) \leqslant \int \phi \, d \pisbm_{x} + \int \phi^{\ast} \, d\gamma,
\]
for every proper convex function $\phi \colon \Rd \rightarrow (-\infty,+\infty]$. In particular, for every $\zeta \in \Rd$, it holds that
\[
\MCov(\pisbm_{x},\gamma) \leqslant \int \psi_{\zeta} \, d \pisbm_{x} + \int \psi_{\zeta}^{\ast} \, d\gamma,
\]
which implies \eqref{cor_new_duality_sec_02}. This completes the proof of Proposition \ref{prop_new_duality_rel}.
\end{proof}

\begin{proof}[Proof of Proposition \ref{prop_new_duality_first_part}] The assertion follows immediately from Proposition \ref{prop_new_duality_rel} and Proposition \ref{prop_new_duality_first} by a ``sandwich argument''. Indeed, we have
\begin{equation} \label{eq_sw_arg}
\tilde{D}(\mu,\nu) = 
D_{\textnormal{rel}}(\mu,\nu)  \leqslant D(\mu,\nu) \leqslant D_{\textnormal{q}}(\mu,\nu) 
= \tilde{D}(\mu,\nu).
\end{equation}
The inequalities in \eqref{eq_sw_arg} are due to the inclusions
\[
\{ \psi \in \Cq\colon \textnormal{$\psi$ convex} \} \subseteq \{ \psi \in L^{1}(\nu) \colon \textnormal{$\psi$ convex} \} \subseteq \{ \textnormal{$\psi$ convex} \colon \mu(\dom \psi) = 1  \};
\]
the equalities on the left-hand side and on the right-hand side of \eqref{eq_sw_arg} are justified by Proposition \ref{prop_new_duality_rel} and Proposition \ref{prop_new_duality_first}, respectively.
\end{proof}

\section{Preparation for the proof of the second part of Theorem \texorpdfstring{\ref{theorem_new_duality}}{1.4}} \label{sec_prep}

The goal of this rather technical section is to outline and prepare the proof of the second part of Theorem \ref{theorem_new_duality}, for which we will need the results of Sections \ref{sec_dual_bass} and \ref{sec_irr_bass}. In Section \ref{sec_irr_bass} we will show that the value $D(\mu,\nu)$ is attained by a convex function $\psiopt$ with $\mu(\operatorname{ri}(\dom\psiopt)) = 1$ if and only if $(\mu,\nu)$ is irreducible. In Section \ref{sec_dual_bass} (see Theorem \ref{theo_du_op_b_m}) we will prove that there is such a dual optimizer $\psiopt$ if and only if there exists a Bass martingale from $\mu$ to $\nu$. At a first reading one might skip the present Section \ref{sec_prep}. 

\smallskip

We recall the primal problem \eqref{eq_primal}, with optimizer $\pisbm \in \MT(\mu,\nu)$, i.e.
\begin{equation} \label{eq_primal_sec5} 
P(\mu,\nu) = \int \MCov(\pisbm_{x},\gamma) \, \mu(dx);
\end{equation}
the dual problem \eqref{eq_dual}, with optimizer $\psiopt$ (supposing that this optimizer exists), i.e.
\begin{equation} \label{eq_dual_sec5} 
\tilde{D}(\mu,\nu) = \Big( \int \psiopt \, d\nu - \int \varphi^{\psiopt} \, d \mu \Big);
\end{equation}
and the fact that there is no duality gap by Theorem \ref{theorem_no_duality_gap}. Therefore, comparing the primal value \eqref{eq_primal_sec5} with the dual value \eqref{eq_dual_sec5}, we see that the minimization problem
\[
\varphi^{\psiopt}(x) = \inf_{p \in \PP_{2}^{x}(\Rd)} \Big( \int \psiopt \, dp - \MCov(p,\gamma)\Big)
\]
(recall \eqref{eq_phi_psi_x}) is attained by $\pisbm_{x} \in \PP_{2}^{x}(\Rd)$, for $\mu$-a.e.\  $x \in \Rd$. To draw the connection to the Bass martingale, we need to express the optimizer $\pisbm_{x}$ in terms of the (single) convex function $\psiopt$. More precisely, we will see as a consequence of Proposition \ref{prop_phi_psi_x_gen_opt} that $\pisbm_{x} = (\nabla \psiopt^{\ast})(\gamma_{\zeta})$, for some appropriate $\zeta = \zeta(x) \in \Rd$. For this purpose, we have to return to the optimization problem \eqref{eq_phi_psi_x} and study its optimizers.

\smallskip
 
First, we need the following technical result, formulated in Lemma \ref{lem_co_fin} below. Let $\psi \colon \Rd \rightarrow (-\infty,+\infty]$ be a lower semicontinuous convex function. We recall from convex analysis that $\psi$ is called co-finite if 
\[
\forall  x \in \Rd \setminus \{ 0 \} \colon \quad \lim_{t \rightarrow +\infty} \frac{\psi(t x)}{t} = +\infty.
\] 
According to \cite[Corollary 13.3.1]{Ro70}, the convex function $\psi$ is co-finite if and only if its convex conjugate $\psi^{\ast}$ is finite everywhere on $\Rd$. We derive now a sufficient condition for the co-finiteness of $\psi$, in terms of the function $\varphi^{\psi}$ defined in \eqref{eq_phi_psi_x}. We defer the proof of Lemma \ref{lem_co_fin} to Appendix \ref{app_tech_lemm}.

\begin{lemma} \label{lem_co_fin} Let $\psi \colon \Rd \rightarrow (-\infty,+\infty]$ be a proper, lower semicontinuous convex function. Then $\varphi^{\psi}$ is convex on $\dom\psi$. If we additionally assume that $\varphi^{\psi}(x) > - \infty$, for some $x \in \operatorname{int} (\dom \psi)$, then $\varphi^{\psi} > - \infty$ on $\operatorname{int}(\dom \psi)$ and $\psi$ is co-finite.
\end{lemma}

Our next result is the analogue of Lemma \ref{lem_phi_psi_gen}. But now, instead of maximizing over all $p \in \PP_{2}(\Rd)$ as in \eqref{eq_phi_psi}, we have to maximize over all $p \in \PP_{2}^{x}(\Rd)$, with a fixed barycenter $x \in \Rd$, as in \eqref{eq_phi_psi_x_max}.

\begin{proposition} \label{prop_phi_psi_x_gen} Let $\psi \colon \Rd \rightarrow (-\infty,+\infty]$ be a lower semicontinuous convex function and assume that $\varphi^{\psi}(x) > - \infty$ for some $x \in \operatorname{int} (\dom \psi)$. Then we have the duality formula
\begin{equation} \label{eq_prop_phi_psi_x_gen_df}
\varphi^{\psi}(x) = \sup_{\zeta \in \Rd} \Big(\langle \zeta,x \rangle - \int \psi^{\ast}(\zeta+z)  \, d\gamma(z)\Big),
\end{equation}
and the right-hand side admits a unique maximizer $\zeta(x) \in \Rd$.
\begin{proof} We deploy a similar strategy as in the proof of Proposition \ref{prop_new_duality_first}. Recalling equations \eqref{prop_new_duality_first_a} -- \eqref{prop_new_duality_first_d}, for each $\zeta \in \Rd$, we have
 \begin{equation} \label{eq_phi_psi_x_gen_01_i}
(\varphi^{\psi})^{\ast}(\zeta) 
= \sup_{p \in \PP_{2}(\Rd)} \Big(\MCov(p,\gamma)- \int \psi_{\zeta} \, dp\Big)
= \int \psi^{\ast}(\zeta+z)  \, d\gamma(z),
\end{equation}
where the function $\psi_{\zeta}$ is defined by $\psi_{\zeta}(y) \coloneqq \psi(y) - \langle \zeta,y \rangle$, for $y \in \Rd$. Since $\varphi^{\psi}(x) > - \infty$ for some $x \in \operatorname{int} (\dom \psi)$, the convex function $\xi \mapsto \varphi^{\psi}(\xi)$ is finite and thus also continuous in a neighbourhood of $x$. Hence we can apply the variant \cite[Proposition 13.44]{BaCo17} of the Fenchel--Moreau theorem and obtain
 \begin{equation} \label{eq_phi_psi_x_gen_01_ii}
\varphi^{\psi}(x) 
= (\varphi^{\psi})^{\ast\ast}(x) 
= \sup_{\zeta \in \Rd} \Big(\langle \zeta,x \rangle - \int \psi^{\ast}(\zeta+z)  \, d\gamma(z)\Big),
\end{equation}
which proves the duality formula \eqref{eq_prop_phi_psi_x_gen_df}.

\smallskip

To prove the existence of a maximizer $\zeta(x) \in \Rd$, we define the function
\begin{equation} \label{eq_phi_psi_x_gen_01}
f_{x}(\zeta) \coloneqq  \sup_{p \in \PP_{2}(\Rd)} \Big(  \MCov(p,\gamma) - \int \psi \,dp + \int \langle \zeta,y -x \rangle \, dp(y) \Big),
\end{equation}
so that $-\varphi^{\psi}(x) = \inf_{\zeta \in \Rd} f_{x}(\zeta)$. Indeed, it follows from \eqref{eq_phi_psi_x_gen_01_i}, \eqref{eq_phi_psi_x_gen_01_ii} that
\begin{align*}
-\varphi^{\psi}(x) 
&= \inf_{\zeta \in \Rd} \Big(\int \psi^{\ast}(\zeta+z)  \, d\gamma(z) - \langle \zeta,x \rangle  \Big) \\
&= \inf_{\zeta \in \Rd}  \sup_{p \in \PP_{2}(\Rd)} \Big(\MCov(p,\gamma)- \int \psi_{\zeta} \, dp - \langle \zeta,x \rangle  \Big).
\end{align*}
By assumption, the function $f_{x}$ takes values in $(-\infty,+\infty]$ and is not constant equal to $+\infty$. Equivalently, we can express $f_{x}(\zeta)$ as
 \begin{equation} \label{eq_phi_psi_x_gen_02}
f_{x}(\zeta) = \int \big( \psi^{\ast}(\zeta+z) + \psi(x) - \langle \zeta + z,x\rangle \big) \, d\gamma(z) - \psi(x).
\end{equation}
We will show that $\zeta \mapsto f_{x}(\zeta)$ is lower semicontinuous and coercive, implying the existence of an optimizer $\zeta(x)$. It is easy to see from the representation \eqref{eq_phi_psi_x_gen_02} that $f_{x}$ is lower semicontinuous. Indeed, this directly follows from the lower semicontinuity of the non-negative integrand in \eqref{eq_phi_psi_x_gen_02} and Fatou's lemma. For the verification of the coercivity, we take a sequence $(\zeta^{(n)})_{n \geqslant 1}$ in $\Rd$ with $\vert \zeta^{(n)} \vert \rightarrow + \infty$. Then there is a coordinate $k \in \{1, \ldots, d\}$ such that $\vert \zeta_{k}^{(n)} \vert \rightarrow + \infty$. Since $x \in \operatorname{int} (\dom\psi)$, we can choose $\varepsilon > 0$ small enough such that $x \pm \varepsilon \, e_{k} \in \dom\psi$, where $e_{k}$ denotes the $k$-th standard basis vector of $\Rd$. Defining
\[
y^{(n)} \coloneqq x + \operatorname{sign}(\zeta_{k}^{(n)}) \, \varepsilon \, e_{k} \in \dom\psi
\]
and taking $p= \delta_{y^{(n)}}$, we conclude from \eqref{eq_phi_psi_x_gen_01} that
\[
f_{x}(\zeta^{(n)}) \geqslant  - \psi(y^{(n)})  + \varepsilon \, \vert \zeta_{k}^{(n)} \vert \rightarrow + \infty,
\]
which shows $\lim_{ \vert \zeta \vert \rightarrow +\infty}f_{x}(\zeta) = +\infty$.

\smallskip 

We now show that the maximizer $\zeta(x) \in \Rd$ of the right-hand side of \eqref{eq_prop_phi_psi_x_gen_df} is unique. As the assumptions of Lemma \ref{lem_co_fin} are satisfied, the convex function $\psi$ is co-finite, and thus $\psi^{\ast}$ is finite everywhere on $\Rd$. In particular, $\psi^{\ast}$ is continuous everywhere and differentiable Lebesgue-a.e.\ on $\Rd$. In order to establish the uniqueness of $\zeta(x)$, we show that the function $\zeta \mapsto \int \psi^{\ast}(\zeta+z) \, d\gamma(z) \in (-\infty,+\infty]$ is strictly convex on its domain. By contradiction, suppose that there are $\zeta_{1},\zeta_{2} \in \Rd$ with $\zeta_{1} \neq \zeta_{2}$ such that both $\int \psi^{\ast}(\zeta_{1}+z) \, d\gamma(z)$ and $\int \psi^{\ast}(\zeta_{2}+z) \, d\gamma(z)$ are finite and 
\[
\int \psi^{\ast}\big(t \zeta_{1}+(1-t)\zeta_{2} + z \big) \, d\gamma(z) 
= t \int \psi^{\ast}(\zeta_{1}+z) \, d\gamma(z) 
+ (1-t) \int \psi^{\ast}(\zeta_{2}+z) \, d\gamma(z),
\]
for some $t \in (0,1)$. Then
\[
\psi^{\ast}\big(t \zeta_{1}+(1-t)\zeta_{2} + z \big) 
= t\psi^{\ast}(\zeta_{1}+z) + (1-t)\psi^{\ast}(\zeta_{2}+z),
\]
for $\gamma$-a.e.\ $z \in \Rd$. Consequently, the function $\psi^{\ast}$ is affine on the line segment from $\zeta_{1} + z$ to $\zeta_{2} + z$, for Lebesgue-a.e.\ $z \in \Rd$. By convexity and finiteness of $\psi^{\ast}$ on $\Rd$, we thus have that the function
\[
\R \in \lambda \longmapsto \psi^{\ast}\big(\lambda(\zeta_{2}-\zeta_{1})+z\big) 
\]
is affine, for every $z \in \Rd$. By \cite[Theorem 8.8]{Ro70}, there exists some $c \in \Rd$ such that  
\[
\psi^{\ast}\big(\lambda(\zeta_{2}-\zeta_{1})+z\big) = \psi^{\ast}(z) + \lambda c.
\]
Differentiating at the point $\lambda = 1$ yields $\langle \nabla \psi^{\ast}(z),\zeta_{2}-\zeta_{1}\rangle= c$, for Lebesgue-a.e.\ $z \in \Rd$. But this is a contradiction to $\varnothing \neq \operatorname{int} (\dom \psi) \subseteq  (\partial \psi^{\ast})(\Rd)$, where the symbol $\partial$, applied to a convex function, denotes its subdifferential.
\end{proof}
\end{proposition}

The final result of this section is the analogue of Lemma \ref{lem_phi_psi}. The duality formula \eqref{eq_prop_phi_psi_x_gen_df} established in Proposition \ref{prop_phi_psi_x_gen} enables us to identify the structure of the optimizer in the maximization problem \eqref{eq_phi_psi_x_max}.
 
\begin{proposition} \label{prop_phi_psi_x_gen_opt} Let $\psi \colon \Rd \rightarrow (-\infty,+\infty]$ be a lower semicontinuous convex function. For fixed $\zeta \in \Rd$ we define the function $\psi_{\zeta}$ by $\psi_{\zeta}(y) \coloneqq \psi(y) - \langle \zeta,y \rangle$, for $y \in \Rd$, so that $\psi_{\zeta}^{\ast}(y) = \psi^{\ast}(\zeta+y)$.
\begin{enumerate}[label=(\roman*)] 
\item \label{prop_phi_psi_x_gen_opt_i} If there are $x \in \Rd$ and $\zeta = \zeta(x) \in \Rd$ such that $(\nabla \psi_{\zeta}^{\ast})(\gamma) \in \PP_{2}^{x}(\Rd)$, then this is actually the unique optimizer of the supremum in \eqref{eq_phi_psi_x_max}. If additionally $\varphi^{\psi}(x) > - \infty$ and $x \in \operatorname{int} (\dom \psi)$, then $\zeta = \zeta(x)$ is the optimizer of the supremum in \eqref{eq_prop_phi_psi_x_gen_df}.
\item \label{prop_phi_psi_x_gen_opt_ii} If $\varphi^{\psi}(x) > - \infty$ for some $x \in \operatorname{int} (\dom \psi)$ and the supremum in \eqref{eq_phi_psi_x_max} is attained by some $\hat{p}_{x} \in \PP_{2}^{x}(\Rd)$, then $\hat{p}_{x} = (\nabla \psi_{\zeta}^{\ast})(\gamma)$ with $\zeta = \zeta(x)$ being the optimizer of the supremum in \eqref{eq_prop_phi_psi_x_gen_df}.
\end{enumerate}
\begin{proof} \ref{prop_phi_psi_x_gen_opt_i} We set $\hat{p}_{x} \coloneqq (\nabla \psi_{\zeta}^{\ast})(\gamma)$ and let $p_{x} \in \PP_{2}^{x}(\Rd)$ be arbitrary. We denote by $T_{\gamma}^{p_{x}}$ the Brenier map from $\gamma$ to $p_{x}$ and note that $T_{\gamma}^{\hat{p}_{x}} = \nabla \psi_{\zeta}^{\ast}$. Similar to the proof of Lemma \ref{lem_phi_psi}, we compute
\begin{align*}
\int \psi \, dp_{x} - \MCov(p_{x},\gamma)
&= \int \Big( \psi_{\zeta}\big(T_{\gamma}^{p_{x}}(z)\big)
+ \big\langle \zeta, T_{\gamma}^{p_{x}}(z)\big\rangle
- \big\langle T_{\gamma}^{p_{x}}(z),z\big\rangle 
 \Big) \, d\gamma(z) \\
&\geqslant 
\langle \zeta, x\rangle + \int \inf_{y \in \Rd} \big( 
\psi_{\zeta}(y) - \langle y,z\rangle
 \big) \, d\gamma(z)  \\
&= \langle \zeta, x\rangle - \int \psi^{\ast}(\zeta + z) \, d\gamma(z)  \\
&= \int \psi \, d\hat{p}_{x} - \MCov(\hat{p}_{x},\gamma),
\end{align*}
with equality if and only if $T_{\gamma}^{p_{x}}(z) = T_{\gamma}^{\hat{p}_{x}}(z)$, for $\gamma$-a.e.\ $z \in \Rd$. This in turn is the case if and only if $p_{x} = \hat{p}_{x}$. We conclude that $\hat{p}_{x}$ is the unique optimizer of the supremum in \eqref{eq_phi_psi_x_max}, i.e.
\begin{equation} \label{eq_phi_psi_x_gen_opt_i}
\varphi^{\psi}(x) 
= \int \psi \, d\hat{p}_{x} - \MCov(\hat{p}_{x},\gamma) 
= \langle \zeta, x\rangle - \int \psi^{\ast}(\zeta + z) \, d\gamma(z).
\end{equation}
On the other hand, if additionally $\varphi^{\psi}(x) > - \infty$ and $x \in \operatorname{int} (\dom \psi)$, by Proposition \ref{prop_phi_psi_x_gen} we have that
\[
\varphi^{\psi}(x) 
= \sup_{\zeta \in \Rd} \Big(\langle \zeta,x \rangle - \int \psi^{\ast}(\zeta+z)  \, d\gamma(z)\Big),
\]
which in light of \eqref{eq_phi_psi_x_gen_opt_i} implies that $\zeta = \zeta(x)$ is the optimizer of the supremum in \eqref{eq_prop_phi_psi_x_gen_df}.

\smallskip

\noindent \ref{prop_phi_psi_x_gen_opt_ii} By assumption and Proposition \ref{prop_phi_psi_x_gen}, we have that
\begin{align*}
\varphi^{\psi}(x) 
&= \langle \zeta,x \rangle + 
\int \Big( \psi_{\zeta}\big(T_{\gamma}^{\hat{p}_{x}}(z)\big)
- \big\langle T_{\gamma}^{\hat{p}_{x}}(z),z\big\rangle 
 \Big) \, d\gamma(z) \\
&= \langle \zeta,x \rangle - \int \psi_{\zeta}^{\ast}(z)  \, d\gamma(z) 
\end{align*}
and therefore
\[
\int \Big( \psi_{\zeta}\big(T_{\gamma}^{\hat{p}_{x}}(z)\big)
+\psi^{\ast}_{\zeta}(z)
- \big\langle T_{\gamma}^{\hat{p}_{x}}(z),z\big\rangle 
 \Big) \, d\gamma(z) = 0.
\]
As the integrand is pointwise non-negative, we conclude that
\[
\psi\big(T_{\gamma}^{\hat{p}_{x}}(z)\big) 
+ \psi^{\ast}(\zeta+z)
= \big\langle T_{\gamma}^{\hat{p}_{x}}(z),\zeta + z\big\rangle 
\]
and consequently $T_{\gamma}^{\hat{p}_{x}}(z) \in \partial\psi^{\ast}(\zeta+z)$, for $\gamma$-a.e.\ $z \in \Rd$. Hence $T_{\gamma}^{\hat{p}_{x}}(z) = \nabla\psi^{\ast}(\zeta+z)$, for $\gamma$-a.e.\ $z \in \Rd$, which implies that $\hat{p}_{x} = (\nabla \psi_{\zeta}^{\ast})(\gamma)$.
\end{proof}
\end{proposition}

\section{Existence of dual optimizers and Bass martingales} \label{sec_dual_bass}

Throughout this section we fix $\mu, \nu \in \PP_{2}(\Rd)$ with $\mu \lc \nu$.

\begin{definition} \label{def_dual_opt} A lower semicontinuous convex function $\psiopt \colon \Rd \rightarrow (-\infty,+\infty]$ satisfying $\mu(\operatorname{ri}(\dom\psiopt))=1$ is called an optimizer of the dual problem \eqref{lem_ext_dual_func_02} (in short, a dual optimizer) if $\tilde{D}(\mu,\nu) = \mathcal{D}(\psiopt)$, for the dual function $\mathcal{D}(\, \cdot \,)$ as defined in \eqref{eq_rep_dual_func}.
\end{definition}

Let us recall that by Assumption \ref{ass_ful_dim} the support of $\nu$ affinely spans $\mathbb{R}^{d}$. As the next result shows, we could have required in the definition of a dual optimizer $\psiopt$ that also the domain of $\psiopt$ affinely spans $\mathbb{R}^{d}$; we defer the proof to Appendix \ref{app_tech_lemm}.

\begin{lemma} \label{lem:dim_dom_psi} If $\psiopt$ is an optimizer according to Definition \ref{def_dual_opt}, then the domain of $\psiopt$ has non-empty interior. In particular, by Assumption \ref{ass_ful_dim}, we have that
\[
\dim (\dom\psiopt) = \dim (\supp\nu) = d.
\]
\end{lemma}

We recall Definition \ref{def:BassMarti_intro} and list some important properties of Bass martingales. Such martingales were introduced in \cite{BaBeHuKa20} under the name of a ``standard stretched Brownian motion''. In this paper, we use --- with Richard Bass' permission --- the term ``Bass martingale'' instead.

\begin{remark} \label{rem:manifold} Let $M = (M_{t})_{0 \leqslant t \leqslant 1}$ be a Bass martingale from $\mu$ to $\nu$, with corresponding convex function $v \colon \Rd \rightarrow \R$ and initial distribution $\PP(\Rd) \ni \alpha \sim B_{0}$ of the underlying Brownian motion $B = (B_{t})_{0 \leqslant t \leqslant 1}$.
\begin{enumerate}[label=(\roman*)] 
\item As shown in \cite{BaBeHuKa20} (and as a result of Theorem \ref{theo_du_op_b_m} below), the martingale transport 
\[
\Law(M_{0},M_{1}) \in \MT(\mu,\nu)
\]
is equal to the unique optimizer $\pisbm$ of \eqref{eq_primal}. Furthermore, the knowledge of $\pisbm$ already determines the martingale $(M_{t})_{0 \leqslant t \leqslant 1}$ as well as the function $v$, which is $(\alpha \ast \gamma)$-a.e.\ (equivalently, Lebesgue-a.e.) unique up to an additive constant.
\item The convex function $v$ and the probability measure $\alpha$ satisfy the identities 
\begin{equation} \label{eq_def_id_bm}
(\nabla v \ast \gamma)(\alpha) = \mu
\qquad \textnormal{ and } \qquad 
\nabla v(\alpha \ast \gamma) = \nu,
\end{equation}
which we summarize in the following graphic:
\[
\begin{tikzcd}
\alpha \ast \gamma  \arrow[r, "\nabla v"] & \nu  \\
\alpha \arrow[u, "\ast"] \arrow[r, "\nabla v \ast \gamma"] & \mu 
\end{tikzcd}
\]
\item We also remark (see \cite{BaBeHuKa20}) that we have 
\begin{equation} \label{rep_eq_bm_n}
M_{t} = (\nabla v \ast \gamma^{1-t})(B_{t})
\quad \textnormal{ and } \quad 
\pisbm = \Law\big((\nabla v \ast \gamma)(B_{0}),\nabla v(B_{1})\big).
\end{equation}
\end{enumerate}
\end{remark}

The next result, Theorem \ref{theo_du_op_b_m} below, explains how the existence of a dual optimizer $\psiopt$ is related to the existence of a Bass martingale. Along with it we need the following technical lemmas, whose proofs we postpone to Appendix \ref{app_tech_lemm}.

\begin{lemma} \label{lem_aux_grad_convolution}
Let $f \colon \Rd \rightarrow \R$ be a finite convex function such that $\nabla f \in L^{2}(\gamma_{\zeta};\Rd)$, for some $\zeta \in \Rd$. Then 
\begin{enumerate}[label=(\roman*)] 
\item \label{lem_aux_grad_convolution_i} $\nabla(f \ast \gamma) = (\nabla f) \ast \gamma$.
\end{enumerate}
If additionally $\operatorname{int} (\dom f^{\ast}) \neq \varnothing$, then 
\begin{enumerate}[label=(\roman*)] 
\setcounter{enumi}{1}
\item \label{lem_aux_grad_convolution_ii} $(f \ast \gamma)^{\ast}$ is differentiable and strictly convex on $\operatorname{int}(\dom(f\ast\gamma)^{\ast}) \supseteq \operatorname{int} (\dom f^{\ast})$,
\item \label{lem_aux_grad_convolution_iii} $\nabla(f\ast\gamma)^{\ast} \colon \operatorname{int} (\dom (f\ast\gamma)^{\ast}) \rightarrow \Rd$ and $\nabla(f\ast\gamma)\colon \Rd \rightarrow \operatorname{int} (\dom (f\ast\gamma)^{\ast})$ are bijections, and we have
\[
(\nabla f \ast \gamma)^{-1} = \nabla(f\ast\gamma)^{\ast}.
\]
\end{enumerate}
\end{lemma}

\begin{lemma} \label{lem_int_dom_range} Let $\psi \colon \Rd \rightarrow (-\infty,+\infty]$ be a lower semicontinuous convex function and assume that $\varphi^{\psi}(x) > - \infty$, for some $x \in \operatorname{int} (\dom \psi)$. Furthermore, we suppose that $\nabla \psi^{\ast} \in L^{2}(\gamma_{\zeta};\Rd)$, for some $\zeta \in \Rd$. Then, for all $\eta \in \Rd$ and all $t > 0$, we have
\begin{equation} \label{eq_m_lem_int_dom_range} 
\operatorname{int}(\dom \psi)
= \operatorname{int} \big( (\partial \psi^{\ast})(\Rd) \big)
= (\nabla \psi^{\ast} \ast \gamma_{\eta}^{t})(\Rd).
\end{equation}
\end{lemma}

\begin{theorem} \label{theo_du_op_b_m} There exists a dual optimizer $\psiopt$ in the sense of Definition \ref{def_dual_opt} if and only if there exists a Bass martingale $(M_{t})_{0 \leqslant t \leqslant 1}$ from $\mu$ to $\nu$. In this case, $\Law(M_{0},M_{1})$ is equal to the optimizer $\pisbm \in \MT(\mu,\nu)$ of the primal problem \eqref{eq_primal}, and $\psiopt,v,\alpha, \pisbm$ are related via $v^{\ast} = \psiopt$, $\alpha = \zeta(\mu)$ and 
\begin{equation} \label{eq_id_pisbmx_v_xi}
\pisbm_{x} = \Law(M_{1} \, \vert \, M_{0} = x) 
= \nabla v\big(\zeta(x) + \, \cdot \,\big)(\gamma),
\end{equation}
where the function $\zeta \colon X \rightarrow \Rd$ is given by
\begin{equation} \label{eq_du_op_b_m_xi}
\zeta(x) \coloneqq 
\operatorname{argmax}_{\zeta \in \Rd} \Big(\langle \zeta,x \rangle - \int v(\zeta+z)  \, d\gamma(z)\Big)
\end{equation}
as in \eqref{eq_prop_phi_psi_x_gen_df}, and we have
\begin{equation} \label{eq_du_op_b_m_xi_fin}
\zeta(x) 
= (\nabla v \ast \gamma)^{-1}(x)
= \nabla(v \ast \gamma)^{\ast}(x),
\end{equation}
for all $x \in X$. The domain $X$ of the function $\zeta$ satisfies $X \subseteq \operatorname{int}(\dom\psiopt)$ and $\mu(X) = 1$.
\begin{proof} ``$\Rightarrow$'': Suppose that a dual optimizer $\psiopt$ exists. Then $\tilde{D}(\mu,\nu) - \D(\psiopt) = 0$ and from \eqref{eq_the_no_duality_gap}, \eqref{eq_rep_dual_func} we obtain
\[
\int  \bigg(  \varphi^{\psiopt}(x) - \Big( \int \psiopt(y) \, \pisbm_{x}(dy) - \MCov(\pisbm_{x},\gamma) \Big) \bigg) \, \mu(dx) = 0.
\]
Hence, for $\mu$-a.e.\ $x \in \Rd$, the infimum in
\[
\varphi^{\psiopt}(x) = \inf_{p \in \PP_{2}^{x}(\Rd)} \Big( \int \psiopt \, dp - \MCov(p,\gamma)\Big)
\]
must be attained by $\pisbm_{x} \in \PP_{2}^{x}(\Rd)$. Furthermore, as $\D(\psiopt)$ is finite, we deduce as in the proof of Lemma \ref{lem_ext_def_dual} that $\int \psiopt \, d\pisbm_{x} < + \infty$, for $\mu$-a.e.\ $x \in \Rd$. Thus also $\varphi^{\psiopt}(x)$ is finite, for $\mu$-a.e.\ $x \in \Rd$. Recalling Lemma \ref{lem:dim_dom_psi}, we have $\mu(\operatorname{int}(\dom\psiopt)) = 1$ by Definition \ref{def_dual_opt}. In particular, $\psiopt$ is co-finite by Lemma \ref{lem_co_fin}. All in all, we can apply part \ref{prop_phi_psi_x_gen_opt_ii} of Proposition \ref{prop_phi_psi_x_gen_opt} to $\psiopt$, for $\mu$-a.e.\ $x \in \Rd$. This yields that $\pisbm_{x} = \nabla v(\zeta(x) + \, \cdot \,)(\gamma)$, with $v \coloneqq \psiopt^{\ast}$ and $\zeta(x) \in \Rd$ as in \eqref{eq_du_op_b_m_xi}. We now define $\alpha \coloneqq \zeta(\mu)$. Since $\pisbm$ has second marginal equal to $\nu$, we obtain $\nabla v(\alpha \ast \gamma) = \nu$. On the other hand, as $\pisbm_{x}$ has barycenter $x$, we have
\begin{equation} \label{eq_du_op_b_m_xi_fin_pre}
\int \nabla v\big(\zeta(x)+z\big)  \, d\gamma(z) 
= (\nabla v \ast \gamma)\big(\zeta(x)\big)
= x.
\end{equation}
We denote the set of points $x \in \Rd$ for which the identity \eqref{eq_du_op_b_m_xi_fin_pre} holds by $X$ and note that $X \subseteq \operatorname{int}(\dom\psiopt)$ as well as $\mu(X) = 1$. Since $\alpha = \zeta(\mu)$, we conclude from \eqref{eq_du_op_b_m_xi_fin_pre} that $(\nabla v \ast \gamma)(\alpha) = \mu$. By \eqref{eq_def_id_bm}, this establishes the existence of a Bass martingale $(M_{t})_{0 \leqslant t \leqslant 1}$ from $\mu$ to $\nu$, which by construction satisfies $\Law(M_{0},M_{1}) = \pisbm$, and connects $\psiopt,v,\alpha, \pisbm$ as claimed.

\smallskip

It remains to check the identity \eqref{eq_du_op_b_m_xi_fin}. Since $\pisbm_{x} \in \PP_{2}^{x}(\Rd)$, and as we already established \eqref{eq_id_pisbmx_v_xi}, we conclude that
\begin{equation} \label{eq_du_op_b_m_psi_v_a_00_fvf}
\int \big\vert \nabla v\big(\zeta(x) + z\big) \big\vert^{2} \, d\gamma(z) < + \infty,
\end{equation}
for every $x \in X$. In particular, we can apply part \ref{lem_aux_grad_convolution_iii} of Lemma \ref{lem_aux_grad_convolution} to the function $f = v$. Then from \eqref{eq_du_op_b_m_xi_fin_pre} we derive \eqref{eq_du_op_b_m_xi_fin}.

\medskip

``$\Leftarrow$'': Conversely, suppose that a Bass martingale $M = (M_{t})_{0 \leqslant t \leqslant 1}$ from $\mu$ to $\nu$ exists. We denote by $v \colon \Rd \rightarrow \R$ the associated convex function and define the proper, lower semicontinuous convex function $\psiopt \coloneqq v^{\ast}$. Recall that $\dim (\supp \nu) = d$ by Assumption \ref{ass_ful_dim}. Since $(\partial v)(\Rd) \subseteq \dom v^{\ast}$ and 
\[
\operatorname{int} \conv \big( (\partial v)(\Rd) \big) = \operatorname{int} \overline{\conv} \big( (\partial v)(\Rd) \big) = \operatorname{int}\widehat{\supp}(\nu) \neq \varnothing,
\]
we also have that $\operatorname{int}(\dom v^{\ast}) \neq \varnothing$. In the following we show that $\psiopt$ is indeed a dual optimizer in the sense of Definition \ref{def_dual_opt}. 

\smallskip

First, we verify that $\varphi^{\psiopt}(x) > - \infty$, for all $x \in \Rd$. As a consequence, we will see that the identity \eqref{eq_du_op_b_m_xi_fin} holds for all $x \in \operatorname{int} (\dom \psiopt)$. Note that from \eqref{eq_def_id_bm} we have the relation $\nabla v(\alpha \ast \gamma) = \nu$. Since $\nu$ has finite second moment, we conclude that 
\begin{equation} \label{eq_du_op_b_m_psi_v_a_00}
\int \vert \nabla v(\zeta + z) \vert^{2} \, d\gamma(z) < + \infty,
\end{equation}
for $\alpha$-a.e.\ $\zeta \in \Rd$. In particular, by analogy with \eqref{lem_phi_psi_ccse}, we obtain that 
\begin{equation} \label{eq_du_op_b_m_psi_v_a_01}
(v \ast \gamma)(\zeta) = \int v(\zeta+z) \, d\gamma(z) \in (-\infty,+\infty),
\end{equation}
for $\alpha$-a.e.\ $\zeta \in \Rd$. Recalling the second equality in \eqref{eq_phi_psi_x_gen_01_ii}, we have
\begin{equation} \label{eq_du_op_b_m_psi_v_a_02}
\varphi^{\psiopt}(x) \geqslant 
(\varphi^{\psiopt})^{\ast\ast}(x) 
= \sup_{\zeta \in \Rd} \Big(\langle \zeta,x \rangle - \int \psiopt^{\ast}(\zeta+z)  \, d\gamma(z)\Big),
\end{equation}
and by \eqref{eq_du_op_b_m_psi_v_a_01}, the right-hand side of \eqref{eq_du_op_b_m_psi_v_a_02} is greater than $-\infty$, for every $x \in \Rd$. In particular, we can apply Proposition \ref{prop_phi_psi_x_gen} to the function $\psiopt$. As a result, for every $x \in \operatorname{int} (\dom \psiopt)$, we obtain an equality in \eqref{eq_du_op_b_m_psi_v_a_02} and the right-hand side admits a unique maximizer $\zeta(x) \in \Rd$. Differentiating under the integral sign, which is justified by part \ref{lem_aux_grad_convolution_i} of Lemma \ref{lem_aux_grad_convolution} applied to the function $f = v = \psiopt^{\ast}$, the first order condition for the optimality of $\zeta(x)$ reads
\begin{equation} \label{eq_du_op_b_m_psi_v_a_03}
(\nabla \psiopt^{\ast} \ast \gamma)\big(\zeta(x)\big)
= x.
\end{equation}
Now we use part \ref{lem_aux_grad_convolution_iii} of Lemma \ref{lem_aux_grad_convolution} and obtain the identity \eqref{eq_du_op_b_m_xi_fin} from \eqref{eq_du_op_b_m_psi_v_a_03} above, for all $x \in \operatorname{int} (\dom \psiopt)$. 

\smallskip

Next, we show that $\alpha = \zeta(\mu)$ and $\mu(\operatorname{int}(\dom\psiopt)) = 1$. Recalling from \eqref{eq_def_id_bm} that $(\nabla v \ast \gamma)(\alpha) = \mu$, and using \eqref{eq_du_op_b_m_xi_fin}, we get $\alpha = (\nabla v \ast \gamma)^{-1}(\mu) = \zeta(\mu)$. For the second claim, we again use \eqref{eq_du_op_b_m_xi_fin} and obtain
\[
1 = \alpha(\Rd) 
= \mu\big((\nabla v \ast \gamma)(\Rd)\big)
= \mu(\operatorname{int}(\dom\psiopt)),
\]
where the last equality follows from Lemma \ref{lem_int_dom_range}.

\smallskip

It remains to check that $\tilde{D}(\mu,\nu) = \mathcal{D}(\psiopt)$ and $\Law(M_{0},M_{1}) = \pisbm$. For each $x \in \operatorname{int} (\dom \psiopt)$, we define a probability measure $\pi_{x}^{v} \coloneqq \nabla v(\zeta(x) + \, \cdot \,)(\gamma)$. By \eqref{eq_du_op_b_m_psi_v_a_00} and \eqref{eq_du_op_b_m_psi_v_a_03}, $\pi_{x}^{v}$ is an element of $\PP_{2}^{x}(\Rd)$, for $\mu$-a.e.\ $x \in \Rd$. Hence we can apply part \ref{prop_phi_psi_x_gen_opt_i} of Proposition \ref{prop_phi_psi_x_gen_opt}, showing that the infimum
\[
\varphi^{\psiopt}(x) = \inf_{p \in \PP_{2}^{x}(\Rd)} \Big( \int \psiopt \, dp - \MCov(p,\gamma)\Big)
\]
is attained by $\pi_{x}^{v}$. 

Set $\pi^v(dx,dy)=\pi^v_x(dy)\mu(dx)$. Since $M$ is a Bass martingale with associated convex function $v$, we have $\pi^v_x= \text{Law}(M_1|M_0=x)$, for $\mu$-a.e.\ $x\in \Rd$. Moreover, by Theorem 1.10 of \cite{BaBeHuKa20}, $M$ is stretched Brownian motion; hence, by uniqueness of the primal optimizer, $$\pi^v=\text{Law}(M_0,M_1)=\pisbm .$$
Therefore, from the representation \eqref{eq_rep_dual_func} of the dual function $\mathcal{D}(\, \cdot \,)$ and the identity $\pi^v=\pisbm$, we obtain
\begin{equation} \label{eq_du_op_b_m_psi_v_a_04}
\D(\psiopt) = \int  \Big( \int \psiopt(y) \, \pi^{v}_{x}(dy)  - \varphi^{\psiopt}(x) \Big) \, \mu(dx)
= \int \MCov(\pi_{x}^{v},\gamma) \, \mu(dx),
\end{equation}
and this is equal to $\leqslant P(\mu,\nu) 
= \tilde{D}(\mu,\nu)$.
 This completes the proof that $\psiopt$ is a dual optimizer. The identity $\Law(M_{0},M_{1})= \pisbm$ has already been obtained above, and the equality $\pi^v_x= \text{Law}(M_1|M_0=x)$ gives \eqref{eq_id_pisbmx_v_xi}.
\end{proof}
\end{theorem}

We now give a one-dimensional example in which a dual optimizer $\psiopt$ in the sense of Definition \ref{def_dual_opt} is not integrable with respect to $\nu$. Example \ref{ex_du_not_int} below illustrates that the form \eqref{eq_def_dual_func} of the dual function $\D(\, \cdot \,)$ may fail to make sense when we allow $\psi$ to range more generally than in $\Cq$. On the other hand, the dual function written in the form \eqref{eq_rep_dual_func} makes perfect sense for general convex functions $\psi \colon \Rd \rightarrow (-\infty,+\infty]$, which are $\mu$-a.s.\ finite (recall Lemma \ref{lem_ext_def_dual}). This leads to the satisfactory characterization of dual attainment as given by Theorem \ref{theo_du_op_b_m}.

\begin{example} \label{ex_du_not_int} Let $\alpha \coloneqq \sum_{n=1}^{\infty} 2^{-n} \delta_{z_n} \in \PP(\R)$, with the sequence $(z_{n})_{n \geqslant 1} \subseteq \R$ satisfying $\lim_{n \rightarrow \infty} \frac{\vert z_n \vert}{2^{n}} = + \infty$. Consider the convex function $v(z) \coloneqq z \arctan z - \frac{1}{2} \log(1+z^{2})$, with the derivative $v'(z) = \arctan z$ being a strictly increasing, continuous and bounded function. Define $\mu \coloneqq (v' \ast \gamma)(\alpha)$, so that $\mu = \sum_{n=1}^{\infty} 2^{-n} \delta_{x_n}$ for some bounded sequence $(x_n)_{n \geqslant 1} \subseteq \R$, and let $\nu \coloneqq v'(\alpha \ast \gamma)$. Note that $\mu$ and $\nu$ are supported by $(-\frac{\pi}{2},\frac{\pi}{2})$. Then, according to \eqref{eq_def_id_bm}, the pair $(v,\alpha)$ defines a Bass martingale from $\mu$ to $\nu$. Theorem \ref{theo_du_op_b_m} shows that the derivative of the dual optimizer is given by 
\[
\tfrac{d}{dy} \psiopt(y) = (v')^{-1}(y) =\tan y.
\]
By definition of $\nu$ and $\alpha$, this function is not $\nu$-integrable, and neither is its antiderivative $\psiopt$. \hfill $\Diamond$ 
\end{example}

We finish this section by deducing a trajectorial property of Bass martingales from Lemma \ref{lem_int_dom_range} and Theorem \ref{theo_du_op_b_m}, namely that a Bass martingale $M$ with $M_{1}\sim \nu$ can only reach the boundary of $\widehat{\supp}(\nu)$ at time $1$. In this paper we only make use of this result as a technical step in the proof of Lemma \ref{lemma_irr_f_p}. However, this is a natural property to examine from the point of view of stochastic analysis (e.g.\ Feller's explosion test is devoted to a related question), and it gives us some intuition about the behavior of Bass martingales.

\begin{corollary} \label{cor_bm_tr_pr} Let $M = (M_{t})_{0 \leqslant t \leqslant 1}$ be a Bass martingale with $M_{1} \sim \nu$. Define $\tau$ as the stopping time 
\[
\tau \coloneqq 
\inf\big\{t\in[0,1] \colon M_{t} \notin \operatorname{ri}\widehat{\supp}(\nu) \big\} \wedge 1, 
\]
i.e., the minimum of $1$ and the first time that $M$ reaches the boundary of $\widehat{\supp}(\nu)$. Then $\tau=1$ a.s.
\begin{proof} 
We will prove that 
\[
\mathbb{P}\big(t \in [0,1) \colon M_{t}\in\operatorname{ri}\widehat{\supp}(\nu)\big)=1. 
\]
Thus, if $M$ ever reaches the boundary of $\widehat{\supp}(\nu)$, this will happen at time $t = 1$, so $\tau = 1$. Otherwise, i.e., if $M$ never reaches the boundary of $\widehat{\supp}(\nu)$, then by definition $\tau = 1$, too.

Recall that $\dim (\supp \nu) = d$ by Assumption \ref{ass_ful_dim}, so we can replace the relative interior operator by the interior operator throughout the argument. By \eqref{rep_eq_bm_n}, the Bass martingale is given by $M_{t} = (\nabla v \ast \gamma^{1-t})(B_{t})$, where $\nabla v(\gamma) =\nu$, for some finite-valued convex function $v$. We now define the proper, lower semicontinuous convex function $\psi \coloneqq v^{\ast}$. As in the proof of the second implication of Theorem \ref{theo_du_op_b_m}, namely the part of the proof corresponding to ``$\Leftarrow$'', we see that $\operatorname{int}(\dom v^{\ast}) \neq \varnothing$, that $\varphi^{\psi}(x) > - \infty$ for all $x \in \Rd$, and that $\nabla \psi^{\ast} \in L^{2}(\gamma_{B_0};\Rd)$ a.s.\ (for the latter, see \eqref{eq_du_op_b_m_psi_v_a_00} and recall that $B_{0}\sim\alpha$). Therefore we can apply Lemma \ref{lem_int_dom_range} and from \eqref{eq_m_lem_int_dom_range} we obtain for all $t \in [0,1)$ the  equality
\[
\operatorname{int}(\dom \psi)
= \operatorname{int}  \big( (\partial \psi^{\ast})(\Rd) \big)
= (\nabla \psi^{\ast} \ast \gamma^{1-t})(\Rd).
\]
As we justify in \eqref{eq_m_lem_int_dom_range_01} in Appendix \ref{app_tech_lemm}, we automatically have the equality
\[
\operatorname{int} \big( (\partial \psi^{\ast})(\Rd) \big) 
= \operatorname{int} \conv \big( (\partial \psi^{\ast})(\Rd) \big), 
\]
so we conclude that 
\[
\operatorname{int} \conv \big( (\partial \psi^{\ast})(\Rd) \big)= (\nabla \psi^{\ast} \ast \gamma^{1-t})(\Rd).
\]
On the other hand, we clearly have
\[
\operatorname{int} \conv \big( (\partial \psi^{\ast})(\Rd) \big) = \operatorname{int} \overline{\conv} \big( (\partial \psi^{\ast})(\Rd) \big) = \operatorname{int}\widehat{\supp}(\nu),
\]
so we obtain for all $t \in [0,1)$ that 
\[
\operatorname{int}\widehat{\supp}(\nu) = (\nabla \psi^{\ast} \ast \gamma^{1-t})(\Rd).
\]
Since $M_t= (\nabla \psi^\ast \ast \gamma^{1-t})(B_{t}) $, we conclude as desired that $M_{t}$ lies in $\operatorname{int}\widehat{\supp}(\nu)$ as long as $t \in [0,1)$. 
\end{proof}

\end{corollary}

\section{Irreducibility and existence of dual optimizers} \label{sec_irr_bass}

\subsection{Outline and objectives}

In this technically demanding section we connect our findings with the groundbreaking analysis 
presented in \cite{DeTo17} by H.\ De March and N.\ Touzi. Their work establishes the existence of a 
unique maximal paving of $\R^d$ into relatively open, convex, invariant sets. Each martingale transport from $\mu$ to $\nu$ maps these relatively open sets into their closures. Our focus here is on the irreducible case, where this paving is trivial, i.e., it consists of a single cell. The general case is treated in the follow-up paper \cite{ScTs24}.

\smallskip

We show that the triviality of the De March--Touzi (DMT) paving (as formalized in Definition \ref{def_dmti} below) is equivalent to the irreducibility condition in Definition \ref{defi:irreducible_intro}. This equivalence, along with additional equivalent characterizations of irreducibility, is established in Theorem \ref{theo_coic} of Appendix \ref{app_sec_irr}, which builds on the results of the present section.

\smallskip

Assuming the triviality of the DMT paving in Assumption \ref{assumptions_single_cell} below, we fix a martingale transport $\pidmt \in \MT(\mu,\nu)$ such that every transition kernel $\pidmt_{x}$ has full convex support. The existence of such a martingale transport $\pidmt$ was established in \cite{DeTo17}. In Lemma \ref{lemma_bsbi_s}, we prove a uniformity property of this full-support feature, valid for $\mu$-a.e.\ $x \in \Rd$.

\smallskip

The main result of this section is Theorem \ref{theorem_single}, which links the triviality of the DMT paving to the existence of an optimizer for the dual problem \eqref{lem_ext_dual_func_02}. As a direct consequence of this theorem, we obtain that triviality of the DMT paving implies the existence of a Bass martingale $M$ from $\mu$ to $\nu$ with $\Law(M_{0},M_{1}) = \pisbm$. Notably, this Bass martingale serves as a concrete example of a De March--Touzi transport since, for $\mu$-a.e.\ $x \in \Rd$, the measures $\pidmt_{x}$ are equivalent to $\nu$.

\smallskip

The first step in proving Theorem \ref{theorem_single} is to identify a limit $\psilim$ of an arbitrary optimizing sequence of convex functions $(\psi_{n})_{n \geqslant 1}$ for the dual problem \eqref{eq_dual}. A crucial observation is that the value of the dual function \eqref{eq_rep_dual_func} remains invariant under the addition of affine functions to $\psi$. This invariance allows us to choose and add affine functions that facilitate our analysis. Concrete examples of such choices of affine functions are provided in \cite{ScTs24} (Examples 6.1 and 6.2). After appropriately selecting representatives of $\psi_{n}$ (by adding affine functions), we establish in Lemma \ref{lemma_bsbi_smao} (by applying Koml\'{o}s' theorem) the existence of a sequence of Ces\`{a}ro means of a subsequence that is bounded on compact subsets of $I = \operatorname{ri} \widehat{\supp}(\nu)$. Using arguments from convex analysis, we extract a further subsequence that converges uniformly on compact subsets $K \subseteq I$ to a convex function $\psilim$.

\smallskip

The second step in proving Theorem \ref{theorem_single} is to verify that the limiting function $\psilim$ is indeed a dual optimizer. To this end, we introduce in Definition \ref{def_gsbmx} the Brenier map $\nabla \vxsbm$ from $\gamma$ to $\pisbm_{x}$. The next goal (see Lemma \ref{lemma_irr_f}) is to show that, for $\mu$-a.e.\ $x \in I$, the function $\psilim$ equals $(\vxsbm)^{\ast}$, modulo the addition of an affine function. Specifically, we introduce (see \eqref{def_set_a}) the set
\[
A = \big\{ x \in I \colon \psilim \not\equiv (\vxsbm)^{\ast} \textnormal{ mod (aff)}\big\},
\]
and prove that $\mu(A) = 0$. This proof proceeds by contradiction. Assuming $\mu(A) > 0$, we construct, for $x \in A$, measures $\check{\pi}_{x} \in \PP^{2}_{x}(\Rd)$ in Lemma \ref{lemma_a_h}, which are adjusted in Lemma \ref{lemma_irr_f_p} to have compact support. The measures $\check{\pi}_{x}$ then contradict the optimality of $\pisbm_{x}$, allowing us to conclude $\mu(A) = 0$ and thereby complete the proof of Lemma \ref{lemma_irr_f}. With the help of Lemma \ref{lemma_ch_o_do} we can then finish the second step and thus also the proof of Theorem \ref{theorem_single}. Broadly speaking, Lemma \ref{lemma_ch_o_do} shows that $\psilim$ is a dual optimizer if and only if $\mu(A) = 0$. The proof of this result is based on Proposition \ref{prop_phi_psi_x_gen_opt} and Theorem \ref{theo_du_op_b_m}.

\smallskip

At this point we encounter a remarkable and somewhat surprising feature. In Proposition \ref{prop_1b} we use the wisdom of hindsight to show that in the arguments outlined above, there is no need of passing to a subsequence or of forming Ces\`{a}ro means of the optimizing sequence $(\psi_{n})_{n \geqslant 1}$. Rather already the original sequence $(\psi_{n})_{n \geqslant 1}$ --- modulo adding affine functions --- converges to the optimizer $\psilim$. Note that in Proposition \ref{prop_1b} the pair $(\mu,\nu)$ is not necessarily irreducible, rather the boundedness assumption \eqref{prop_1b_ass} is made (cf.\ Lemma \ref{lemma_bsbi_smao}). This refinement plays a crucial role in the follow-up paper \cite{ScTs24}. 

\smallskip

Towards the end of Section \ref{sec_irr_bass} we are finally in the position to prove Theorem \ref{MainTheorem} as well as the second part of Theorem \ref{theorem_new_duality}, thereby completing the main body of the paper (apart from four appendices).

\smallskip

In the following subsections we now proceed with a formal and detailed treatment of the outlined program for Section \ref{sec_irr_bass}.

\subsection{Irreducibility and dual attainment}

We fix $\mu, \nu \in \PP_{2}(\Rd)$ with $\mu \lc \nu$. The following notation from \cite{DeTo17} will be used throughout this section.

\begin{definition} We denote by $C \coloneqq \widehat{\supp}(\nu)$ the closed convex hull of the support of $\nu$ and by $I \coloneqq \operatorname{ri} C$ its relative interior. 
\end{definition}

Recall that $\dim (\supp \nu) = d$ by Assumption \ref{ass_ful_dim}, so that $I$ is open in $\Rd$.

\begin{definition} \label{def_dmti} We say that the pair $(\mu,\nu)$ is \textit{De March--Touzi irreducible}, if the irreducible convex paving of De March--Touzi \cite{DeTo17} consists of the single irreducible component $I$. By \cite[Theorem 2.1]{DeTo17}, this condition means that
\begin{enumerate}[label=(\roman*)] 
\item \label{dmt_p1} the set $I = \operatorname{ri} C$, with $C = \widehat{\supp}(\nu)$, satisfies $\mu(I) = 1$,
\item \label{dmt_p2} there exists some martingale transport $\pidmt \in \MT(\mu,\nu)$ with the property that $C = \widehat{\supp}(\pidmt_{x})$, for $\mu$-a.e.\ $x \in \Rd$. We refer to $\pidmt$ as a \textit{De March--Touzi transport}. 
\end{enumerate}
\end{definition}

The main assumption of this section (with the exception of Lemma \ref{lemma_psi_sup}, Lemma \ref{lemma_ch_o_do}, Proposition \ref{prop_1b} and Corollary \ref{cor_1b}) is the following.

\begin{assumption} \label{assumptions_single_cell} The pair $(\mu,\nu)$ is De March--Touzi irreducible. 
\end{assumption}

Assumption \ref{assumptions_single_cell} is equivalent to the irreducibility of the pair $(\mu,\nu)$ in the sense of Definition \ref{defi:irreducible_intro}. For a proof of this result and for further equivalent characterizations of irreducibility we refer to Theorem \ref{theo_coic} in Appendix \ref{app_sec_irr}. 

\smallskip

In Lemma \ref{lemma_bsbi_s} below we give a more quantitative description of the defining property of a De March--Touzi transport $\pidmt \in \MT(\mu,\nu)$ between an irreducible pair $(\mu,\nu)$. 

\begin{definition} Let $K \subseteq I$ be compact. For an element $y^{\ast}$ of the unit sphere $\mathbb{S}^{d-1}$ in $\Rd$ we define the slice $S_{y^{\ast}}$ of $C$ beyond $K$ by 
\begin{equation} \label{def_slice}
S_{y^{\ast}} \coloneqq \big\{ y \in C \colon \langle y,y^{\ast} \rangle > \sup \{ \langle \tilde{y},y^{\ast} \rangle \colon \tilde{y} \in K\} \big\}.
\end{equation}
\end{definition}

\begin{lemma} \label{lemma_bsbi_s} Under Assumption \ref{assumptions_single_cell}, for $\mu$-a.e.\ $x \in \Rd$ and for every compact set $K \subseteq I$, there exists a constant $\delta(K,x) > 0$ such that $\pidmt_{x}(S_{y^{\ast}}) \geqslant \delta(K,x)$ for every $y^{\ast} \in \mathbb{S}^{d-1}$.
\begin{proof} We fix $x \in \Rd$ with $C = \widehat{\supp}(\pidmt_{x})$ and note that the map $r_{x} \colon \mathbb{S}^{d-1} \rightarrow [0,1]$ given by $r_{x}(y^{\ast}) \coloneqq \pidmt_{x}(S_{y^{\ast}})$, for $y^{\ast} \in \mathbb{S}^{d-1}$, is lower semicontinuous. As the set $C$ equals the closed convex hull of the support of $\pidmt_{x}$, the map $r_{x}$ is also strictly positive. Indeed, if there was some $y^{\ast} \in \mathbb{S}^{d-1}$ with $r_{x}(y^{\ast}) = 0$, then the closed convex set $C \setminus S_{y^{\ast}}$ would support $\pidmt_{x}$, which is in contradiction to $C \setminus S_{y^{\ast}} \subsetneq C = \widehat{\supp}(\pidmt_{x})$. Since $\mathbb{S}^{d-1}$ is compact we conclude that $r_{x}$ attains its minimum $\delta(K,x) \coloneqq \min_{y^{\ast} \in \mathbb{S}^{d-1}} r_{x}(y^{\ast}) > 0$.
\end{proof}
\end{lemma}

The main result of this section is the following.

\begin{theorem} \label{theorem_single} Under Assumption \ref{assumptions_single_cell}, there exists a dual optimizer $\psiopt$ in the sense of Definition \ref{def_dual_opt}.
\end{theorem}

Together with Theorem \ref{theo_du_op_b_m}, this has the following consequence for the primal optimizer $\pisbm \in \MT(\mu,\nu)$ of the primal problem \eqref{eq_primal}.

\begin{corollary} \label{thm:equivalence_kernels} Under Assumption \ref{assumptions_single_cell}, there exists a Bass martingale $(M_{t})_{0 \leqslant t \leqslant 1}$ from $\mu$ to $\nu$. Moreover, $\Law(M_{0},M_{1}) = \pisbm$ and, for $\mu$-a.e.\ $x \in \Rd$, the measure $\pisbm_{x}$ is equivalent to $\nu$. In particular, $\pisbm$ is a De March--Touzi transport in the sense of Definition \ref{def_dmti}, \ref{dmt_p2}.
\begin{proof} Admitting Theorem \ref{theorem_single}, there exists a dual optimizer $\psiopt$. Thus, by Theorem \ref{theo_du_op_b_m}, there is a Bass martingale from $\mu$ to $\nu$ with the property that $\Law(M_{0},M_{1}) = \pisbm$. In particular, by \eqref{eq_id_pisbmx_v_xi}, there is a measurable set $A$ with $\mu(A) = 1$, such that $\pisbm_{x}$ and $\pisbm_{x'}$ are image measures of suitably translated Gaussians under the same function $\nabla v$, for all $x,x' \in A$. Hence $\pisbm_{x} \sim \pisbm_{x'}$, for all $x, x' \in A$, and this in turn implies the equivalence of $ \nu(dy) = \int \pisbm_{x}(dy) \, \mu(dx)$ with $\pisbm_{x'}(dy)$, for each $x' \in A$. In particular, $\supp (\pisbm_{x}) = \supp(\nu)$, for $\mu$-a.e.\ $x \in \Rd$, which implies Definition \ref{def_dmti}, \ref{dmt_p2} for $\pisbm$.
\end{proof}
\end{corollary}

\subsection{Proof of Theorem \texorpdfstring{\ref{theorem_single}}{7.6}} \label{subs_theorem_single} The proof is split into two main steps. 

\subsubsection{Step 1 of the proof of Theorem \texorpdfstring{\ref{theorem_single}}{7.6}} We first construct a \textit{convergent} optimizing sequence of convex functions $(\psi_{n})_{n \geqslant 1} \subseteq \Cqaff$ (see Remark \ref{rem_cqaff_def} below) for the dual problem \eqref{eq_dual}.

\begin{proposition} \label{prop_step_1_the_sing} Under Assumption \ref{assumptions_single_cell}, there is an optimizing sequence $(\psi_{n})_{n \geqslant 1}$ of non-negative convex functions in $\Cqaff$ for the dual problem \eqref{eq_dual}, which converges compactly on $I$ (i.e., uniformly on compact subsets $K \subseteq I$) to some convex function $\psilim \colon I \rightarrow [0,+\infty)$. 
\end{proposition}

Before we turn to the proof of Proposition \ref{prop_step_1_the_sing}, we still need some preparation. The following auxiliary result does not require the irreducibility Assumption \ref{assumptions_single_cell}, but solely relies on the finiteness of the value $\tilde{D}(\mu,\nu)$ of the dual problem \eqref{eq_dual}.

\begin{lemma} \label{lemma_psi_sup} Let $(\psi_{n})_{n \geqslant 1} \subseteq \Cq$ be an optimizing sequence for the dual problem \eqref{eq_dual} and take any $\pi \in \MT(\mu,\nu)$. Then
\begin{equation} \label{lemma_psi_sup_aa}
\sup_{n \geqslant 1} \int \int \big( \psi_{n}(y) - \psi_{n}(x) \big) \, \pi_{x}(dy) \, \mu(dx) < +\infty.
\end{equation}
\begin{proof} Recalling \eqref{eq_phi_psi_x}, we consider for $n \geqslant 1$ the functions
\begin{equation} \label{eq_phi_psi_n_al}
\Rd \ni x \longmapsto \varphi^{\psi_{n}}(x) = \inf_{p \in \PP_{2}^{x}(\Rd)} \Big( \int \psi_{n} \, dp - \MCov(p,\gamma)\Big).
\end{equation}
By taking $p = \delta_{x}$ in \eqref{eq_phi_psi_n_al} we obtain the trivial estimate $\varphi^{\psi_{n}}(x) \leqslant \psi_{n}(x)$ and consequently
\[
\int \int \big( \psi_{n}(y) - \psi_{n}(x) \big) \, \pi_{x}(dy) \, \mu(dx) 
\leqslant \int \int \big( \psi_{n}(y) - \varphi^{\psi_{n}}(x) \big) \, \pi_{x}(dy) \, \mu(dx).
\]
Since $\psi_n\in \Cq$ and $\pi$ has second marginal $\nu$, the definition \eqref{eq_def_dual_func} gives
\[
\D(\psi_{n}) = \int  \Big( \int \psi_{n}(y) \, \pi_{x}(dy)  - \varphi^{\psi_{n}}(x) \Big) \, \mu(dx),
\]
and as the sequence of real numbers $(\mathcal{D}(\psi_{n}))_{n \geqslant 1}$ converges to the finite number $\tilde{D}(\mu,\nu)$, we conclude \eqref{lemma_psi_sup_aa}.
\end{proof}
\end{lemma}

\begin{definition} \label{def_psi_nx} We fix an arbitrary optimizing sequence $(\psi_{n})_{n \geqslant 1} \subseteq \Cq$ of \textit{convex} functions for the dual problem \eqref{eq_dual}, which is possible thanks to Proposition \ref{prop_crucial}. For $x \in \Rd$, we define
\begin{equation} \label{eq_def_psi_n_x_opt_s}
\psi_{n}^{x}(\, \cdot \,) \coloneqq \psi_{n}(\, \cdot \, )  - \psi_{n}(x)  - \langle \partial \psi_{n}(x), \, \cdot \, - x\rangle,
\end{equation}
so that $\psi_{n}^{x}(x) = 0$ and $\psi_{n}^{x}(\, \cdot \,)$ takes values in $[0,+\infty)$, for every $n \geqslant 1$. Here --- by abuse of notation --- $\partial \psi_{n}(x)$ denotes a subgradient of $\psi_{n}$ at $x$, i.e., an arbitrary element of the subdifferential of $\psi_{n}$ at $x$.
\end{definition}

\begin{remark} \label{rem_cqaff_def} We observe that passing from $\psi \in \Cq$ to $\psi + \aff$, where $\aff \colon \Rd \rightarrow \R$ is an arbitrary affine function, does not change the value of the dual function \eqref{eq_rep_dual_func}, i.e., $\D(\psi) = \D(\psi + \aff)$. In particular, we note that the sequence $(\psi_{n}^{x})_{n \geqslant 1}$ defined in \eqref{eq_def_psi_n_x_opt_s} is contained in the class of functions $\Cqaff$, as introduced in \eqref{def_cqaff}, and hence $\D(\psi_{n}^{x}) = \D(\psi_{n})$.
\end{remark}

Our goal is to show that --- under the irreducibility Assumption \ref{assumptions_single_cell} and after passing to Ces\`{a}ro means of a suitable subsequence --- the sequence $(\psi_{n}^{x})_{n \geqslant 1}$ of \eqref{eq_def_psi_n_x_opt_s} is bounded on compact subsets of $I$, for $\mu$-a.e.\ $x \in I$. This is the content of the following lemma.

\begin{lemma} \label{lemma_bsbi_smao} Under Assumption \ref{assumptions_single_cell}, for any given optimizing sequence $(\psi_{n})_{n \geqslant 1} \subseteq \Cq$ of convex functions, there is a sequence of Ces\`{a}ro means of a subsequence, still denoted by $(\psi_{n})_{n \geqslant 1}$, such that, for $\mu$-a.e.\ $x \in I$ and for every compact set $K \subseteq I$, we have that $\sup_{n \geqslant 1, y \in K} \psi_{n}^{x}(y) < + \infty$, for every choice of $(\psi_{n}^{x})_{n \geqslant 1} \subseteq \Cqaff$ as in \eqref{eq_def_psi_n_x_opt_s}.
\begin{proof} We rely on \cite{DeTo17} and fix a De March--Touzi transport $\pidmt \in \MT(\mu,\nu)$ satisfying point \ref{dmt_p2} of Definition \ref{def_dmti}, which exists by the irreducibility Assumption \ref{assumptions_single_cell}. Let $(\psi_{n})_{n \geqslant 1}$ and $(\psi_{n}^{x})_{n \geqslant 1}$, measurable selected for $x \in \Rd$, be as in Definition \ref{def_psi_nx} above. By Lemma \ref{lemma_psi_sup}, the sequence 
\[
\Psi_{n}(x) \coloneqq  \int \psi_{n}^{x}(y) \, \pidmt_{x}(dy) 
= \int \big( \psi_{n}(y) - \psi_{n}(x) \big) \, \pidmt_{x}(dy) 
\]
is bounded in $L^{1}(\mu)$. Applying Koml\'{o}s' theorem \cite[Theorem 1]{Koml67}, we can find a subsequence $(n_{k})_{k \geqslant 1} \subseteq \mathbb{N}$, such that the Ces\`{a}ro means 
\[
\overline{\Psi}_{k} \coloneqq \frac{\Psi_{n_{1}} + \Psi_{n_{2}}+ \ldots + \Psi_{n_{k}}}{k}
\]
converge $\mu$-a.s.\ to some random variable $\Psi \in L^{1}(\mu)$. If we define the Ces\`{a}ro means
\[
\bar{\psi}_{k}^{x} \coloneqq \frac{\psi_{n_{1}}^{x} + \psi_{n_{2}}^{x}+ \ldots + \psi_{n_{k}}^{x}}{k},
\]
we have 
\[
\overline{\Psi}_{k}(x) = \int \bar{\psi}_{k}^{x}(y) \, \pidmt_{x}(dy), \qquad x \in \Rd.
\]
Note that passing to a subsequence $(n_{k})_{k \geqslant 1} \subseteq \mathbb{N}$ and forming Ces\`{a}ro means $(\bar{\psi}_{k}^{x})$ preserves the property of being an optimizing sequence of convex functions in $\Cqaff$ of the form \eqref{eq_def_psi_n_x_opt_s}. Therefore we can replace the original sequences $(\Psi_{n})$ and $(\psi_{n}^{x})$ by $(\overline{\Psi}_{k})$ and $(\bar{\psi}_{k}^{x})$, respectively, and may again relabel them as $(\Psi_{n})$ and $(\psi_{n}^{x})$, respectively. With that said, as a consequence of the $\mu$-a.s.\ convergence to a finite limit, we have that
\begin{equation} \label{eq_ko_sd_v}
m(x) \coloneqq \sup_{n \geqslant 1}  \Psi_{n}(x) 
= \sup_{n \geqslant 1}  \int \psi_{n}^{x}(y) \, \pidmt_{x}(dy) < +\infty,
\end{equation}
for $\mu$-a.e.\ $x \in \Rd$.

\smallskip

Arguing by contradiction to the statement of Lemma \ref{lemma_bsbi_smao}, we assume that there is a compact set $K_{0} \subseteq I$ such that the set
\[
A_{1} \coloneqq \Big\{ x \in \Rd \colon \sup_{n \geqslant 1, y \in K_{0}} \psi_{n}^{x}(y) = + \infty \Big\}
\]
has positive $\mu$-measure. Furthermore, by Lemma \ref{lemma_bsbi_s}, the set
\[
A_{2} \coloneqq \Big\{ x \in \Rd \colon \ \exists \, \delta(K_{0},x) > 0 \textnormal{ such that } \pidmt_{x}(S_{y^{\ast}}) \geqslant \delta(K_{0},x) \textnormal{ for every } y^{\ast} \in \mathbb{S}^{d-1} \Big\}
\]
has full $\mu$-measure. By \eqref{eq_ko_sd_v} above, also the set
\[
A_{3} \coloneqq \{ x \in \Rd \colon m(x) < + \infty \}
\]
has full $\mu$-measure, so that the intersection $A \coloneqq A_{1} \cap A_{2} \cap A_{3}$ has positive $\mu$-measure. Pick a point $x_{0} \in A$. As $x_{0} \in A_{1}$, for arbitrarily large $M > m(x_{0}) / \delta(K_{0},x_{0})$, we can find $n_{0} \geqslant 1$ and $y_{0} \in K_{0}$ such that $\psi_{n_{0}}^{x_{0}}(y_{0}) \geqslant M$. The function 
\[
\ell(y) \coloneqq \big\langle \partial \psi_{n_{0}}^{x_{0}}(y_{0}), y-y_{0} \big\rangle + M, \qquad y \in \Rd,
\]
satisfies $\ell \leqslant \psi_{n_{0}}^{x_{0}}$ and $\ell(y) \geqslant M$, for all $y \in S_{y^{\ast}}$, with $y^{\ast} \coloneqq \frac{z}{\vert z \vert}$ and $z \coloneqq \partial \psi_{n_{0}}^{x_{0}}(y_{0})$. We conclude with
\[
m(x_{0}) \geqslant
\int \psi_n^x(y) \, \pidmt_{x_{0}}(dy)
\geqslant \int \ell^{+}(y) \, \pidmt_{x_{0}}(dy)
\geqslant M \, \pidmt_{x_{0}}(S_{y^{\ast}}) 
\geqslant M \, \delta(K_{0},x_{0}),
\]
which is the desired contradiction, as $M$ is arbitrarily large.
\end{proof}
\end{lemma}

\begin{proof}[Proof of Proposition \ref{prop_step_1_the_sing}] By Lemma \ref{lemma_bsbi_smao} we can fix some $x_{0} \in I$ such that the sequence of convex functions $(\psi_{n}^{x_{0}})_{n \geqslant 1}$ is bounded on all compact subsets $K \subseteq I$. In particular, the sequence $(\psi_{n}^{x_{0}})_{n \geqslant 1}$ is pointwise bounded on $I$. By \cite[Theorem 10.9]{Ro70} we can select a subsequence, still denoted by $(\psi_{n}^{x_{0}})_{n \geqslant 1}$, which converges uniformly on compact subsets $K \subseteq I$ to some convex function $\psilim^{x_{0}} \colon I \rightarrow [0,+\infty)$. Dropping the superscript $x_{0}$ to ease notation, we arrive at a sequence $(\psi_{n})_{n \geqslant 1}$ with limit $\psilim$ as required in the statement of Proposition \ref{prop_step_1_the_sing}.
\end{proof}

\subsubsection{Step 2 of the proof of Theorem \texorpdfstring{\ref{theorem_single}}{7.6}} We extend the function $\psilim \colon I \rightarrow [0,+\infty)$ of Proposition \ref{prop_step_1_the_sing} to a lower semicontinuous convex function $\psilim \colon \Rd \rightarrow [0,+\infty]$ which is equal to $+\infty$ on $\Rd \setminus C$. Since $I \subseteq \dom\psilim$, $I$ is open and $\mu(I) = 1$, it follows that $\mu(\operatorname{int}(\dom\psilim))=1$. 

\begin{proposition} \label{prop_step_2_the_sing} Under Assumption \ref{assumptions_single_cell}, the convex function $\psilim \colon \Rd \rightarrow [0,+\infty]$ is a dual optimizer in the sense of Definition \ref{def_dual_opt}, i.e., satisfies $\tilde{D}(\mu,\nu) = \mathcal{D}(\psilim)$.
\end{proposition}

\begin{definition} \label{def_gsbmx} For $\mu$-a.e.\ $x \in \Rd$, we denote by $\vxsbm$ the Brenier potential from $\gamma$ to $\pisbm_{x}$, so that $\nabla \vxsbm(\gamma) = \pisbm_{x}$ and 
\[
\MCov(\pisbm_{x},\gamma) 
= \int  \big\langle \nabla \vxsbm(z), z\big\rangle \, \gamma(dz);
\]
we write $\psixsbm \coloneqq (\vxsbm)^{\ast}$ for its convex conjugate.
\end{definition}

\begin{remark} \label{def_gsbmx_rem} As $\gamma$ has full support, the convex function $\vxsbm$ is finite-valued and continuous everywhere on $\Rd$. In particular, $\vxsbm$ is unique, up to an additive constant. Therefore also its convex conjugate $\psixsbm \colon \Rd \rightarrow (-\infty,+\infty]$ is unique, up to an additive constant. Also note that $\psixsbm$ is finite-valued on $I$ and takes the value $+\infty$ on $\Rd \setminus C$, with $C = \widehat{\supp}(\nu)$ and $I = \operatorname{int} C$.
\end{remark}

Our goal is to show that $\psilim \colon \Rd \rightarrow [0,+\infty]$ is a dual optimizer. For this purpose, we will prove that, for $\mu$-a.e.\ $x \in I$, the function $\psilim(\, \cdot \,)$ equals $\psixsbm(\, \cdot \,)$, modulo adding an affine function; in short, $\psilim \equiv \psixsbm$ mod (aff), for $\mu$-a.e.\ $x \in \Rd$. In Lemma \ref{lemma_irr_f} below we will show that the set
\begin{equation} \label{def_set_a}
A \coloneqq \big\{ x \in I \colon \psilim \not\equiv \psixsbm  \textnormal{ mod (aff)}\big\}
\end{equation}
indeed has $\mu$-measure zero. It will then follow from Lemma \ref{lemma_ch_o_do} that $\psilim$ is actually a dual optimizer. First, we need some auxiliary results.

\begin{lemma} \label{lemma_a_h} Under Assumption \ref{assumptions_single_cell}, for $\mu$-a.e.\ $x \in A$, there exists a measure $\check{\pi}_{x} \in \PP_{2}^{x}(\Rd)$ supported by $C$ and there is a constant $\tilde{\beta}(x) > 0$ such that
\begin{equation} \label{lemma_a_h_in}
\MCov(\check{\pi}_{x},\gamma) + \int \psilim \, d(\pisbm_{x}-\check{\pi}_{x}) 
\geqslant \MCov(\pisbm_{x},\gamma) + \tilde{\beta}(x).
\end{equation}
\begin{proof} Recalling the representation \eqref{eq_rep_dual_func} of the dual function $\D(\, \cdot \,)$ and the definition \eqref{eq_phi_psi_x} of the function $x \mapsto \varphi^{\psi}(x)$, we consider for $\mu$-a.e.\ $x \in I$ the functions 
\begin{equation} \label{eq_vph}
\begin{aligned}
\phi^{\psixsbm}(x) \coloneqq \, & \int \psixsbm(y) \, \pisbm_{x}(dy) - \varphi^{\psixsbm}(x) \\
= \, &  \sup_{p \in \PP_{2}^{x}(\Rd)} \Big( \MCov(p,\gamma) + \int \psixsbm \, d(\pisbm_{x}-p) \Big)
\end{aligned}
\end{equation}
and
\begin{equation} \label{eq_vpc} 
\begin{aligned}
\phi^{\psilim}(x)  \coloneqq \, & \int \psilim(y) \, \pisbm_{x}(dy) - \varphi^{\psilim}(x) \\
= \, &  \sup_{p \in \PP_{2}^{x}(\Rd)} \Big( \MCov(p,\gamma) + \int \psilim \, d(\pisbm_{x}-p) \Big).
\end{aligned}
\end{equation}
Recalling Definition \ref{def_gsbmx}, we have $\nabla (\psixsbm)^{\ast}(\gamma) = \pisbm_{x} \in \PP_{2}^{x}(\Rd)$. Therefore, we can apply part \ref{prop_phi_psi_x_gen_opt_i} of Proposition \ref{prop_phi_psi_x_gen_opt}, which yields that the supremum in \eqref{eq_vph} is attained by $\pisbm_{x}$, so that $\phi^{\psixsbm}(x) = \MCov(\pisbm_{x},\gamma)$. By taking $p = \pisbm_{x}$ in \eqref{eq_vpc}, we obtain the inequality $\phi^{\psilim}(x) \geqslant \phi^{\psixsbm}(x)$, for $\mu$-a.e.\ $x \in I$.

\smallskip

Now we define the sets
\[
B \coloneqq \big\{ x \in I \colon \phi^{\psilim}(x) > \phi^{\psixsbm}(x) \big\}, \qquad
X \coloneqq \big\{ x \in I \colon \phi^{\psixsbm}(x) < + \infty \big\}
\]
and claim that for the set $A$ defined in \eqref{def_set_a} we have the relation $\tilde{A} \coloneqq A \cap X = B$. In other words, since $\mu(X) = 1$, the sets $A$ and $B$ are equal, up to a set of $\mu$-measure zero.

In fact, if $x \notin \tilde{A}$, then $\psilim \equiv \psixsbm$ mod (aff) or $\phi^{\psixsbm}(x) = + \infty$, so that in both cases $\phi^{\psilim}(x) = \phi^{\psixsbm}(x)$. Conversely, if $x \notin B$, then $x \notin X$ or we have the equality $\phi^{\psilim}(x) = \phi^{\psixsbm}(x)$ of real numbers and thus also the supremum in the definition \eqref{eq_vpc} of $\phi^{\psilim}(x)$ is attained by
\begin{equation} \label{eq_opt_1_x}
\pisbm_{x} = \nabla (\psixsbm)^{\ast}(\gamma) \in \PP_{2}^{x}(\Rd).
\end{equation} 
Hence we can apply part \ref{prop_phi_psi_x_gen_opt_ii} of Proposition \ref{prop_phi_psi_x_gen_opt} to the lower semicontinuous convex function $\psilim \colon \Rd \rightarrow [0,+\infty]$ satisfying $\mu(\operatorname{int}(\dom\psilim))=1$, which tells us that the optimizer \eqref{eq_opt_1_x} of the supremum in \eqref{eq_vpc} is equal to 
\[
\nabla \psilim^{\ast}(\gamma_{\zeta(x)}),
\]
for some $\zeta(x) \in \Rd$. We conclude the $\gamma$-a.s.\ equality
\[
\nabla (\psixsbm)^{\ast} = \nabla \psilim^{\ast}\big(\, \cdot \, + \, \zeta(x)\big),
\]
which is equivalent to the $\gamma$-a.s.\ equality
\begin{equation} \label{eq_opt_asgehe}
(\psixsbm)^{\ast} = \psilim^{\ast}\big(\, \cdot \, + \, \zeta(x)\big) + c,
\end{equation}
for some constant $c \in \R$. Since the convex function $(\psixsbm)^{\ast}$ is finite-valued and continuous everywhere on $\Rd$, the equality \eqref{eq_opt_asgehe} has to hold everywhere on $\Rd$. This implies that $\psilim \equiv \psixsbm$ mod (aff) and we deduce that $x \notin A$. Altogether, we have seen that $x \notin B$ implies that $x \notin \tilde{A}$.

\smallskip

As a consequence of $A = B$, up to a set of $\mu$-measure zero, for $\mu$-a.e.\ $x \in A$ we have $\phi^{\psilim}(x) > \phi^{\psixsbm}(x)$, so that we can measurably select some $\check{\pi}_{x} \in \PP_2^x(\Rd)$ with
\begin{equation} \label{lemma_a_h_za} 
\MCov(\check{\pi}_{x},\gamma) + \int \psilim \, d(\pisbm_{x}-\check{\pi}_{x}) 
> \MCov(\pisbm_{x},\gamma),
\end{equation}
which gives \eqref{lemma_a_h_in}. As the right-hand side of \eqref{lemma_a_h_za} is finite, for $\mu$-a.e.\ $x \in I$, and since $\psilim(y) = + \infty$ for $y \in \Rd \setminus C$, we see that $\check \pi_{x}$ is supported by $C$. 
\end{proof}
\end{lemma} 

Our next step is to modify the measures $\{\check \pi_{x}\}_{x \in A}$ of Lemma \ref{lemma_a_h}, so that they have compact support and still satisfy \eqref{lemma_a_h_in} for some $\beta(x) > 0$ instead of $\tilde{\beta}(x)$. To this end, we choose an increasing sequence $(K_{j})_{j \geqslant 1}$ of compact subsets of $I$ such that $\bigcup_{j \geqslant 1}K_{j} = I$. Denoting by $M^{x}$ the Bass martingale from $\delta_{x}$ to $\check{\pi}_{x}$, and by $\tau_{j}^{x}$ the first exit time of $M^{x}$ from $K_{j}$ (similarly as in Corollary \ref{cor_bm_tr_pr}), we define $\check{\pi}_{x}^{j} \coloneqq \Law(M_{ \tau_{j}^{x} \wedge 1}^{x})$. By optional sampling, $\check{\pi}_{x}^{j} \in \PP_{2}^{x}(\Rd)$ and, by definition, $\check{\pi}_{x}^{j}$ is supported by the compact set $K_{j}$.   

\begin{lemma} \label{lemma_irr_f_p} Under Assumption \ref{assumptions_single_cell}, for $\mu$-a.e.\ $x \in A$, there exists $\check \pi_{x}^{j(x)} \in \PP_{2}^{x}(\Rd)$ supported by $K_{j(x)}$ for some $j(x) \in \mathbb{N}$ and there is a constant $\beta(x) > 0$ such that 
\begin{equation} \label{lemma_a_h_in_se}
\MCov(\check{\pi}_{x}^{j(x)},\gamma) + \int \psilim \, d(\pisbm_{x}-\check{\pi}^{j(x)}_{x}) 
\geqslant \MCov(\pisbm_{x},\gamma) + \beta(x).
\end{equation}
\begin{proof} From Lemma \ref{lemma_a_h} we already have the inequality \eqref{lemma_a_h_in}. In order to derive \eqref{lemma_a_h_in_se}, we have to show that
\begin{equation} \label{lemma_a_h_zb} 
\lim_{j \rightarrow +\infty} \MCov(\check \pi_{x}^{j},\gamma) = \MCov(\check \pi_{x},\gamma)
\end{equation}
and
\begin{equation} \label{lemma_a_h_zc} 
\limsup_{j \rightarrow + \infty} \int \psilim(y) \, d\check \pi_{x}^{j}(y) 
\leqslant \int \psilim(y) \, d\check \pi_{x}(y).
\end{equation}

\smallskip

We begin with the proof of \eqref{lemma_a_h_zb}. First, observe that $M_{\tau_{j}^{x}\wedge 1}^{x} \rightarrow M_{ \tau^{x} \wedge 1}^{x}$ in $L^{2}$, where $\tau^{x}$ is the first exit time of $M^{x}$ from $I$. Since $\tau^{x} \wedge 1=1$ a.s.\ by Corollary \ref{cor_bm_tr_pr}, we conclude that $M_{\tau_{j}^{x}\wedge 1}^{x} \rightarrow M_{1}^{x}$ in $L^{2}$. Consequently, 
\[
\W_{2}^{2}(\check \pi_{x}^{j},\check \pi_{x}) \coloneqq \inf_{q \in \Cpl(\check \pi_{x}^{j},\check \pi_{x})} \int \vert x_{1} - x_{2} \vert^{2} \, q(dx_{1},dx_{2})
\]
converges to zero as $j \rightarrow + \infty$, and the inequality
\[
\vert \MCov(\check \pi_{x}^{j},\gamma) - \MCov(\check \pi_{x},\gamma) \vert
\leqslant \W_{2}(\check \pi_{x}^{j},\check \pi_{x}) \, \sqrt{d}
\]
yields \eqref{lemma_a_h_zb}.

\smallskip

Finally, we show \eqref{lemma_a_h_zc}. Note that by optional sampling, $\check{\pi}_{x}^{j} \in \PP_{2}^{x}(\Rd)$. Since a martingale composed with a convex function is a submartingale, it follows that
\[
\forall j \geqslant 1 \colon \int \psilim \, d\check\pi_{x}^{j} 
\leqslant \int \psilim \, d\check\pi_{x}
\]
and we obtain \eqref{lemma_a_h_zc}.
\end{proof}
\end{lemma}

\begin{lemma} \label{lemma_irr_f} Under Assumption \ref{assumptions_single_cell}, the set $A \subseteq \Rd$ defined in \eqref{def_set_a} has $\mu$-measure zero. 
\begin{proof} We assume for contradiction that $\mu(A) > 0$. Thanks to Lemma \ref{lemma_irr_f_p}, the set $A$ is $\mu$-a.s.\ equal to the union
\[
\bigcup_{\substack{j \in \mathbb{N},\\ \beta\in \mathbb{Q}_{+} \setminus \{0\}}} 
\bigg\{x\in A \colon \MCov(\check{\pi}_{x}^{j},\gamma) + \int \psilim \, d(\pisbm_{x}-\check{\pi}^{j}_{x}) 
\geqslant \MCov(\pisbm_{x},\gamma) + \beta\bigg\}.
\]
Hence we can find a subset $B \subseteq A$ with $\mu(B) > 0$, such that \eqref{lemma_a_h_in_se} holds with uniform constants $j \in \mathbb{N}$ and $\beta > 0$ for all $x \in B$, i.e.,
\begin{equation} \label{lem_a_m_z_02}
\exists  j \ \exists  \beta \ \forall x \in B \colon \quad 
\MCov(\check{\pi}_{x}^{j},\gamma) + \int \psilim \, d(\pisbm_{x}-\check{\pi}^{j}_{x}) \geqslant \MCov(\pisbm_{x},\gamma) + \beta.
\end{equation}
Now we define a measurable collection of probability measures $\{\tilde{\pi}_{x} \}_{x \in \Rd} \subseteq \PP_{2}(\Rd)$ with $\bary(\tilde{\pi}_{x}) = x$ by 
\begin{equation} \label{def_tilde_pi}
\tilde{\pi}_{x} \coloneqq
\begin{cases}
\check{\pi}_{x}^{j}, & x \in B, \\
\pisbm_{x},& x \in \Rd \setminus B.
\end{cases}
\end{equation}
From the inequality \eqref{lem_a_m_z_02} and the definition \eqref{def_tilde_pi} we deduce that
\begin{align}
&\int \int \psilim \, d(\pisbm_{x}-\tilde{\pi}_{x}) \, d\mu(x) + \int \MCov(\tilde{\pi}_{x},\gamma)\, d\mu(x) \label{lemma_irr_f_i} \\
& \qquad \geqslant 
\int \MCov(\pisbm_{x},\gamma)  \, d\mu(x)
+ \beta \, \mu(B) \label{lemma_irr_f_ii} \\
&\qquad = P(\mu,\nu) + \beta \, \mu(B). \label{lemma_irr_f_iii}
\end{align}
Recall from Proposition \ref{prop_step_1_the_sing} that there is an optimizing sequence $(\psi_{n})_{n \geqslant 1}$ of convex functions $\psi_{n} \colon \Rd \rightarrow [0,+\infty)$ for the dual problem \eqref{eq_dual}, which converges uniformly on compact subsets $K \subseteq I$ to $\psilim$. To find the desired contradiction, we want to replace $\psilim$ in \eqref{lemma_irr_f_i} by $\lim_{n \rightarrow \infty} \psi_{n}$ and write the limit outside of the integral. This is clearly not a problem if $x \notin B$. For $x \in B$, note that the measure $\tilde{\pi}_{x} = \check{\pi}_{x}^{j}$ is supported by the compact set $K_{j}$. As $(\psi_{n})_{n \geqslant 1}$ converges to $\psilim$ uniformly on $K_{j}$, we have
\begin{equation} \label{lemma_irr_f_iv}
\lim_{n \rightarrow \infty} \int_{B} \int \psi_{n}(y)  \, d\tilde{\pi}_{x}(y) \, d\mu(x) 
= \int_{B} \int \psilim(y) \, d \tilde{\pi}_{x}(y) \, d\mu(x).
\end{equation}
On the other hand, Fatou's lemma gives
\begin{equation} \label{lemma_irr_f_v}
\liminf_{n \rightarrow \infty} \int_{B} \int \psi_{n}(y) \, d\pisbm_{x} (y) \, d\mu(x) 
\geqslant \int_{B} \int \psilim(y)\, d \pisbm_{x}(y)\, d\mu(x).
\end{equation}
Now combining \eqref{lemma_irr_f_i} -- \eqref{lemma_irr_f_iii} with \eqref{lemma_irr_f_iv}, \eqref{lemma_irr_f_v} yields the inequality
\[
\liminf_{n \rightarrow \infty}\int \int \psi_{n} \, d(\pisbm_{x}-\tilde{\pi}_{x}) \, d\mu(x) + \int \MCov(\tilde{\pi}_{x},\gamma)\, d\mu(x)  
 \geqslant P(\mu,\nu) + \beta \, \mu(B).
\]
Recalling the definition \eqref{eq_phi_psi_x} of the function $x \mapsto \varphi^{\psi}(x)$ and the dual function \eqref{eq_rep_dual_func}, we observe that the left-hand side of this inequality is less than or equal to
\[
\liminf_{n \rightarrow \infty} \D(\psi_{n}) 
= \liminf_{n \rightarrow \infty} \int  \Big( \int \psi_{n}(y) \, \pisbm_{x}(dy)  - \varphi^{\psi_{n}}(x) \Big) \, \mu(dx).
\]
Since $(\psi_{n})_{n \geqslant 1}$ is an optimizing sequence for the dual problem, it follows that
\[
\tilde{D}(\mu,\nu) 
= \lim_{n \rightarrow \infty} \D(\psi_{n})  
= \liminf_{n \rightarrow \infty}  \D(\psi_{n}) 
\geqslant P(\mu,\nu) + \beta \, \mu(B)
> P(\mu,\nu).
\]
But this is a contradiction to the fact that there is no duality gap by Theorem \ref{theorem_no_duality_gap}.
\end{proof}
\end{lemma}

\begin{lemma} \label{lemma_ch_o_do} A lower semicontinuous convex function $\psiopt \colon \Rd \rightarrow (-\infty,+\infty]$ satisfying $\mu(\operatorname{int}(\dom \psiopt))=1$ is a dual optimizer in the sense of Definition \ref{def_dual_opt} if and only if
\begin{equation} \label{eq_ch_o_do_b}
\mu\big(\big\{ x \in \Rd \colon \psiopt \equiv \psixsbm  \textnormal{ mod (aff)}\big\}\big) = 1.
\end{equation}
\begin{proof} ``$\Leftarrow$'': As we have seen in the proof of Lemma \ref{lemma_a_h}, for $\mu$-a.e.\ $x \in \Rd$, the supremum in \eqref{eq_vph} is attained by $p = \pisbm_{x}$. Hence integrating with respect to $\mu(dx)$ yields
\[
\int \MCov(\pisbm_{x},\gamma) \, \mu(dx) 
= \int \Big( \int \psixsbm(y) \, \pisbm_{x}(dy) - \varphi^{\psixsbm}(x) \Big) \, \mu(dx).
\]
By assumption \eqref{eq_ch_o_do_b} it follows that
\begin{equation} \label{eq_ch_o_do_a}
\int \MCov(\pisbm_{x},\gamma) \, \mu(dx) 
= \int \Big( \int \psiopt(y) \, \pisbm_{x}(dy) - \varphi^{\psiopt}(x) \Big) \, \mu(dx).
\end{equation}
Recalling the representation \eqref{eq_rep_dual_func} of the dual function $\mathcal{D}(\, \cdot \,)$, we see that the expression on the right-hand side of \eqref{eq_ch_o_do_a} equals $\mathcal{D}(\psiopt)$. On the other hand, the left-hand side is equal to $P(\mu,\nu) = \tilde{D}(\mu,\nu)$ by Theorem \ref{theorem_no_duality_gap}, so that $\psiopt$ is a dual optimizer.

\smallskip

``$\Rightarrow$'': Let $\psiopt$ be a dual optimizer. As in the proof of the implication ``$\Rightarrow$'' in Theorem \ref{theo_du_op_b_m} we conclude that 
\[
\pisbm_{x} = \nabla \psiopt^{\ast}\big(\zeta(x) + \, \cdot \,\big)(\gamma),
\]
for $\mu$-a.e.\ $x \in \Rd$, with $\zeta(x) \in \Rd$ as in \eqref{eq_du_op_b_m_xi}. On the other hand, by Definition \ref{def_gsbmx} we have $\pisbm_{x} = \nabla (\psixsbm)^{\ast}(\gamma)$, for $\mu$-a.e.\ $x \in \Rd$. Arguing as in the proof of Lemma \ref{lemma_a_h}, we obtain \eqref{eq_ch_o_do_b}.
\end{proof} 
\end{lemma}

Let us summarize why the above arguments complete the proof of Proposition \ref{prop_step_2_the_sing} and thus also of Theorem \ref{theorem_single}.

\begin{proof}[Proof of Proposition \ref{prop_step_2_the_sing}] According to Lemma \ref{lemma_irr_f}, the function $\psilim \colon \Rd \rightarrow [0,+\infty]$ equals $\psixsbm$ mod (aff), for $\mu$-a.e.\ $x \in I$. By Lemma \ref{lemma_ch_o_do}, this implies that $\psilim$ is a dual optimizer. 
\end{proof}

\begin{proof}[Proof of Theorem \ref{theorem_single}] The existence of a dual optimizer follows from Proposition \ref{prop_step_1_the_sing} and Proposition \ref{prop_step_2_the_sing}.
\end{proof}

Let us have one more look at the structure of the proof of Theorem \ref{theorem_single} and Corollary \ref{thm:equivalence_kernels} above. We started with an arbitrary optimizing sequence $(\psi_{n})_{n \geqslant 1} \subseteq \Cq$ of convex functions for the dual problem \eqref{eq_dual}, which we normalized to obtain $(\psi_{n}^{x})_{n \geqslant 1}$ as in \eqref{eq_def_psi_n_x_opt_s}. Then we showed in Proposition \ref{prop_step_1_the_sing} that under the irreducibility Assumption \ref{assumptions_single_cell}, which guarantees the existence of a De March--Touzi transport $\pidmt \in \MT(\mu,\nu)$, we could find a limiting function $\psilim$. However, this was only possible \textit{after passing to a subsequence, forming Ces\`{a}ro means, and then choosing a further subsequence of $(\psi_{n}^{x})_{n \geqslant 1}$}. In Proposition \ref{prop_step_2_the_sing} we argued that $\psilim$ is a dual optimizer and hence, by Theorem \ref{theo_du_op_b_m}, there is a Bass martingale from $\mu$ to $\nu$.

\smallskip

In the follow-up paper \cite{ScTs24}, the general case of a pair $(\mu,\nu)$ which is not necessarily irreducible --- and thus does not necessarily satisfy Assumption \ref{assumptions_single_cell} --- is treated. In this case, a variant of this line of reasoning is needed. Suppose we \textit{already know} that an optimizing sequence $(\psi_{n})_{n \geqslant 1}$ is pointwise bounded on a relatively open convex set $I \subseteq \Rd$ with $\mu(I) = 1$ and such that $\widehat{\supp}(\nu)$ is contained in the closure of $I$. Under this assumption --- \textit{but without imposing irreducibility on the pair} $(\mu,\nu)$ --- we will prove in Proposition \ref{prop_1b} below that there is no need of passing to a subsequence or of forming convex combinations of the optimizing sequence $(\psi_{n})_{n \geqslant 1}$. Rather already the original sequence $(\psi_{n})_{n \geqslant 1}$ --- modulo adding affine functions --- converges. 

\begin{proposition} \label{prop_1b} Let $\mu, \nu \in \PP_{2}(\Rd)$ with $\mu \lc \nu$. Let $I \subseteq \Rd$ be a relatively open convex set with $\mu(I) = 1$. Denote by $C$ the closure of $I$ and assume that $\widehat{\supp}(\nu) \subseteq C$. Let $(\psi_{n})_{n \geqslant 1}$ be an optimizing sequence of non-negative convex functions in $\Cqaff$ for the dual problem \eqref{eq_dual} such that
\begin{equation} \label{prop_1b_ass}
\forall y \in I \colon \ \sup_{n \geqslant 1} \psi_{n}(y) < + \infty.
\end{equation}
Then there is a lower semicontinuous convex function $\psilim \colon \Rd \rightarrow [0,+\infty]$ and a sequence $(\tilde{\psi}_{n})_{n \geqslant 1}$ such that $\psi_{n} \equiv \tilde{\psi}_{n}$ mod (aff), for each $n \geqslant 1$, and
\begin{align}
\forall y \in I \colon \ \psilim(y) &= \lim_{n \rightarrow \infty} \tilde{\psi}_{n}(y) < + \infty, \label{prop_1b_i} \\
\forall y \in \Rd \setminus C \colon \ \psilim(y) &= \lim_{n \rightarrow \infty} \tilde{\psi}_{n}(y) = + \infty. \label{prop_1b_ii}
\end{align}
The convergence in \eqref{prop_1b_i} is uniform on compact subsets of $I$. Moreover, we have that $C = \widehat{\supp}(\nu)$ and $\psilim$ is a dual optimizer, which is unique modulo adding affine functions.
\end{proposition}

The proof of Proposition \ref{prop_1b} is delayed until Appendix \ref{app_prop_gen_case}. 

\begin{corollary} \label{cor_1b} Under the assumptions of Proposition \ref{prop_1b}, there exists a Bass martingale $(M_{t})_{0 \leqslant t \leqslant 1}$ from $\mu$ to $\nu$. Moreover, $\Law(M_{0},M_{1}) = \pisbm$ and, for $\mu$-a.e.\ $x \in I$, the measure $\pisbm_{x}$ is equivalent to $\nu$. In particular, $\pisbm$ is a De March--Touzi transport and the pair $(\mu,\nu)$ is irreducible.    
\begin{proof} By Proposition \ref{prop_1b}, the limiting function $\psilim$ is a dual optimizer. Therefore, by analogy with the deduction of Corollary \ref{thm:equivalence_kernels} from Theorem \ref{theorem_single}, we deduce Corollary \ref{cor_1b} from Proposition \ref{prop_1b}. To see that the pair $(\mu,\nu)$ is irreducible we refer to Theorem \ref{theo_coic}.
\end{proof}
\end{corollary}

\subsection{Proof of Theorem \texorpdfstring{\ref{MainTheorem}}{1.3}} \label{subsec_pr_int_a}

We are now in the position to prove our first main result of the introduction.

\begin{proof}[Proof of Theorem \ref{MainTheorem}]
The implication ``\ref{MainTheorem_2} $\Rightarrow$ \ref{MainTheorem_1}'' is Theorem 1.10 of \cite{BaBeHuKa20}. For the proof of ``\ref{MainTheorem_1} $\Rightarrow$ \ref{MainTheorem_2}'' we apply Corollary \ref{thm:equivalence_kernels} and obtain the existence of a Bass martingale from $\mu$ to $\nu$. By the uniqueness results of Theorem 2.2 in \cite{BaBeHuKa20}, this Bass martingale has to agree with the given stretched Brownian motion.
\end{proof}

\subsection{Proof of the second part of Theorem \texorpdfstring{\ref{theorem_new_duality}}{1.4}} \label{subsec_pr_int_b_ii}

Finally, we complete the proof of Theorem \ref{theorem_new_duality}. Recalling the results of Section \ref{sec_pr_int_b_i} and the definition \eqref{theorem_new_duality_sec_eq_s_rel_for} of $\mathcal{E}(\, \cdot \,)$, we can formulate the second part of Theorem \ref{theorem_new_duality} equivalently as follows.

\begin{proposition} \label{prop_new_duality_second_part} Let $\mu, \nu \in \PP_{2}(\Rd)$ with $\mu \lc \nu$. The value
\begin{equation} \label{eq_drel_du_opt}
D_{\textnormal{rel}}(\mu,\nu) = \inf_{\substack{\mu(\dom \psi) = 1, \\ \textnormal{$\psi$ convex}}} \mathcal{E}(\psi)
\end{equation}
is attained by a lower semicontinuous convex function $\psiopt \colon \Rd \rightarrow (-\infty,+\infty]$ satisfying $\mu(\operatorname{ri}(\dom\psiopt)) = 1$ if and only if $(\mu,\nu)$ is irreducible. In this case the (unique) optimizer to \eqref{MBMBB} is given by the Bass martingale
\[
M_{t} \coloneqq \E[\nabla v(B_{1}) \, \vert \, \sigma(B_{s} \colon s \leqslant t)]
= \E[\nabla v(B_{1}) \, \vert \, B_{t}], \qquad 0 \leqslant t \leqslant 1,
\]
where $v = \psiopt^{\ast}$ and $B_{0} \sim \nabla (\psiopt^{\ast} \ast \gamma)^{\ast}(\mu)$.
\begin{proof} We call a function $\psiopt$ as in the statement of the proposition a dual optimizer of \eqref{eq_drel_du_opt} and first show that this notion is equivalent to a dual optimizer in the sense of Definition \ref{def_dual_opt}. Indeed, if $\psiopt$ is a dual optimizer of \eqref{eq_drel_du_opt}, it follows from the inequality $(\psi^{\ast} \ast \gamma)^{\ast} = (\varphi^{\psi})^{\ast\ast} \leqslant \varphi^{\psi}$ that $\psiopt$ is also a dual optimizer according to Definition \ref{def_dual_opt}. Conversely, suppose that $\psiopt$ is an optimizer in the latter sense. Then it follows from Lemma \ref{lemma_ch_o_do} that \eqref{eq_ch_o_do_b} is satisfied. Note that by the classical Kantorovich duality we have
\[
\MCov(\pisbm_{x},\gamma) = \int \psixsbm \, d \pisbm_{x} + \int (\psixsbm)^{\ast} \, d\gamma,
\]
for $\mu$-a.e.\ $x \in \Rd$. Thus, by \eqref{eq_ch_o_do_b}, for $\mu$-a.e.\ $x \in \Rd$ there exists $\zeta(x) \in \Rd$ such that
\[
\MCov(\pisbm_{x},\gamma) = \int \big(\psiopt(\, \cdot \,) - \langle \zeta(x), \, \cdot \, \rangle\big) \, d \pisbm_{x} + \int \big(\psiopt(\, \cdot \,) - \langle \zeta(x), \, \cdot \, \rangle\big)^{\ast} \, d\gamma.
\]
Reading the proof of Proposition \ref{prop_new_duality_rel} backwards we conclude that
\[
\int \MCov(\pisbm_{x},\gamma) \, d\mu(x)
=  \int \Big( \int \psiopt(y) \, \pisbm_{x}(dy)  - (\psiopt^{\ast} \ast \gamma)^{\ast}(x)\Big) \, \mu(dx),
\]
i.e., $D_{\textnormal{rel}}(\mu,\nu) = \mathcal{E}(\psiopt)$. This shows that $\psiopt$ is also a dual optimizer of \eqref{eq_drel_du_opt}. 

\smallskip

Now that we know that both definitions of a dual optimizer are equivalent, we can conclude the assertions of Proposition \ref{prop_new_duality_second_part} from the results we have already established. Indeed, by Theorem \ref{theo_du_op_b_m}, the existence of a dual optimizer is equivalent to the existence of a Bass martingale $(M_{t})_{0 \leqslant t \leqslant 1}$ in the given form. Finally, by Theorem \ref{MainTheorem} and Theorem \ref{theo_coic}, the existence of a Bass martingale is equivalent to the irreducibility of $(\mu,\nu)$.
\end{proof}
\end{proposition}

\begin{appendix}
\section{Proof of Theorem \texorpdfstring{\ref{theorem_no_duality_gap}}{3.3}}
\label{app_proof_no_duality_gap}

\begin{proof}[Proof of Theorem \ref{theorem_no_duality_gap}] Existence and uniqueness of the optimizer $\pisbm \in \MT(\mu,\nu)$ of the primal problem \eqref{eq_primal}, as well as finiteness of the primal value $P(\mu,\nu)$ in \eqref{eq_the_no_duality_gap}, were proved in \cite[Theorem 2.2]{BaBeHuKa20}. In order to show that there is no duality gap, we apply \cite[Theorem 1.3]{BaBePa18} with the cost function
\[
C(x,p) \coloneqq
\begin{cases}
-\MCov(p,\gamma) + \tfrac{1}{2} \int \vert y \vert^{2} \, dp, & \textnormal{ if } \bary(p) = x,\\
+\infty, & \textnormal{ if } \bary(p) \neq x,
\end{cases}
\]
for $x \in \Rd$ and $p \in \PP_{2}(\Rd)$. This function is bounded from below and convex in the second argument. If we equip $\PP_{2}(\Rd)$ with the topology induced by the quadratic Wasserstein distance one verifies that $C(x,p)$ is also jointly lower semicontinuous with respect to the product topology on $\Rd \times \PP_{2}(\Rd)$. We introduce the space of continuous functions which are bounded from below and have at most quadratic growth
\[
\Cbq \coloneqq \big\{ \tilde{\psi} \colon \Rd \rightarrow \R \textnormal{ continuous s.t.\ } \exists \, a,k,\ell \in \R \textnormal{ with } \ell \leqslant \tilde{\psi}(\, \cdot \,) \leqslant a + k \vert \cdot  \vert^{2} \big\}.
\]
Then by \cite[Theorem 1.3]{BaBePa18} the value $P(\mu,\nu)$ of the primal problem \eqref{eq_primal} equals
\[
\tilde{D}_{\textnormal{b},\textnormal{q}}(\mu,\nu) \coloneqq \inf_{\tilde{\psi} \in \Cbq} \Big( \int \big( \tilde{\psi}(\, \cdot \,) + \tfrac{\vert \, \cdot \, \vert^{2}}{2}\big) \, d\nu - \int \tilde{\varphi}^{\tilde{\psi}} \, d\mu \Big),
\]
where
\[
\tilde{\varphi}^{\tilde{\psi}}(x) \coloneqq \inf_{p \in \PP_{2}^{x}(\Rd)} \Big( \int \big( \tilde{\psi}(\, \cdot \,) + \tfrac{\vert \, \cdot \, \vert^{2}}{2}\big) \, dp - \MCov(p,\gamma)\Big).
\]
Finally, passing from the functions $\tilde{\psi} \in \Cbq$ to $\psi(\, \cdot \,) \coloneqq \tilde{\psi}(\, \cdot \,) + \frac{\vert \, \cdot \, \vert^{2}}{2} \in \Cq$, we see that $\tilde{D}_{\textnormal{b},\textnormal{q}}(\mu,\nu) = \tilde{D}(\mu,\nu)$.
\end{proof}

\section{Proofs of Lemmas \texorpdfstring{\ref{lem_phi_psi_gen}}{4.4}, \texorpdfstring{\ref{lem_co_fin}}{5.1}, \texorpdfstring{\ref{lem:dim_dom_psi}}{6.2}, \texorpdfstring{\ref{lem_aux_grad_convolution}}{6.4} and \texorpdfstring{\ref{lem_int_dom_range}}{6.5}}
\label{app_tech_lemm}

\begin{proof}[Proof of Lemma \ref{lem_phi_psi_gen}] Using probabilistic notation, we rewrite the supremum in \eqref{eq_phi_psi} as
\begin{equation} \label{lem_phi_psi_gen_01}
\varrho^{\psi} = \sup \E\big[ \langle Y, Z \rangle - \psi(Y)\big],
\end{equation}
where the supremum in \eqref{lem_phi_psi_gen_01} is taken over all probability spaces such that $Z \sim \gamma$ and $Y$ is an $\Rd$-valued random variable with finite second moment. Replacing $Y$ by $\E[Y \, \vert \, Z]$, we observe that the maximization in \eqref{lem_phi_psi_gen_01} can be restricted to random variables $Y$ which are measurable functions of $Z$, and we obtain
\begin{equation} \label{lem_phi_psi_gen_02}
\varrho^{\psi} 
= \sup_{Y \in L^{2}(\gamma;\Rd)} 
\int \Big( \big\langle Y(z), z \big\rangle - \psi\big(Y(z)\big) \Big) \, \gamma(dz),
\end{equation}
where $L^{2}(\gamma;\Rd)$ denotes the space of $\Rd$-valued Borel measurable functions on $\Rd$, which are square-integrable under $\gamma$. Clearly, for any $Y \in L^{2}(\gamma;\Rd)$ we have 
\[
\int \Big( \big\langle Y(z), z \big\rangle - \psi\big(Y(z)\big) \Big) \, \gamma(dz)
\leqslant \int \sup_{y \in \Rd} \big(\langle y,z \rangle - \psi(y)\big) \, \gamma(dz) = \int \psi^{\ast}(z) \, \gamma(dz),
\]
which shows the inequality $\varrho^{\psi} \leqslant \int \psi^{\ast}  \, d\gamma$. In order to see the reverse inequality, we define the auxiliary problem
\begin{equation} \label{lem_phi_psi_gen_03}
\varrho_{\infty}^{\psi} 
\coloneqq \sup_{Y \in L^{\infty}(\gamma;\Rd)} 
\int \Big( \big\langle Y(z), z \big\rangle - \psi\big(Y(z)\big) \Big) \, \gamma(dz),
\end{equation}
where $L^{\infty}(\gamma;\Rd)$ denotes the space of $\Rd$-valued Borel measurable functions on $\Rd$, which are bounded $\gamma$-a.e. Comparing \eqref{lem_phi_psi_gen_02} with \eqref{lem_phi_psi_gen_03}, we obviously have $\varrho^{\psi} \geqslant \varrho_{\infty}^{\psi}$. Now we claim that
\begin{equation} \label{lem_phi_psi_gen_05}
\varrho_{\infty}^{\psi} 
\geqslant \int \sup_{y \in \Rd} \big(\langle y,z \rangle - \psi(y)\big) \, \gamma(dz) 
= \int \psi^{\ast}(z) \, \gamma(dz),
\end{equation}
which will finish the proof of \eqref{lem_phi_psi_gen_04}. To see this, we first write
\[
\varrho_{\infty}^{\psi} 
= \lim_{N \rightarrow \infty} \, \sup_{\substack{Y \in L^{\infty}(\gamma;\Rd), \\ \vert Y  \vert \leqslant N}} \, 
\int \Big( \big\langle Y(z), z \big\rangle - \psi\big(Y(z)\big) \Big) \, \gamma(dz).
\]
Using a measurable selection argument, we obtain
\[
\sup_{\substack{Y \in L^{\infty}(\gamma;\Rd), \\ \vert Y  \vert \leqslant N}} \, 
\int \Big( \big\langle Y(z), z \big\rangle - \psi\big(Y(z)\big) \Big) \, \gamma(dz) 
\geqslant
\int \sup_{\substack{y \in \Rd, \\ \vert y  \vert \leqslant N}} \,  \big(\langle y,z \rangle - \psi(y)\big) \, \gamma(dz).
\]
Since $\psi$ is proper we can choose $y_{0} \in \dom \psi \neq \varnothing$. Then for $N$ large enough we have
\[
\sup_{\substack{y \in \Rd, \\ \vert y  \vert \leqslant N}} \,  \big(\langle y,z \rangle - \psi(y)\big) 
\geqslant \langle y_{0},z \rangle - \psi(y_{0}),
\]
with the right-hand side being integrable with respect to $\gamma(dz)$. Hence we can apply the monotone convergence theorem and deduce that
\[
\lim_{N \rightarrow \infty} \int \sup_{\substack{y \in \Rd, \\ \vert y  \vert \leqslant N}} \,  \big(\langle y,z \rangle - \psi(y)\big) \, \gamma(dz)
= \int \sup_{y \in \Rd} \big(\langle y,z \rangle - \psi(y)\big) \, \gamma(dz),
\]
which completes the proof of \eqref{lem_phi_psi_gen_05}.
\end{proof}

\begin{proof}[Proof of Lemma \ref{lem_co_fin}] To see that $\varphi^{\psi}$ is convex on $\dom\psi$, we let $x \coloneqq c x_{1}+(1-c)x_{2}$ for $x_{1},x_{2} \in \dom\psi$ and $c\in(0,1)$. If $\varphi^{\psi}(x_{1}) = -\infty$, then there are $p_{x_{1}}^{(n)} \in \PP_{2}^{x_{1}}(\Rd)$ with $\int \psi \, dp_{x_{1}}^{(n)} - \MCov(p_{x_{1}}^{(n)},\gamma) \rightarrow -\infty$. We observe that $p_{x}^{(n)} \coloneqq c p_{x_{1}}^{(n)} + (1-c)\delta_{x_{2}}\in  \PP_{2}^{x}(\Rd)$, and consequently
\begin{align*}
\varphi^{\psi}(x)
&\leqslant c \int \psi \, dp_{x_{1}}^{(n)} + (1-c) \psi(x_{2}) - \MCov(p_{x}^{(n)},\gamma) \\
&\leqslant c \Big( \int \psi \, dp_{x_{1}}^{(n)}  - \MCov(p_{x_{1}}^{(n)},\gamma) \Big) 
+(1-c)\psi(x_{2}),
\end{align*}
where we have used the convexity of $\PP_{2}(\Rd) \ni p \mapsto - \MCov(p,\gamma)$. We conclude that  $\varphi^{\psi}(x) = -\infty$. The case $\varphi^{\psi}(x_{2}) = -\infty$ is treated similarly. If, on the other hand, both $\varphi^{\psi}(x_{1}) > - \infty$ and $\varphi^{\psi}(x_{2}) > -\infty$, then $   \varphi^{\psi}(x)\leqslant c \varphi^{\psi}(x_{1}) + (1-c)\varphi^{\psi}(x_{2})$ follows by standard arguments. All in all we see that $\varphi^{\psi}$ is convex on $\dom \psi$. 

\smallskip

If $\varphi^{\psi}(x) > - \infty$ for one $x \in \operatorname{int} (\dom \psi)$, then $\varphi^{\psi}(\tilde{x}) > - \infty$ for all $\tilde{x} \in \operatorname{int}(\dom \psi)$, as can be seen directly by convexity. 

\smallskip

Finally, given some $x \in \operatorname{int} (\dom \psi)$ with $\varphi^{\psi}(x) > - \infty$, we show that $\psi$ is co-finite. Without loss of generality, we assume that $x = 0$ and $\psi(0) = 0$. Using probabilistic notation, we rewrite the supremum in \eqref{eq_phi_psi_x_max} as
\begin{equation} \label{lem_co_fin_01}
- \varphi^{\psi}(0) = \sup \E\big[ \langle Y, Z \rangle - \psi(Y)\big] < + \infty,
\end{equation}
where the supremum in \eqref{lem_co_fin_01} is taken over all probability spaces such that $Z \sim \gamma$ and $Y$ is an $\Rd$-valued random variable with finite second moment and $\E[Y] = 0$. By contradiction, suppose there is $e \in \Rd \setminus \{ 0 \}$ such that
\begin{equation} \label{lem_co_fin_02}
\lim_{t \rightarrow +\infty} \frac{\psi(t e)}{t} < +\infty.
\end{equation}
We define the convex function $\bar{\psi} \colon \R \rightarrow (-\infty,+\infty]$ by $\bar{\psi}(t) \coloneqq \psi(te)$, for $t \in \R$. By assumption we have $0 \in \operatorname{int} (\dom\psi)$, thus also $0 \in \operatorname{int} (\dom\bar{\psi})$. In particular, $\bar{\psi}$ is continuous in a neighbourhood of $0$. By \eqref{lem_co_fin_02}, there is a constant $K$ such that $\bar{\psi}(t)\leqslant Kt$, for all $t$ large enough. As $\bar{\psi}(0) = 0$ and $\bar{\psi}$ is convex, we conclude that $\bar{\psi}(t)\leqslant Kt$, for all $t \geqslant 0$. Next we introduce four parameters, $A<0$ and $\delta,M,C>0$, and define the $\Rd$-valued function
\[
\Rd \ni z \longmapsto f_{\delta,A}^{M,C}(z) \coloneqq
\begin{cases}
M e, & \textnormal{ if } \langle e,z\rangle \geqslant C,\\
-\delta e, & \textnormal{ if } \langle e,z\rangle \in (A,0),\\
0, & \textnormal{ else}.
\end{cases}
\]
Given $Z \sim \gamma$, this induces a random variable $Y_{\delta,A}^{M,C} \coloneqq f_{\delta,A}^{M,C}(Z)$. We impose the relation 
\begin{equation} \label{lem_co_fin_03}
M \, \mathbb{P}\big[ \langle e,Z\rangle \geqslant C\big] 
= \delta \, \mathbb{P} \big[\langle e,Z\rangle \in (A,0)\big]
\end{equation}
between $M$ and $\delta$, so that $\E[Y_{\delta,A}^{M,C}]=0$. Now we fix $A<0$ and some $\varepsilon >0$. Then we choose $\delta > 0$ small enough such that 
\[
\vert \bar{\psi}(-\delta) \vert \vee \delta \, \E\big[\vert \langle e,Z\rangle \vert\big] \leqslant \tfrac{\varepsilon}{2},
\] 
which we achieve by continuity of $\bar{\psi}$ and the fact that $\bar{\psi}(0)=0$. We leave $C$ as a free parameter, which fixes $M$ via \eqref{lem_co_fin_03}. We compute
\begin{align*}
& \, \E\big[\langle Y_{\delta,A}^{M,C},Z \rangle - \psi(Y_{\delta,A}^{M,C}) \big] \\
= & \, \E\big[  
\boldsymbol{1}_{\langle e,Z\rangle \geqslant C} \big( M \langle e,Z\rangle - \bar{\psi}(M)\big) 
+ \boldsymbol{1}_{\langle e,Z\rangle \in (A,0)} \big(-\delta \langle e,Z\rangle - \bar{\psi}(-\delta)\big)
\big] \\
\geqslant & \, M \, \E\big[  
\boldsymbol{1}_{\langle e,Z\rangle \geqslant C} \big( \langle e,Z\rangle - K\big)\big] 
-\delta \, \E\big[\boldsymbol{1}_{\langle e,Z\rangle \in (A,0)}  \langle e,Z\rangle \big]
- \bar{\psi}(-\delta) \, \mathbb{P}\big[\langle e,Z\rangle \in (A,0)\big] \\
\geqslant & \, M (C-K) \, \mathbb{P}\big[\langle e,Z\rangle \geqslant C\big] - \varepsilon \\
= & \, \delta (C-K) \, \mathbb{P} \big[\langle e,Z\rangle \in (A,0)\big] - \varepsilon.
\end{align*}
Now taking $C \nearrow + \infty$ we conclude that $- \varphi^{\psi}(0) =+\infty$, which is a contradiction to \eqref{lem_co_fin_01}.
\end{proof}

\begin{proof}[Proof of Lemma \ref{lem:dim_dom_psi}] By contradiction, we suppose that 
\[
\dim (\dom\psiopt) < \dim (\supp\nu) = d. 
\]
Defining $B_{\ell} \coloneqq \{x \in \Rd \colon \psiopt(x) \leqslant \ell\}$ for $\ell \geqslant 1$, we clearly have $\mu(B_{\ell}) \nearrow 1$ and furthermore 
\[
\sup_{\ell \geqslant 1} \ \pisbm\big( B_{\ell} \times (\Rd \setminus \dom\psiopt) \big) >0.
\]
Indeed, otherwise $\nu$ would be concentrated on $\dom\psiopt$, contradicting the assumption that $\dim (\dom\psiopt) < \dim (\supp\nu)$. As a consequence, 
\[
\int_{B_{\ell} \times \Rd} \psiopt(y) \, \pisbm(dx,dy) = +\infty
\]
for all $\ell \geqslant 1$ large enough. We now show that $\psiopt$ could not have been optimal, by establishing that
\[
\D(\psiopt) = \int  \bigg( \sup_{p \in \PP_{2}^{x}(\Rd)} \Big(\MCov(p,\gamma)- \int \psiopt \, dp \Big) + \int \psiopt(y) \, \pisbm_{x}(dy)  \bigg) \, \mu(dx) = + \infty.
\]
To wit, selecting $p_{x} = \pisbm_{x}$ for $x \in \Rd \setminus B_{\ell}$ and $p_{x} = \delta_{x}$ for $x \in B_{\ell}$ we find
\begin{align*}
\D(\psiopt) 
&\geqslant \int_{\Rd \setminus B_{\ell}} \MCov(\pisbm_{x},\gamma) \, \mu(dx) 
- \int_{B_{\ell}}\Big(\psiopt(x) - \int \psiopt(y) \, \pisbm_{x}(dy) \Big) \, \mu(dx)\\
&\geqslant c - \ell + \int_{B_{\ell} \times \Rd} \psiopt(y) \, \pisbm(dx,dy),
\end{align*}
for a finite constant $c$. We conclude by taking $\ell \geqslant 1$ large enough.
\end{proof}

\begin{proof}[Proof of Lemma \ref{lem_aux_grad_convolution}] We first prove the identity \ref{lem_aux_grad_convolution_i}. We denote $F \coloneqq \nabla f$, which is well defined Lebesgue-a.e. For $y,\eta \in \Rd$ we have to compute
\[
\frac{(f \ast \gamma)(y+h\eta) -(f \ast \gamma)(y)}{h}
 = \int \frac{f(y+h\eta+z) -f(y+z) }{h} \, d\gamma(z),
\]
as $h \rightarrow 0$. For all $y,\eta$ and for $\gamma$-a.e.\ $z \in \Rd$, we have
\[
\lim_{h \rightarrow 0} \frac{f(y+h\eta+z) - f(y+z)}{h} 
= \langle F(y+z),\eta\rangle,
\]
so we only need to justify the exchange of limit and integral. To this end, it suffices to show the uniform integrability, with respect to the reference measure $\gamma$, of the family 
\begin{equation} \label{eq_uifu}
\mathscr{U} \coloneqq 
\bigg\{ \frac{f(y+h\eta+ \, \cdot \, ) -f(y+ \, \cdot \, )}{h} \colon 0 \leqslant h \leqslant 1 \bigg\}.
\end{equation}
Using twice the above-tangent characterization of convexity, we obtain
\begin{equation} \label{eq:grad_bounds}
\langle F(y+z),\eta\rangle
\leqslant \frac{f(y+h\eta+z) -f(y+z) }{h}
\leqslant \langle F(y+h\eta+z),\eta\rangle.
\end{equation}
We take $r \in (1,2)$, $p \coloneqq 2/r$, and $q \coloneqq p/(p-1)$ its H\"older conjugate, and compute
\begin{align}
&\int \vert F(y+h\eta+z)\vert^{r} \, d\gamma(z) 
= \int \vert F(\zeta+z)\vert^{r} \, d\gamma_{y+h\eta-\zeta}(z) \nonumber \\
=\, & \int \vert F(\zeta+z) \vert^{r} \, 
\exp\big(-\vert y-\zeta+h\eta \vert^{2}/2+\langle z ,y-\zeta+h\eta \rangle\big) \, d\gamma(z) \nonumber \\
\leqslant \, & \sqrt[p]{\Vert F \Vert^{2}_{L^{2}(\gamma_{\zeta};\Rd)}} \,
\sqrt[q]{\int \exp\big(-q\vert y-\zeta+h\eta \vert^{2}/2+q\langle z ,y-\zeta+h\eta \rangle\big)} \, d\gamma(z) \nonumber \\
\leqslant \, & \sqrt[p]{\Vert F \Vert^{2}_{L^{2}(\gamma_{\zeta};\Rd)}} \, 
\sqrt[q]{\int \exp\big(q\langle z ,y-\zeta+h\eta \rangle\big)  \, d\gamma(z)} \nonumber \\
= \, & \sqrt[p]{\Vert F \Vert^{2}_{L^{2}(\gamma_{\zeta};\Rd)}} \,
\sqrt[q]{ \exp\big(q^{2}\vert y-\zeta+h\eta \vert^{2}/2\big)}. \label{a_b_ui_f}
\end{align}
We see that the first factor in \eqref{a_b_ui_f} is finite (and independent of $h$) by assumption, while the second factor is uniformly bounded for $0 \leqslant h \leqslant 1$. Hence both the upper and lower bounds in \eqref{eq:grad_bounds} are uniformly integrable for $0 \leqslant h \leqslant 1$, implying the same property for the family $\mathscr{U}$ of \eqref{eq_uifu}.

\smallskip

We now turn to \ref{lem_aux_grad_convolution_ii} and \ref{lem_aux_grad_convolution_iii}. The arguments given at the end of the proof of Proposition \ref{prop_phi_psi_x_gen} show that, under the assumptions made here (but with $f$ instead of $\psi^{\ast}$), the function $f\ast\gamma$ is strictly convex. We already know that this function is differentiable, and that the interior of its domain is $\Rd$. Hence by \cite[Theorem 26.5]{Ro70} we deduce the stated properties. The only thing that merits an explanation is the inclusion $\operatorname{int} (\dom (f\ast\gamma)^{\ast})\supseteq \operatorname{int} (\dom f^\ast)$. To this end, observe that 
\begin{align*}
(f\ast\gamma)^{\ast}(x)
&= \sup_{\zeta \in \Rd} \Big\{\langle \zeta,x\rangle - \int f(\zeta+z) \, d\gamma(z)\Big\} \leqslant \int \sup_{\zeta \in \Rd} \{\langle \zeta,x\rangle -f(\zeta+z) \} \, d\gamma(z) \\
&= \int \sup_{\zeta \in \Rd} \big\{\langle \zeta-z,x\rangle -f(\zeta) \big\} \, d\gamma(z) 
= \int \big( f^{\ast}(x) - \langle z,x\rangle \big) \, d\gamma(z) 
= f^{\ast}(x),
\end{align*}
so $\dom (f\ast\gamma)^{\ast} \supseteq \dom f^{\ast}$.
\end{proof}

\begin{proof}[Proof of Lemma \ref{lem_int_dom_range}] The first equality in \eqref{eq_m_lem_int_dom_range} follows from 
\[
\operatorname{int}(\dom \psi)\subseteq \dom \partial \psi \subseteq \dom \psi
\]
and $\dom \partial \psi = (\partial \psi^{\ast})(\Rd)$. We also note that
\begin{equation} \label{eq_m_lem_int_dom_range_01} 
\operatorname{int} \big( (\partial \psi^{\ast})(\Rd) \big) 
= \operatorname{int} \conv \big( (\partial \psi^{\ast})(\Rd) \big),
\end{equation}
where $\conv$ denotes the convex hull. Indeed, as a consequence of a result by Mirsky \cite{Mi61}, we have that $\operatorname{int} \conv ( (\partial \psi^{\ast})(\Rd) ) \subseteq (\partial \psi^{\ast})(\Rd)$, since the subdifferential (multi-valued) mapping of any lower semicontinuous proper convex function is also a maximal monotone mapping by \cite[Corollary 31.5.2]{Ro70}. 

\smallskip

Let us turn to the proof of the second equality in \eqref{eq_m_lem_int_dom_range}. Without loss of generality we can assume that $\eta = 0$ and $t = 1$. Since we already verified the first equality and by virtue of \eqref{eq_m_lem_int_dom_range_01}, it suffices to show the inclusions 
\begin{equation} \label{eq_m_lem_int_dom_range_fi} 
(\nabla \psi^{\ast} \ast \gamma)(\Rd) \subseteq \operatorname{int} \conv ( (\partial \psi^{\ast})(\Rd))
\end{equation}
and
\begin{equation} \label{eq_m_lem_int_dom_range_si} 
\operatorname{int}(\dom \psi) \subseteq (\nabla \psi^{\ast} \ast \gamma)(\Rd).
\end{equation}

Note that we have $(\nabla \psi^{\ast} \ast \gamma)(\Rd) \subseteq \overline{\conv} ( (\partial \psi^{\ast})(\Rd))$ and therefore
\begin{equation} \label{eq_m_lem_int_dom_range_si_aa} 
\operatorname{int} (\nabla \psi^{\ast} \ast \gamma)(\Rd) 
\subseteq \operatorname{int} \overline{\conv} ( (\partial \psi^{\ast})(\Rd))
= \operatorname{int} \conv ( (\partial \psi^{\ast})(\Rd)).
\end{equation}
Under the assumptions of Lemma \ref{lem_int_dom_range}, the function $\psi$ is co-finite by Lemma \ref{lem_co_fin} and therefore we can apply Lemma \ref{lem_aux_grad_convolution} to its convex conjugate $f = \psi^{\ast}$. In particular, by \ref{lem_aux_grad_convolution_iii}, the set $(\nabla \psi^{\ast} \ast \gamma)(\Rd)$ is open and hence the inclusion \eqref{eq_m_lem_int_dom_range_fi} follows from \eqref{eq_m_lem_int_dom_range_si_aa}.

In order to prove \eqref{eq_m_lem_int_dom_range_si}, we first observe that $\operatorname{int} (\dom \psi) \subseteq \operatorname{int}(\dom(\psi^{\ast}\ast\gamma)^{\ast})$ by Lemma \ref{lem_aux_grad_convolution}, \ref{lem_aux_grad_convolution_ii}. Since $\operatorname{int} (\dom (\psi^{\ast}\ast\gamma)^{\ast}) = (\nabla \psi^{\ast} \ast\gamma)(\Rd)$ by Lemma \ref{lem_aux_grad_convolution}, \ref{lem_aux_grad_convolution_iii}, the inclusion \eqref{eq_m_lem_int_dom_range_si} follows.
\end{proof}

\section{Proof of Proposition \texorpdfstring{\ref{prop_1b}}{7.20}} \label{app_prop_gen_case}

\begin{lemma} \label{lemma_1b} For every sequence $(\psi_{n})_{n \geqslant 1}$ of lower semicontinuous convex functions $\psi_{n} \colon \Rd \rightarrow [0,+\infty)$ there is a closed convex set $\Clim \subseteq \Rd$ with relative interior $\Ilim$ and a subsequence $(\psi_{n_{k}})_{k \geqslant 1}$ such that the limits
\begin{alignat}{2}
\psilim(y) &\coloneqq \lim_{k \rightarrow \infty} \psi_{n_{k}}(y) < + \infty,
\qquad &&y \in \Ilim \label{lemma_1b_i} \\
\psilim(y) &\coloneqq  \lim_{k \rightarrow \infty} \psi_{n_{k}}(y) = + \infty, 
\qquad &&y \in \Rd \setminus \Clim \label{lemma_1b_ii}
\end{alignat}
exist. Moreover, the convergence in \eqref{lemma_1b_i} is uniform on compact subsets of $\Ilim$.
\begin{proof} \textit{Step 1.} We denote by $\mathcal{N}$ the collection of increasing subsequences $\mathscr{N} = (n_{k})_{k \geqslant 1}$ of $\mathbb{N}$. For two subsequences $\mathscr{N}, \mathscr{N}^{\prime} \in \mathcal{N}$ we call $\mathscr{N}^{\prime}$ \textit{finer} than $\mathscr{N}$, denoted by $\mathscr{N}^{\prime} \succcurlyeq \mathscr{N}$, if $\mathscr{N}^{\prime} \subseteq \mathscr{N}$, up to finitely many elements.

\smallskip

We write $\mathcal{I}$ for the collection of relatively open convex sets $I \subseteq \Rd$. For two pairs $(I,\mathscr{N}), (I^{\prime},\mathscr{N}^{\prime}) \in \mathcal{I} \times \mathcal{N}$ we say that $(I^{\prime},\mathscr{N}^{\prime})$ is \textit{finer} than $(I,\mathscr{N})$, again denoted by $(I^{\prime},\mathscr{N}^{\prime}) \succcurlyeq (I,\mathscr{N})$, if $I^{\prime} \supsetneq I$ and $\mathscr{N}^{\prime} \succcurlyeq \mathscr{N}$.

\smallskip

We call a pair $(I,\mathscr{N}) \in \mathcal{I} \times \mathcal{N}$ \textit{admissible} if it satisfies 
\begin{equation} \label{eq_def_admissible}
\forall y \in I \colon \ \sup_{n_{k} \in \mathscr{N}} \psi_{n_{k}}(y) < + \infty.
\end{equation}

\medskip 

\noindent \textit{Step 2.} Let $(I,\mathscr{N})$ be admissible. Write $\mathscr{N} = (n_{k})_{k \geqslant 1} \in \mathcal{N}$ and denote by $C$ the closure of $I$. Then one of the following two alternatives hold: either
\begin{equation} \label{eq_def_maximal}
\lim_{k \rightarrow \infty} \psi_{n_{k}}(y) = + \infty, 
\qquad y \in \Rd \setminus C, 
\end{equation}
in which case we call $(I,\mathscr{N})$ \textit{maximal}; or there is some $y_{0} \in \Rd \setminus C$ and a subsequence $\mathscr{N}^{\prime} = (n_{k}^{\prime})_{k \geqslant 1} \in \mathcal{N}$ finer than $\mathscr{N}$ such that
\[
\sup_{n_{k}^{\prime} \in \mathscr{N}^{\prime}} \psi_{n_{k}^{\prime}}(y_{0}) < + \infty.
\]
Now we define the relatively open convex set $I^{\prime} \coloneqq \operatorname{ri} \conv(y_{0},I)$ and note that $I^{\prime} \supsetneq I$. Then the pair $(I^{\prime},\mathscr{N}^{\prime})$ is admissible
and finer than $(I,\mathscr{N})$. In this second case we say that $(I,\mathscr{N})$ is \textit{refinable}.

\medskip 

\noindent \textit{Step 3.} Assume that $(I,\mathscr{N})$ is maximal, so that both \eqref{eq_def_admissible} and \eqref{eq_def_maximal} are satisfied. In particular, the sequence $(\psi_{n_{k}})_{k \geqslant 1}$, with $\mathscr{N} = (n_{k})_{k \geqslant 1}$, is pointwise bounded on $I$. By \cite[Theorem 10.9]{Ro70} we can select a further subsequence, still denoted by $(\psi_{n_{k}})_{k \geqslant 1}$, such that \eqref{lemma_1b_i} holds for all $y \in \Ilim \coloneqq I$ and the convergence is uniform on compact subsets of $\Ilim$. Furthermore, condition \eqref{lemma_1b_ii} is still satisfied for this further subsequence thanks to \eqref{eq_def_maximal}.

\medskip 

\noindent \textit{Step 4.} In light of Step 3, in order to prove Lemma \ref{lemma_1b}, it suffices to show the existence of a maximal pair $(I,\mathscr{N})$, which we will denote by $(\Ilim,\Nlim)$. The construction of the maximal pair $(\Ilim,\Nlim)$ is done via transfinite recursion along the countable ordinal numbers $\alpha < \omega_{1}$, where $\omega_{1}$ is the first uncountable ordinal. 

\smallskip

We start the construction as follows. If $\sup_{n \in \mathbb{N}} \psi_{n}(y) = + \infty$ holds for all $y \in \Rd$, the pair $(I_{0},\mathscr{N}_{0}) = (\varnothing,\mathbb{N})$ is admissible. Otherwise, we can find some $y_{0} \in \Rd$ such that $\sup_{n \in \mathbb{N}} \psi_{n}(y_{0}) < + \infty$ and then the pair $(I_{0},\mathscr{N}_{0}) = (\{y_{0}\},\mathbb{N})$ is admissible.

\smallskip

Now let $\alpha < \omega_{1}$ be an ordinal number and suppose that the transfinite family $(I_{\beta},\mathscr{N}_{\beta})_{\beta < \alpha}$ in $\mathcal{I} \times \mathcal{N}$ is such that 
\begin{enumerate}[label=(\roman*)] 
\item \label{tfi_i} the pair $(I_{\beta},\mathscr{N}_{\beta})$ is admissible, for all $\beta < \alpha$, 
\item \label{tfi_ii} the pair $(I_{\beta_{2}},\mathscr{N}_{\beta_{2}})$ is finer than $(I_{\beta_{1}},\mathscr{N}_{\beta_{1}})$, for all $\beta_{1} < \beta_{2} < \alpha$.
\end{enumerate}

\smallskip

We first assume that $\alpha$ is a successor ordinal, i.e., there is an ordinal $\beta$ such that $\alpha = \beta + 1$. Since $(I_{\beta},\mathscr{N}_{\beta})$ is admissible, by Step 2 only two cases are possible. If $(I_{\beta},\mathscr{N}_{\beta})$ is maximal, we set $(\Ilim,\Nlim) \coloneqq (I_{\beta},\mathscr{N}_{\beta})$ and finish the construction. If $(I_{\beta},\mathscr{N}_{\beta})$ is refinable, we can find a pair $(I_{\alpha},\mathscr{N}_{\alpha})$ which is finer than $(I_{\beta},\mathscr{N}_{\beta})$.

\smallskip

Suppose that $\alpha$ is a limit ordinal, i.e., it can be written as $\alpha = \{ \beta \colon \beta < \alpha\}$. Since $\alpha < \omega_{1}$, the ordinal $\alpha$ is countable and can be enumerated as $\alpha = \{ \alpha_{k} \colon k < \omega\}$, where $\omega$ is the first infinite ordinal. Letting $\beta_{0} \coloneqq 0$ and defining inductively $\beta_{m+1} \coloneqq \alpha_{\ell}$, where $\ell \coloneqq \min \{k \colon \beta_{m} < \alpha_{k}\}$, we obtain a strictly increasing cofinal sequence $(\beta_{m})_{m < \omega}$ in $\alpha$. Now we define $I_{\alpha} \coloneqq \bigcup_{m < \omega} I_{\beta_{m}}$. By \ref{tfi_ii}, $(I_{\beta_{m}})_{m < \omega}$ is a strictly increasing sequence of relatively open convex sets, therefore also $I_{\alpha} \in \mathcal{I}$. Again by  \ref{tfi_ii}, the sequence $(\mathscr{N}_{\beta_{m}})_{m < \omega}$ satisfies $\mathscr{N}_{\beta_{m_{1}}} \succcurlyeq \mathscr{N}_{\beta_{m_{2}}}$, for all $m_{1} < m_{2} < \omega$. Thus we can find a diagonal subsequence $\mathscr{N}_{\alpha} \in \mathcal{N}$ of $(\mathscr{N}_{\beta_{m}})_{m < \omega}$, meaning that $\mathscr{N}_{\alpha} \succcurlyeq \mathscr{N}_{\beta_{m}}$, for all $m < \omega$. Since by \ref{tfi_i} the pair $(I_{\beta_{m}},\mathscr{N}_{\beta_{m}})$ is admissible, for all $m < \omega$, by construction also $(I_{\alpha},\mathscr{N}_{\alpha})$ is admissible. Furthermore, $(I_{\alpha},\mathscr{N}_{\alpha})$ satisfies $(I_{\alpha},\mathscr{N}_{\alpha}) \succcurlyeq (I_{\beta_{m}},\mathscr{N}_{\beta_{m}})$, for all $m < \omega$. Since $(\beta_{m})_{m < \omega}$ is cofinal, it follows that $(I_{\alpha},\mathscr{N}_{\alpha}) \succcurlyeq (I_{\beta},\mathscr{N}_{\beta})$, for all $\beta < \alpha$.

\smallskip

After countably many steps, this transfinite recursion must stop, which happens when $\alpha < \omega_{1}$ is such that $(I_{\alpha},\mathscr{N}_{\alpha})$ is maximal and we can then set $(\Ilim,\Nlim) \coloneqq (I_{\alpha},\mathscr{N}_{\alpha})$. Indeed, otherwise we would have a transfinite sequence $(I_{\beta})_{\beta < \omega_{1}}$ of strictly increasing relatively open convex sets, indexed by the countable ordinals, which is impossible. 
\end{proof}
\end{lemma}

\begin{proof}[Proof of Proposition \ref{prop_1b}:] \

\noindent \textit{Step 1.} By Lemma \ref{lemma_1b}, there is a closed convex set $\Clim \subseteq \Rd$ with relative interior $\Ilim$ and a subsequence $(\psi_{n_{k}})_{k \geqslant 1}$ such that the limits
\begin{alignat}{2}
\psilim(y) &\coloneqq \lim_{k \rightarrow \infty} \psi_{n_{k}}(y) < + \infty,
\qquad &&y \in \Ilim \label{lemma_1b_i_apl} \\
\psilim(y) &\coloneqq  \lim_{k \rightarrow \infty} \psi_{n_{k}}(y) = + \infty, 
\qquad &&y \in \Rd \setminus \Clim \label{lemma_1b_ii_apl}
\end{alignat}
exist, where the convergence in \eqref{lemma_1b_i_apl} is uniform on compact subsets of $\Ilim$. Via \eqref{lemma_1b_i_apl} and \eqref{lemma_1b_ii_apl} we define a lower semicontinuous convex function $\psilim \colon \Rd \rightarrow [0,+\infty]$ on all of $\Rd$.

\medskip 

\noindent \textit{Step 2.} By the hypothesis \eqref{prop_1b_ass} and the construction of $\Ilim$, we clearly have $I \subseteq \Ilim$. Our goal is to show that $\Ilim = I$. To this end, we define the set $A \subseteq \Rd$ as in \eqref{def_set_a}, now with $I$ replaced by $\Ilim$, i.e.,
\[
A = \big\{ x \in \Ilim \colon \psilim \not\equiv \psixsbm  \textnormal{ mod (aff)}\big\}.
\]
Again, we will show that $\mu(A) = 0$. Repeating the reasoning of Lemma \ref{lemma_a_h} we obtain, for $\mu$-a.e.\ $x \in A$, measures $\check{\pi}_{x} \in \PP_{2}^{x}(\Rd)$ satisfying \eqref{lemma_a_h_in}, which now are supported by $\Clim$. Then we choose an increasing sequence $(K_{j})_{j \geqslant 1}$ of relatively compact subsets of $\Ilim$ such that $\bigcup_{j \geqslant 1}K_{j} = \Ilim$. As in Lemma \ref{lemma_irr_f_p}, for $\mu$-a.e.\ $x \in A$, we find measures $\check \pi_{x}^{j(x)} \in \PP_{2}^{x}(\Rd)$ satisfying \eqref{lemma_a_h_in_se} and supported by $K_{j(x)}$, for some $j(x) \in \mathbb{N}$. Arguing as in Lemma \ref{lemma_irr_f} we conclude that $\mu(A) = 0$, i.e., $\psilim \equiv \psixsbm$ mod (aff), for $\mu$-a.e.\ $x \in \Ilim$. In particular, since $\Ilim = \operatorname{ri} (\dom \psilim)$ and $\operatorname{ri}(\dom \psixsbm) = \operatorname{ri} \widehat{\supp}(\nu)$, we conclude from the assumption $\widehat{\supp}(\nu) \subseteq C$ that $\Ilim = I$ and consequently $C = \widehat{\supp}(\nu)$. From Lemma \ref{lemma_ch_o_do}, we conclude that the limiting function $\psilim$ is a dual optimizer. Together with Step 1, this proves the statement of Proposition \ref{prop_1b} for the subsequence $(\psi_{n_{k}})_{k \geqslant 1}$.


\medskip 

\noindent \textit{Step 3.} Applying Step 1 and Step 2 to an arbitrary subsequence $(\psi_{m_{\ell}})_{\ell \geqslant 1}$ instead of the original sequence $(\psi_{n})_{n \geqslant 1}$, we obtain a further subsequence $(\psi_{m_{\ell_{j}}})_{j \geqslant 1}$ and a lower semicontinuous convex function $\tpsilim \colon \Rd \rightarrow [0,+\infty]$ such that 
\begin{align*}
\forall y \in I \colon \ \tpsilim(y) &= \lim_{j \rightarrow \infty} \psi_{m_{\ell_{j}}}(y) < + \infty, \\
\forall y \in \Rd \setminus C \colon \ \tpsilim(y) &= \lim_{j \rightarrow \infty} \psi_{m_{\ell_{j}}}(y) = + \infty, 
\end{align*}
and $\tpsilim \equiv \psixsbm$ mod (aff), for $\mu$-a.e.\ $x \in I$. We conclude that there is an affine function $\aff$ depending on the subsequence $({m_{\ell_{j}}})_{j \geqslant 1}$ such that $\psilim = \tpsilim + \aff$. In particular,
\begin{equation} \label{eq_step_3_fin}
\forall y \in I \cup (\Rd \setminus C) \colon \ 
\lim_{j \rightarrow \infty} \big( \psi_{m_{\ell_{j}}}(y) + \aff(y) \big) = \psilim(y).
\end{equation}

\medskip 

\noindent \textit{Step 4.} Let $x_{0}, \ldots, x_{d}$ be affinely independent points in $I \cup (\Rd \setminus C)$. For each $n \geqslant 1$ there is a unique affine function $\aff_{n}$ such that
\begin{equation} \label{eq_st4_i}
\forall i \in \{0, \ldots, d\} \colon \ \psi_{n}(x_{i}) + \aff_{n}(x_{i}) = \psilim(x_{i}).
\end{equation}
We claim that
\begin{equation} \label{eq_st4_ii}
\forall y \in I \cup (\Rd \setminus C) \colon \ 
\lim_{n \rightarrow \infty} \big( \psi_{n}(y) + \aff_{n}(y) \big) = \psilim(y).
\end{equation}
Indeed, if this were not the case, there would be some $y_{0} \in I \cup (\Rd \setminus C)$ and a subsequence $(m_{\ell})_{\ell \geqslant 1}$ such that
\begin{equation} \label{eq_st4_iii}
\lim_{\ell \rightarrow \infty} \big( \psi_{m_{\ell}}(y_{0}) + \aff_{m_{\ell}}(y_{0}) \big) \neq \psilim(y_{0}).
\end{equation}
By Step 3 there is a further subsequence $(m_{\ell_{j}})_{j \geqslant 1}$ and an affine function $\aff$ such that \eqref{eq_step_3_fin} holds. In particular,
\begin{align} 
\forall i \in \{0, \ldots, d\} \colon \ 
\lim_{j \rightarrow \infty} \big( \psi_{m_{\ell_{j}}}(x_{i}) + \aff(x_{i}) \big) &= \psilim(x_{i}), \label{eq_st4_iv} \\
\lim_{j \rightarrow \infty} \big( \psi_{m_{\ell_{j}}}(y_{0}) + \aff(y_{0}) \big) &= \psilim(y_{0}). \label{eq_st4_v}
\end{align}
From \eqref{eq_st4_i} and \eqref{eq_st4_iv} we obtain that
\begin{equation} \label{eq_st4_vi}
\forall i \in \{0, \ldots, d\} \colon \ 
\lim_{j \rightarrow \infty}  \aff_{m_{\ell_{j}}}(x_{i}) = \aff(x_{i}). 
\end{equation}
Since the points $x_{0}, \ldots, x_{d}$ are affinely independent, the convergence in \eqref{eq_st4_vi} holds for every $y \in I \cup (\Rd \setminus C)$. Then \eqref{eq_st4_v} leads to a contradiction to \eqref{eq_st4_iii}, which proves claim \eqref{eq_st4_ii}. Defining $\tilde{\psi}_{n} \coloneqq \psi_{n} + \aff_{n}$, for $n \geqslant 1$, we derive from \eqref{eq_st4_ii} the limiting assertions \eqref{prop_1b_i} and \eqref{prop_1b_ii}. Finally, by \cite[Theorem 10.8]{Ro70} the convergence in \eqref{prop_1b_i} is uniform on compact subsets of $I$ and we have already seen in Step 2 that $\psilim \equiv \psixsbm$ mod (aff), for $\mu$-a.e.\ $x \in I$.
\end{proof}

\section{Characterization of irreducibility} \label{app_sec_irr}

\begin{theorem} \label{theo_coic} Let $\mu,\nu$ in $\PP_{2}(\Rd)$ with $\mu \lc \nu$. Then the following are equivalent. 
\begin{enumerate}[label=(\arabic*)] 
\item \label{theo_coic_1} The pair $(\mu,\nu)$ is De March--Touzi irreducible in the sense of Definition \ref{def_dmti}.
\item \label{theo_coic_2} There exists $\pi \in \MT(\mu, \nu)$ such that $\pi_{x} \sim \nu$, for $\mu$-a.e.\ $x \in \Rd$.
\item \label{theo_coic_3} The pair $(\mu,\nu)$ is irreducible in the sense of Definition \ref{defi:irreducible_intro}, i.e., for all Borel sets $A, B \subseteq \Rd$ with $\mu(A), \nu(B)>0$ there is a martingale $(X_{t})_{0 \leqslant t \leqslant 1}$ with $X_{0} \sim \mu$, $X_{1} \sim \nu$ such that $\mathbb{P}(X_{0}\in A, X_{1}\in B) >0$. 
\item \label{theo_coic_4} For all Borel sets $A,B \subseteq \Rd$ with $\mu(A), \nu(B)>0$ there exists $\pi \in \MT(\mu, \nu)$ such that $\pi(A\times B)>0$. 
\item \label{theo_coic_5} For all compact sets $A \subseteq \R^d$ and open halfspaces $H$ with $\mu(A),\nu(H)>0$ there exists $\pi \in \MT(\mu, \nu)$ such that $\pi(A\times H)>0$.
\end{enumerate}
\begin{proof} For the proof of ``\ref{theo_coic_1} $\Rightarrow$ \ref{theo_coic_2}'' we can take $\pi = \pisbm$ in \ref{theo_coic_2} by Corollary \ref{thm:equivalence_kernels}. We turn to the proof of the implication ``\ref{theo_coic_2} $\Rightarrow$ \ref{theo_coic_3}'': Clearly condition \ref{theo_coic_2} implies \ref{theo_coic_4}. Thus, in order to show \ref{theo_coic_3}, it suffices to construct a continuous-time martingale $(X_{t})_{0 \leqslant t \leqslant 1}$ with $\operatorname{Law}(X_{0},X_{1}) = \pi$. This can be achieved as in the proof of \cite[Theorem 2.2]{BaBeHuKa20}. For the proof of ``\ref{theo_coic_3} $\Rightarrow$ \ref{theo_coic_4}'' we can take $\pi = \Law(X_{0},X_{1})$ in \ref{theo_coic_4}. The implication ``\ref{theo_coic_4} $\Rightarrow$ \ref{theo_coic_5}'' is trivial. It remains to show that ``\ref{theo_coic_5} $\Rightarrow$ \ref{theo_coic_1}'': We fix an open halfspace $H$ satisfying $\nu(H)>0$ and set
\[
m \coloneqq \sup_{\pi \in \MT(\mu, \nu)} \Big\{ \mu\big(\{x \in \Rd \colon  \pi_{x}(H)>0\}\big)\Big\}.
\]
Considering countable convex combinations of elements of $\MT(\mu,\nu)$ it follows that the supremum is attained by some $\bar{\pi} \in \MT(\mu,\nu)$. If  the set 
\[
\{ x \in \Rd \colon \bar{\pi}_{x}(H) = 0\}
\]
had positive $\mu$-measure, it would contain a compact set $A$ with positive $\mu$-measure in contradiction to condition \ref{theo_coic_5}. Hence $m=1$. 

Observe that there is a countable family of open halfspaces $\{H_{n}\}_{n \geqslant 1}$ with $\nu(H_{n})>0$ such that, for every open halfspace $H$ with $\nu(H)>0$, there is some $n \geqslant 1$ satisfying $H_{n} \subseteq H$. 

Next, for each $n \geqslant 1$, we pick $\pi^{(n)} \in \MT(\mu, \nu)$ such that $\pi^{(n)}_{x}(H_{n}) >0$, for $\mu$-a.e.\ $x \in \Rd$. Set $\hat{\pi} \coloneqq \sum_{n \geqslant 1} 2^{-n} \pi^{(n)} \in \MT(\mu,\nu)$. Then by Lemma \ref{ConvexSuppChar} below it follows that $\widehat{\supp}(\nu) \subseteq \widehat{\supp}(\hat{\pi}_{x})$, for $\mu$-a.e.\ $x \in \Rd$. 
\end{proof}
\end{theorem}

\begin{lemma} \label{ConvexSuppChar} Let $\rho \in \PP(\Rd)$. Then $y \in \widehat{\supp}(\rho)$ if and only if $\rho(H)>0$ for every open halfspace $H$ such that $y\in H$. 
\begin{proof} This is a simple consequence of Hahn--Banach. 
\end{proof}
\end{lemma}

\begin{remark} \label{rem_irr_nec} We note that irreducibility is not only a sufficient assumption for the existence of a Bass martingale from $\mu$ to $\nu$, but in fact also necessary. Indeed, as in the proof of Corollary \ref{thm:equivalence_kernels} one sees that the Bass martingale does the job of connecting any two sets which are charged by $\mu$ and $\nu$, implying the irreducibility of $(\mu,\nu)$ by Theorem \ref{theo_coic}.
\end{remark}

\end{appendix}

\begin{acks}[Acknowledgments]
The authors thank Ben Robinson and Gudmund Pammer for their valuable feedback and comments during the preparation of this paper.
\end{acks}

\begin{funding}
This research was funded by the Austrian Science Fund (FWF) [Grant DOIs: 10.55776/P35197 and 10.55776/P35519]. For open access purposes, the authors have applied a CC BY public copyright license to any author-accepted manuscript version arising from this submission.
\end{funding}

\bibliographystyle{imsart-number.bst} 
\bibliography{references.bib}

\end{document}